\documentclass[a4paper]{amsart}
\pdfoutput=1
\usepackage{amsmath,amssymb}
\usepackage{array}

\DeclareSymbolFont{bbold}{U}{bbold}{m}{n}
\DeclareSymbolFontAlphabet{\mathbbold}{bbold}

\newcommand{\bbDelta}{\boldsymbol{\Delta}}
\newcommand{\bbTheta}{\boldsymbol{\Theta}}
\newcommand{\bbOmega}{\boldsymbol{\Omega}}

\newcommand{\bbGamma}{\boldsymbol{\Gamma}}
\newcommand{\bbXi}{\boldsymbol{\Xi}}

\usepackage[shortlabels]{enumitem}
  \setitemize[1]{leftmargin=2em}
  \setenumerate[1]{leftmargin=2em}
  \usepackage{tikz-cd}
\usetikzlibrary{decorations.pathmorphing}  
  \tikzcdset{
    active/.style={rightsquigarrow},
    lactive/.style={leftsquigarrow},    
    inert/.style={tail}
  }

\usepackage[unicode,colorlinks=true,linktocpage=true,citecolor=blue]{hyperref}
\usepackage{graphicx,color}

\newtheorem{theorem}{Theorem}
\numberwithin{theorem}{section}
\newtheorem{thm}[theorem]{Theorem}
\newtheorem{proposition}[theorem]{Proposition}
\newtheorem{propn}[theorem]{Proposition}

\newtheorem{corollary}[theorem]{Corollary}
\newtheorem{cor}[theorem]{Corollary}
\newtheorem{lemma}[theorem]{Lemma}
\newtheorem{question}[theorem]{Question}
\theoremstyle{definition}
\newtheorem{definition}[theorem]{Definition}
\newtheorem{defn}[theorem]{Definition}
\newtheorem{example}[theorem]{Example}
\newtheorem{examples}[theorem]{Examples}
\newtheorem{notation}[theorem]{Notation}
\newtheorem{remark}[theorem]{Remark}
\newtheorem{warning}[theorem]{Warning}
\newtheorem{variant}[theorem]{Variant}

\providecommand{\op}{\mathrm{op}}
\providecommand{\xel}{\mathrm{el}}

\providecommand{\xext}{\mathrm{ext}}

\providecommand{\xact}{\mathrm{act}}
\providecommand{\xint}{\mathrm{int}}

\newcommand{\xFun}{\operatorname{Fun}}

\newcommand{\act}{\name{act}}
\newcommand{\Act}{\name{Act}}
\newcommand{\el}{\name{el}}

\DeclareMathOperator{\colimP}{colim}
\newcommand{\colim}{\mathop{\colimP}}

\newcommand{\Map}{\operatorname{Map}}

\newcommand{\xAlg}{\operatorname{Alg}}

\DeclareMathOperator{\xSeg}{Seg}

\newcommand{\xA}{\mathcal{A}}

\newcommand{\xE}{\mathcal{E}}

\newcommand{\xI}{\mathcal{I}}

\newcommand{\xxO}{\mathcal{O}}
\newcommand{\xxP}{\mathcal{P}}

\newcommand{\xcc}{\mathcal{C}}

\newcommand{\xS}{\mathcal{S}}

\newcommand{\xW}{\mathcal{W}}

\newcommand{\xU}{\mathcal{U}}

\newcommand{\id}{\operatorname{id}}

\newcommand{\xF}{\mathbb{F}}

\newcommand*\cocolon{%
\nobreak
\mskip6mu plus1mu
\mathpunct{}%
\nonscript
\mkern-\thinmuskip
{:}%
\mskip2mu
\relax
}

\newcommand{\icat}{$\infty$-category}
\newcommand{\icatl}{$\infty$-categorical}
\newcommand{\igpd}{$\infty$-groupoid}
\newcommand{\igpds}{$\infty$-groupoids}
\newcommand{\icats}{$\infty$-categories}
\newcommand{\iopd}{$\infty$-operad}
\newcommand{\iopds}{$\infty$-operads}
\newcommand{\isoto}{\xrightarrow{\sim}}

\newcommand{\xto}[1]{\xrightarrow{#1}}

\newcommand{\from}{\leftarrow}

\newcommand{\csquare}[8]{ %
\[ %
\begin{tikzpicture} %
\matrix (m) [matrix of math nodes,row sep=3em,column sep=2.5em,text height=1.5ex,text depth=0.25ex] %
{ #1 \pgfmatrixnextcell #2 \\ %
  #3 \pgfmatrixnextcell #4 \\ }; %
\path[->,font=\footnotesize] %
(m-1-1) edge node[auto] {$#5$} (m-1-2)%
(m-1-1) edge node[left] {$#6$} (m-2-1)%
(m-1-2) edge node[auto] {$#7$} (m-2-2)%
(m-2-1) edge node[below] {$#8$} (m-2-2);%
\end{tikzpicture}%
\]%
}

\newcommand{\nolabelcsquare}[4]{\csquare{#1}{#2}{#3}{#4}{}{}{}{}}
\makeatletter
\def\@tocline#1#2#3#4#5#6#7{\relax
  \ifnum #1>\c@tocdepth 
  \else
    \par \addpenalty\@secpenalty\addvspace{#2}%
    \begingroup \hyphenpenalty\@M
    \@ifempty{#4}{%
      \@tempdima\csname r@tocindent\number#1\endcsname\relax
    }{%
      \@tempdima#4\relax
    }%
    \parindent\z@ \leftskip#3\relax \advance\leftskip\@tempdima\relax
    \rightskip\@pnumwidth plus4em \parfillskip-\@pnumwidth
    #5\leavevmode\hskip-\@tempdima
      \ifcase #1
       \or \hskip -1em \or \hskip 1em \or \hskip 3em \else \hskip 5em \fi%
      #6\nobreak\relax
    \hfill\hbox to\@pnumwidth{\@tocpagenum{#7}}
      \par
    \nobreak
    \endgroup
  \fi}
\makeatother

\newcommand{\name}[1]{\ensuremath{\text{\textup{#1}}}}

\newcommand{\simp}{\bbDelta}
\newcommand{\Dop}{\simp^{\op}}
\newcommand{\DF}{\simp_{\mathbb{F}}}

\newcommand{\bbO}{\bbOmega}

\newcommand{\Seg}{\name{Seg}}

\newcommand{\Fun}{\name{Fun}}

\newcommand{\blank}{\text{\textendash}}

\newcommand{\Cat}{\name{Cat}}
\newcommand{\CatI}{\Cat_{\infty}}
\newcommand{\LCatI}{\widehat{\Cat}_{\infty}}
\newcommand{\IFF}{if and only if}

\newcommand{\Alg}{\name{Alg}}

\newcommand{\ie}{i.e.\@}

\newcommand{\AlgPatt}{\name{AlgPatt}}
\newcommand{\FlAlgPatt}{\name{FlAlgPatt}}
\newcommand{\AlgPattSS}{\AlgPatt^{\Seg}_{\name{sat}}}
\newcommand{\AlgPattSE}{\AlgPatt^{\Seg}_{\name{ext}}}
\newcommand{\AlgPattSES}{\AlgPatt^{\Seg}_{\name{slim,ext}}}
\newcommand{\PolyMnd}{\name{PolyMnd}}
\newcommand{\cPolyMnd}{\name{cPolyMnd}}

\newcommand{\angled}[1]{\langle #1 \rangle}

\newcommand{\actto}{\rightsquigarrow}
\newcommand{\intto}{\rightarrowtail}

\newcommand{\LO}{\Lambda_{\mathcal{O}}}
\newcommand{\LOi}{\Lambda_{\mathcal{O}}^{\xint}}
\newcommand{\LOpi}{\Lambda_{\mathcal{O}}^{(\xint)}}
\newcommand{\yO}{y_{\mathcal{O}}}

\newcommand{\yOpi}{y_{\mathcal{O}}^{(\xint)}}
\newcommand{\olO}{\overline{\mathcal{O}}}
\usepackage{eucal}

\usepackage[utf8]{inputenc}
\title{Homotopy-Coherent Algebra via Segal Conditions}
\author{Hongyi Chu}
\address{Max Planck Institute for Mathematics, Bonn, Germany}

\author{Rune Haugseng}
\address{Norwegian University of Science and Technology (NTNU), Trondheim, Norway}
\date{\today}
\begin{document}
\begin{abstract}
  Many homotopy-coherent algebraic structures can be described by
  Segal-type limit conditions determined by an ``algebraic pattern'',
  by which we mean an $\infty$-category equipped with a factorization
  system and a collection of ``elementary'' objects. Examples of
  structures that occur as such ``Segal $\mathcal{O}$-spaces'' for an
  algebraic pattern $\mathcal{O}$ include $\infty$-categories,
  $(\infty,n)$-categories, $\infty$-operads (including symmetric,
  non-symmetric, cyclic, and modular ones), $\infty$-properads, and
  algebras for a (symmetric) $\infty$-operad in spaces.

  In the first part of this paper we set up a general framework for
  algebraic patterns and their associated Segal objects, including
  conditions under which the latter are preserved by left and right
  Kan extensions. In particular, we obtain necessary and sufficent
  conditions on a pattern $\mathcal{O}$ for free Segal
  $\mathcal{O}$-spaces to be described by an explicit colimit formula,
  in which case we say that $\mathcal{O}$ is ``extendable''.

  In the second part of the paper we explore the relationship between
  extendable algebraic patterns and polynomial monads, by which we
  mean cartesian monads on presheaf $\infty$-categories that are
  accessible and preserve weakly contractible limits. We first show
  that the free Segal $\mathcal{O}$-space monad for an extendable
  pattern $\mathcal{O}$ is always polynomial. Next, we prove an
  $\infty$-categorical version of Weber's Nerve Theorem for polynomial
  monads, and use this to define a canonical extendable pattern from
  any polynomial monad, whose Segal spaces are equivalent to the
  algebras of the monad. These constructions yield functors between
  polynomial monads and extendable algebraic patterns, and we show
  that these exhibit full subcategories of ``saturated'' algebraic
  patterns and ``complete'' polynomial monads as localizations, and moreover
  restrict to an equivalence between the $\infty$-categories of
  saturated patterns and complete polynomial monads.
\end{abstract}

\maketitle
\tableofcontents

\section{Introduction}
Homotopy-coherent algebraic structures, where identities between
operations are replaced by an infinite hierarchy of compatible
coherence equivalences, have played an important role in algebraic
topology since the 1960s\footnote{More general frameworks for
  homotopy-coherent algebra, such as operads, arose out of work on
  infinite loop spaces by Boardman--Vogt~\cite{BoardmanVogt} and
  May~\cite{May} in the early 1970s.}, when they were first introduced
in the special case of $A_{\infty}$-spaces by
Stasheff~\cite{StasheffHAss}, and have since found a variety of
applications in many fields of mathematics. From a modern perspective,
homotopy-coherent algebraic structures can be considered as the
natural algebraic structures in the setting of \icats{} (which are
themselves the homotopy-coherent analogues of categories).

It turns out that many interesting homotopy-coherent algebraic
structures can be described by ``Segal conditions'', \ie{} they can be
described as functors satisfying a specific type of limit condition.
The canonical (and original) example is Segal's \cite{SegalCatCohlgy}
description of homotopy-coherently commutative monoids in spaces (or
\emph{$E_{\infty}$-spaces}) as ``special $\Gamma$-spaces''. In
\icatl{} language, these are functors
$F \colon \xF_{*} \to \mathcal{S}$, where $\xF_{*}$ is a skeleton of
the category of pointed finite sets, with objects
$\angled{n} := (\{0,1,\ldots,n\}, 0)$, and $\mathcal{S}$ is the
\icat{} of spaces (or $\infty$-groupoids), which are required to satisfy the following
condition:
\begin{quotation}
  For all $n$, the map
  \[ F(\angled{n}) \to \prod_{i = 1}^{n}F(\angled{1}), \]
  induced by the morphisms $\rho_{i} \colon \angled{n} \to \angled{1}$
  given by
  \[
    \rho_{i}(j) =
    \begin{cases}
      0, & j \neq i,\\
      1, & j = i,
    \end{cases}
  \]
  is an equivalence.
\end{quotation}
Other key examples of structures described by Segal conditions include:
\begin{itemize}
\item associative (or $A_{\infty}$- or $E_{1}$-)monoids, using the
  simplex category $\Dop$ (in unpublished work of Segal),
\item \icats{}, again using $\Dop$, in the form of Rezk's Segal spaces
  \cite{Rezk},
\item $(\infty,n)$-categories, using Joyal's categories
  $\bbTheta_{n}^{\op}$, also in work of Rezk \cite{RezkThetan},
\item $\infty$-operads, using the dendroidal category $\bbOmega^{\op}$
  of Moerdijk--Weiss~\cite{MoerdijkWeiss}, in work of
  Cisinski and Moerdijk~\cite{CisinskiMoerdijk2},
\item algebras for an \iopd{} $\mathcal{O}$ (in the sense of
  \cite{ha}) in $\mathcal{S}$, using the ``category of operators''
  $\mathcal{O}$ itself.
\end{itemize}
Given these and other examples (many of which we will discuss below in
\S\ref{sec:ex}), we might wonder why so many different algebraic
structures can be described by Segal conditions. Our main results in
this paper provide an explanation of this situation, by answering the
following question:
\begin{question}
  Which homotopy-coherent algebraic structures can be described (in a
  reasonable way) by Segal conditions, and how canonical is this
  description?
\end{question}

Before we describe our answer, we need to formulate a more precise
version of this question, by defining the terms that appear. First of
all, we will consider algebraic structures on (families of) spaces,
which we take to mean algebras for monads on functor \icats{}
$\Fun(\mathcal{I}, \mathcal{S})$ (where $\mathcal{I}$ is any small
\icat{}). Next, let us specify what precisely we mean by ``Segal
conditions''. Returning to the example of special $\Gamma$-spaces, the
category $\xF_{*}$ has the following features that we wish to
abstract:
\begin{itemize}
\item A morphism $\phi \colon \angled{n} \to \angled{m}$ is called
  \emph{inert} if $|\phi^{-1}(j)| = 1$ for $j \neq 0$, and
  \emph{active} if $\phi^{-1}(0) = \{0\}$. The inert and active
  morphisms form a factorization system on $\xF_{*}$: every morphism
  factors as an inert morphism followed by an active morphism, and
  this decomposition is unique up to isomorphism.
\item The morphisms $\rho_{i}$ are precisely the inert morphisms
  $\angled{n} \to \angled{1}$.
\item If $\xF_{*}^{\xint}$ denotes the subcategory of $\xF_{*}$ with
  only inert morphisms, then the special $\Gamma$-spaces are precisely
  the functors $F \colon \xF_{*} \to \mathcal{S}$ such that the
  restriction $F|_{\xF_{*}^{\xint}}$ is a right Kan extension of
  $F|_{\{\angled{1}\}}$.
\end{itemize}
These features recur in our other examples, which suggests that the
input data for a class of ``Segal conditions'' should consist of an
\icat{} $\mathcal{O}$ equipped with a factorization system (whereby
every morphism factors as an ``active'' morphism followed by an
``inert'' morphism) and a class of ``elementary'' objects (or
generators). From this data, which we will refer to as an
\emph{algebraic pattern}\footnote{This terminology is inspired by
  Lurie's \emph{categorical patterns} \cite[\S B]{ha}, the key
  examples of which all arise from algebraic patterns in our sense,
  and should not be confused with the notion of ``pattern'' considered
  by Getzler~\cite{GetzlerOpdRev}.}, we obtain the relevant Segal-type
limit condition on a functor $F \colon \mathcal{O} \to \mathcal{S}$ by
imposing the requirement that for every $O \in \mathcal{O}$ the object
$F(O)$ is the limit over all inert morphisms to elementary objects,
\[ F(O) \isoto \lim_{E \in \mathcal{O}^{\el}_{O/}} F(E);\] we say that
such a functor $F$ is a \emph{Segal
  $\mathcal{O}$-space}.\footnote{Here we write $\mathcal{O}^{\xint}$
  for the subcategory of $\mathcal{O}$ containing only the inert
  morphisms, $\mathcal{O}^{\el}$ for the full subcategory of
  $\mathcal{O}^{\xint}$ spanned by the elementary objects, and define
  $\mathcal{O}^{\el}_{O/}:= \mathcal{O}^{\el}
  \times_{\mathcal{O}^{\xint}} \mathcal{O}^{\xint}_{O/}$.}  If
$\mathcal{O}$ is any algebraic pattern, and
$\Seg_{\mathcal{O}}(\mathcal{S})$ denotes the full subcategory of
$\Fun(\mathcal{O}, \mathcal{S})$ on the Segal $\mathcal{O}$-spaces,
then the restriction functor
\[\Seg_{\mathcal{O}}(\mathcal{S}) \to \Fun(\mathcal{O}^{\el},
\mathcal{S})\] has a left adjoint. This adjunction is always monadic,
and we write $T_{\mathcal{O}}$ for the corresponding monad on
$\Fun(\mathcal{O}^{\el}, \mathcal{S})$. The monad $T_{\mathcal{O}}$ is
then ``described by'' the algebraic pattern $\mathcal{O}$. In general,
however, it is not possible to describe this monad explicitly, because the
left adjoint involves an abstract localization. We only want to
consider a pattern to be ``reasonable'' if this localization is
unnecessary, in which case $T_{\mathcal{O}}$ is given
by a concrete formula, namely as
\[ T_{\mathcal{O}}F(E) \simeq \colim_{X \in \Act_{\mathcal{O}}(E)}
  \lim_{E' \in \mathcal{O}^{\el}_{X/}} F(E'),\]
where $\Act_{\mathcal{O}}(E)$ is the space of active morphisms to $E$
in $\mathcal{O}$. We call such patterns $\mathcal{O}$
\emph{extendable}, and give explicit necessary and sufficient
conditions for a pattern to be extendable in
Proposition~\ref{prop:extendable}.

We can now state the precise version of the previous question that we
will address:
\begin{question}
  Which monads on presheaf \icats{} can be described as the free Segal
  $\mathcal{O}$-space monad for an extendable algebraic pattern
  $\mathcal{O}$, and how canonical is this description?
\end{question}

We will characterize these monads as a certain class of
\emph{polynomial}\footnote{The analogous monads on
    ordinary categories are sometimes called \emph{strongly cartesian}
  monads.} monads, by which we mean the monads on presheaf
\icats{} that are \emph{cartesian}\footnote{The cartesian monads are
  those whose multiplication and unit transformations are
  \emph{cartesian} natural transformations, which in turn means that
  their naturality squares are all cartesian, \ie{} are pullback squares.} and whose underlying
endofunctors are accessible and preserve weakly contractible
limits. Our first main result provides functors in both directions
between \icats{} of extendable patterns and of polynomial monads:
\begin{thm}\
  \begin{enumerate}[(i)]
  \item   If $\mathcal{O}$ is an extendable algebraic pattern then the free
  Segal $\mathcal{O}$-space monad $T_{\mathcal{O}}$ is
  polynomial. This determines a functor $\mathfrak{M}$ from extendable
  patterns to polynomial monads.
\item If $T$ is a polynomial monad on $\Fun(\mathcal{I}, \mathcal{S})$
  then there exists a canonical extendable algebraic pattern
  $\mathcal{W}(T)$ such that $\Seg_{\mathcal{W}(T)}(\mathcal{S})$ is
  equivalent to the \icat{} of $T$-algebras. This determines a functor
  $\mathfrak{P}$ from polynomial monads to extendable patterns.
  \end{enumerate}
\end{thm}
We prove part (i) in \S\ref{sec:poly} and part (ii) in
\S\ref{sec:poly1}. Part (ii) depends on an \icatl{} version of Weber's
\emph{nerve theorem} \cite{WeberFamilial}, which we prove in
\S\ref{sec:nerve} and use to construct a factorization system on the
Kleisli \icat{} of a polynomial monad in \S\ref{sec:Kleislifact}.

Our second main result characterizes the images of these functors:
\begin{thm}\ 
  \begin{enumerate}[(i)]
  \item Restricting to \emph{slim}\footnote{This is a mild technical
      hypothesis; it is satisfied in almost all examples, and the
      patterns $\mathcal{W}(T)$ are always slim. Moreover, any
      extendable pattern can be replaced by a full subcategory that is
      slim and determines the same monad.} extendable patterns, there
    is a natural transformation
    $\sigma \colon \id \to \mathfrak{P}\mathfrak{M}$, and the component
    $\sigma_{\mathcal{O}}$ is an equivalence
    \IFF{} the pattern $\mathcal{O}$ is \emph{saturated}, meaning
    that it is a slim extendable pattern such that the functors
    \[\Map_{\mathcal{O}}(O,\blank) \colon \mathcal{O} \to
      \mathcal{S}\] are Segal $\mathcal{O}$-spaces for
    $O \in \mathcal{O}$. The pattern $\mathcal{W}(T)$ for a polynomial
    monad $T$ is always saturated, and the transformation $\sigma$
    exhibits the full subcategory of saturated patterns as a
    localization of the \icat{} of slim extendable patterns.
  \item There is a natural transformation
    $\tau \colon \id \to \mathfrak{M}\mathfrak{P}$, and $\tau_{T}$ is
    an equivalence for a polynomial monad $T$ on
    $\Fun(\mathcal{I}, \mathcal{S})$ \IFF{} $T$ is \emph{complete},
    meaning that the essentially surjective functor
    $\mathcal{I} \to \mathcal{W}(T)^{\el}$ is an equivalence. The
    monad $T_{\mathcal{O}}$ for an extendable pattern $\mathcal{O}$ is
    always complete, and the transformation $\tau$ exhibits the full
    subcategory of complete polynomial monads as a localization of the
    \icat{} of polynomial monads.
  \item The functors $\mathfrak{P}$ and $\mathfrak{M}$ restrict to an
    equivalence between the \icats{} of saturated patterns and
    complete polynomial monads.
  \end{enumerate}
\end{thm}
We will prove part (i) in \S\ref{sec:poly2} and parts (ii) and (iii)
in \S\ref{sec:poly3}.

The answer to our question above is thus that the monads of the form
$T_{\mathcal{O}}$ for an extendable pattern $\mathcal{O}$ are
precisely the \emph{complete} polynomial monads, and there is a
\emph{unique} extendable pattern describing this monad that is
\emph{saturated}, namely the canonical pattern
$\mathcal{W}(T_{\mathcal{O}})$. For example, returning to our initial
example of commutative monoids described by an algebraic pattern
structure on $\xF_{*}$, this pattern is extendable, with free
commutative monoids describe by the expected formula
\[ X \mapsto \coprod_{n=0}^{\infty} X^{\times n}_{h \Sigma_{n}},\] but
it is \emph{not} saturated. The corresponding saturated pattern is
instead the \icat{} of free commutative monoids on finite sets (\ie{}
the \emph{Lawvere theory} for commutative monoids), which by work of
Cranch~\cite{Cranch} can be identified with the $(2,1)$-category
$\name{Span}(\xF)$ of finite sets with spans (or correspondences) as
morphisms; see \ref{ex:commspan} for more details.

\subsection{Overview}
In the first part of the paper we set up a general categorical
framework for algebraic patterns and Segal objects. In
\S\ref{sec:pattern} we introduce these objects more formally and prove
some of their basic properties, before we look at examples of
algebraic patterns and their Segal objects in \S\ref{sec:ex}. We then
introduce morphisms of algebraic patterns in \S\ref{sec:pattmor} and
construct an \icat{} of algebraic patterns in \S\ref{sec:pattcat},
where we also prove that this has limits and filtered colimits. Next,
we provide conditions under which Segal objects are preserved by right
and left Kan extensions in \S\ref{sec:RKE} and \S\ref{sec:LKE},
respectively.

In \S\ref{sec:free} we apply our work on left Kan
extensions to analyze free Segal objects; in particular, we obtain
necessary and sufficient conditions for a pattern $\mathcal{O}$ to be
extendable, meaning that free Segal $\mathcal{O}$-spaces are described
by a colimit formula. In \S\ref{sec:WSF} we study (weak) Segal
fibrations, which generalize Lurie's definitions of symmetric monoidal
\icats{} and symmetric \iopds{}. We show that any weak Segal fibration
over an extendable base is again extendable, and moreover left Kan
extension along any morphism of weak Segal fibrations preserves
Segal objects; this recovers, for example, the formula of \cite{ha} for
operadic left Kan extensions of $\infty$-operad algebras in
cartesian monoidal \icats{}.

In \S\ref{sec:poly} we introduce polynomial monads, and prove that the
free Segal $\mathcal{O}$-space monad for any extendable pattern is
polynomial. We then prove an \icatl{} version of Weber's Nerve Theorem for
presheaf \icats{} in \S\ref{sec:nerve}, and apply this to define a
factorization system on the Kleisli \icat{} of a polynomial monad in
\S\ref{sec:Kleislifact}. This gives a canonical algebraic pattern for
every polynomial monad, which we study in \S\ref{sec:poly1}. Next, we
study the relationship between an extendable
pattern and the canonical pattern of its free Segal space monad;
under a mild hypothesis there is a functor between these, and we show
that this is an equivalence precisely when the pattern is
saturated. Finally, in \S\ref{sec:poly3} we study complete polynomial
monads, and prove that there is an equivalence between these and
saturated patterns.

\subsection{Related Work}
There is an extensive literature on using (finite) limit conditions to
describe algebraic structures in category theory, going back at least
to Lawvere's thesis~\cite{Lawvere}, where he introduced algebraic
theories. Our work is in particular closely related to the ``nerve
theorem'', one version of which \emph{almost} says that a strongly
cartesian monad on a presheaf category is described by Segal
conditions; this version was first proved in unpublished work of
Leinster (though his proof did not use the factorization system), and
later extended by Weber~\cite{WeberFamilial} to a description
of certain \emph{weakly} cartesian monads.\footnote{Weakly cartesian
  monads are of interest in the case of ordinary categories, as many
  ``algebraic'' monads that involve symmetries, such as the free
  commutative monoid monad, are not cartesian. This issue disappears
  if we replace sets by groupoids, and so weakly cartesian monads
  are not relevant in our \icatl{} setting.} We were particularly
inspired by the simpler proof given by Berger, Melli\`es, and
Weber~\cite{BergerMelliesWeber}. Their work has more recently been
extended by Bourke and Garner~\cite{BourkeGarner}, who study general
classes of monads that can be described by some notion of ``theories
with arities'', including in the enriched context.

\subsection{Acknowledgments}
H.C. thanks the Labex CEMPI (ANR-11-LABX-0007-01) and Max Planck
Institute for Mathematics for their support during the process of
writing this article. The first version of this paper was written
while R.H.\ was employed by the IBS Center for Geometry and Physics in
a position funded through the
grant IBS-R003-D1 of the Institute for Basic Science, Republic
of Korea. We thank Philip Hackney for helpful comments on the first
version of the paper, and Roman Kositsyn for pointing out that the
conclusion of Lemma~\ref{lem:stronglyext} follows from extendability.

\section{Algebraic Patterns and Segal Objects}\label{sec:pattern}
In this section we introduce the basic structures we will study in
this paper, namely algebraic patterns and their Segal objects.

\begin{definition}
  An \emph{algebraic pattern} $\mathfrak{O}$ is an \icat{}
  $\mathcal{O}$ equipped with:
\begin{itemize}
\item a factorization system
$(\mathcal{O}^{\xint}, \mathcal{O}^{\act})$, the morphisms in
which we refer to as the \emph{inert} and \emph{active} morphisms
in $\mathfrak{O}$,
\item a full subcategory
$\mathcal{O}^{\el} \subseteq \mathcal{O}^{\xint}$ whose objects we
call the \emph{elementary} objects of $\mathfrak{O}$.
\end{itemize}
Unless stated otherwise, we will assume by default that algebraic
patterns are essentially small.
\end{definition}

\begin{remark}
  Here a \emph{factorization system} on an \icat{} $\mathcal{C}$ means
  a pair of subcategories $(\mathcal{C}^{L}, \mathcal{C}^{R})$ such
  that both contain all objects of $\mathcal{C}$, and for every
  morphism $f \colon X \to X'$ in $\mathcal{C}$, the space of
  factorizations
\[ \left\{
\begin{tikzcd}
{} & Y \arrow{dr}{r} \\
X \arrow{ur}{l} \arrow{rr}{f} & & X'
\end{tikzcd}
: l \in \mathcal{C}^{L}, r \in \mathcal{C}^{R}
\right\}
\]
is contractible.
\end{remark}

\begin{remark}
  We will often abuse notation and conflate an algebraic pattern with
  its underlying \icat{} $\mathcal{O}$, \ie{} we will simply say that
  $\mathcal{O}$ is an algebraic pattern.
\end{remark}

\begin{notation}
  If $\mathcal{O}$ is an algebraic pattern, we will often indicate an inert
  map between objects $O,O'$ of $\mathcal{O}$ as $O \intto O'$ and an
  active map as $O \actto O'$. These symbols are not meant to suggest
  any intuition about the nature of inert and active maps.
\end{notation}

\begin{notation}
  If $\mathcal{O}$ is an algebraic pattern and $X$ is an object of
  $\mathcal{O}$, then we write $\mathcal{O}^{\el}_{X/}$ for the fibre
  product of \icats{}
  $\mathcal{O}^{\el} \times_{\mathcal{O}^{\xint}}
  \mathcal{O}^{\xint}_{X/}$. Thus the objects of
  $\mathcal{O}^{\el}_{X/}$ are inert morphisms $X \intto E$ where $E$ is
  elementary, and the morphisms are commutative triangles
\[
\begin{tikzcd}
{} & X \arrow[inert]{dr} \arrow[inert]{dl} \\
E \arrow[inert]{rr} & & E'
\end{tikzcd}
\]
where all morphisms are inert, and $E$ and $E'$ are elementary.
\end{notation}

\begin{defn}
Let $\mathcal{O}$ be an algebraic pattern. We say an \icat{} is
\emph{$\mathcal{O}$-complete} if it has limits of shape
$\mathcal{O}^{\el}_{X/}$ for all $X \in\mathcal{O}$.
\end{defn}

\begin{defn}
Let $\mathcal{O}$ be an algebraic pattern and $\mathcal{C}$ an
$\mathcal{O}$-complete \icat{}. A \emph{Segal $\mathcal{O}$-object}
in $\mathcal{C}$ is a functor $F \colon \mathcal{O} \to \mathcal{C}$
such that for every $X \in \mathcal{O}$ the canonical map
\[ F(X) \to \lim_{E \in \mathcal{O}^{\el}_{X/}} F(E)\]
is an equivalence. We write $\Seg_{\mathcal{O}}(\mathcal{C})$ for
the full subcategory of $\Fun(\mathcal{O}, \mathcal{C})$ spanned by
the Segal $\mathcal{O}$-objects. 
\end{defn}

\begin{notation}
We will often refer to Segal $\mathcal{O}$-objects in the \icat{}
$\mathcal{S}$ of spaces as \emph{Segal $\mathcal{O}$-spaces}, and to
Segal $\mathcal{O}$-objects in the \icat{} $\CatI$ of \icats{} as
\emph{Segal $\mathcal{O}$-\icats{}}.
\end{notation}

\begin{lemma}
Let $\mathcal{C}$ be an $\mathcal{O}$-complete \icat{}. Then $F \colon \mathcal{O} \to \mathcal{C}$ is a Segal
$\mathcal{O}$-object \IFF{} the restriction
$F|_{\mathcal{O}^{\xint}}$ is a right Kan extension of
$F|_{\mathcal{O}^{\el}}$.
\end{lemma}
\begin{proof}
Since $\mathcal{C}$ is $\mathcal{O}$-complete,   $F|_{\mathcal{O}^{\xint}}$ is a right Kan extension of
$F|_{\mathcal{O}^{\el}}$ \IFF{} for all $X \in \mathcal{O}^{\xint}$,
the natural map
\[ F(X) \to \lim_{E \in \mathcal{O}^{\el}_{X/}} F(E) \]
is an equivalence.
\end{proof}

\begin{defn}
Let $\mathcal{O}$ be an algebraic pattern. For $O \in \mathcal{O}$
we write $y(O)_{\Seg}$ for the colimit
$\colim_{E \in (\mathcal{O}^{\el}_{O/})^{\op}} y(E)$ in
$\Fun(\mathcal{O}, \mathcal{S})$, where $y$ denotes the Yoneda
embedding $\mathcal{O}^{\op} \to \Fun(\mathcal{O}, \mathcal{S})$. If
$\mathcal{C}$ is a cocomplete \icat{}, and thus is tensored over
$\mathcal{S}$, then we can consider $C \otimes y(O)$ and
$C \otimes y(O)_{\Seg}$ in $\Fun(\mathcal{O}, \mathcal{C})$ for
$C \in \mathcal{C}$.
\end{defn}

\begin{lemma}\label{lem:locality}
Let $\mathcal{O}$ be an algebraic pattern and $\mathcal{C}$ a
cocomplete \icat{}.
\begin{enumerate}[(i)] 
\item $F \in \Fun(\mathcal{O}, \mathcal{C})$ is a Segal
$\mathcal{O}$-object \IFF{} $F$ is local with respect to the
canonical maps $C \otimes y(O)_{\Seg} \to C \otimes y(O)$ for all
$O \in \mathcal{O}$.
\item If $\mathcal{C}$ is $\kappa$-presentable, then $F$ is a Segal
$\mathcal{O}$-object \IFF{} $F$ is local with respect to these
maps where $C$ is $\kappa$-compact.
\item If $\mathcal{C}$ is presentable, then the full subcategory
$\Seg_{\mathcal{O}}(\mathcal{C})$ is an accessible localization of
$\Fun(\mathcal{O}, \mathcal{C})$.
\item If $\mathcal{C}$ is presentable, then so is the \icat{} $\Seg_{\mathcal{O}}(\mathcal{C})$.
\end{enumerate}
\end{lemma}
\begin{proof}
The object $F$ is local with respect to $C \otimes y(O)_{\Seg} \to C
\otimes y(O)$ precisely when the morphism of spaces
\[ \Map_{\Fun(\mathcal{O}, \mathcal{C})}(C \otimes y(O), F)
\to \Map_{\Fun(\mathcal{O}, \mathcal{C})}(C \otimes y(O)_{\Seg},
F)\]
is an equivalence. Here we have equivalences
\[ \Map_{\Fun(\mathcal{O}, \mathcal{C})}(C \otimes y(O), F) \simeq
\Map_{\Fun(\mathcal{O}, \mathcal{S})}(y(O), \Map_{\mathcal{C}}(C,
F)) \simeq \Map_{\mathcal{C}}(C, F(O)),\]
using the Yoneda Lemma, and similarly
\[
  \begin{split}
  \Map_{\Fun(\mathcal{O}, \mathcal{C})}(C \otimes y(O)_{\Seg}, F)
& \simeq \Map_{\Fun(\mathcal{O}, \mathcal{S})}(y(O)_{\Seg}, \Map_{\mathcal{C}}(C,
F))
\\ & \simeq \lim_{E \in
\mathcal{O}^{\el}_{O/}} \Map_{\Fun(\mathcal{O},
\mathcal{S})}(y(E), \Map_{\mathcal{C}}(C, F)) \\ & \simeq
\Map_{\mathcal{C}}(C, \lim_{E \in \mathcal{O}^{\el}_{O/}} F(E)).    
  \end{split}
\]
Thus $F$ is local with respect to this morphism for all $C$ and $O$
\IFF{} $F(O) \isoto \lim_{E \in \mathcal{O}^{\el}_{O/}} F(E)$ for
all $O$, \ie{} $F$ is a Segal object. This proves (i). If
$\mathcal{C}$ is $\kappa$-presentable, then to conclude that the
Segal map $F(O) \to \lim_{E \in \mathcal{O}^{\el}_{O/}}F(E)$ is an
equivalence it suffices to consider $C$ in $\mathcal{C}^{\kappa}$,
which proves (ii).

It follows that if $\mathcal{C}$ is presentable, then
$\Seg_{\mathcal{O}}(\mathcal{C})$ is the full subcategory of objects
in $\Fun(\mathcal{O}, \mathcal{C})$ that are local with respect to a
\emph{set} of morphisms. Parts (iii) and (iv) then follow from
\cite[Proposition 5.5.4.15]{ht}.   
\end{proof}

\section{Examples of Algebraic Patterns}\label{sec:ex}
In this section we will briefly describe some examples of algebraic
patterns and their associated Segal objects.

\begin{example}\label{ex:xF*flat}
  We write $\xF_{*}^{\flat}$ for the algebraic pattern structure on $\xF_{*}$
  given by the inert--active factorization system we discussed above in the
  introduction, with $\xF_{*}^{\flat,\el}$ containing the single
  object $\angled{1}$. Then a Segal $\xF_{*}^{\flat}$-space is
  precisely a commutative monoid, or equivalently a special $\Gamma$-space in the sense of
  \cite{SegalCatCohlgy}.
\end{example}

\begin{example}
  We can also consider another pattern structure on
  $\xF_{*}$: We define $\xF_{*}^{\natural}$ by the same factorization
  system, but now $\xF_{*}^{\natural,\el}$ contains the two objects
  $\angled{0}$ and $\angled{1}$, with the unique inert morphism
  $\angled{1} \to \angled{0}$. Segal $\xF_{*}^{\natural}$-objects are
  functors $F \colon \xF_{*}^{\natural} \to \mathcal{C}$ such that
  \[ F(\angled{n}) \simeq F(\angled{1})^{\times_{F(\angled{0})} n},\]
  where the right-hand side denotes an iterated fibre product over
  $F(\angled{0})$; 
  this is equivalently a commutative monoid in the slice
  $\mathcal{C}_{/F(\angled{0})}$.
\end{example}

\begin{example}\label{ex:simp}
  We write $\simp$ for the simplex category, \ie{} the category of
  non-empty finite ordered sets $[n] := \{0,\ldots,n\}$ and
  order-preserving maps between them. A morphism
  $f \colon [n] \to [m]$ is \emph{inert} if it is the inclusion of a
  sub-interval, \ie{} $f(i) = f(0)+i$ for all $i$, and \emph{active}
  if $f$ preserves the end-points, \ie{} $f(0) = 0$ and
  $f(n)=m$. Every morphism in $\simp$ factors uniquely as an active
  morphism followed by an inert one, so this determines an
  inert--active factorization system on $\simp^{\op}$. Using this
  factorization system we can define two interesting algebraic pattern
  structures on $\simp^{\op}$:
  \begin{itemize}
  \item $\simp^{\op,\natural}$ denotes the pattern where
    $\simp^{\op,\natural,\el}$ contains the two objects $[0]$ and
    $[1]$, and the two inert morphisms $[1] \rightrightarrows [0]$,
  \item $\simp^{\op,\flat}$ denotes the pattern where
  $\simp^{\op,\flat,\el}:= \{[1]\}$.
  \end{itemize}
  A Segal $\simp^{\op,\natural}$-object is a functor $F \colon \Dop
  \to \mathcal{C}$ such that
  \[ F([n]) \isoto F([1]) \times_{F([0])} \cdots \times_{F([0])}
    F([1]).\]
  In particular, a Segal $\simp^{\op,\natural}$-space is precisely a
  \emph{Segal space} in the sense
  of Rezk~\cite{Rezk}, which describes the algebraic structure of an
  \icat{}. On the other hand, a Segal $\simp^{\op,\flat}$-object $F$
  satisfies
  \[ F([n]) \simeq F([1])^{\times n},\] and describes an associative
  monoid.
\end{example}

\begin{example}
  For any integer $n$ the product $\simp^{n,\op} := (\Dop)^{\times n}$
  has a coordinate-wise factorization system (\ie{} a morphism is
  active or inert precisely when all of its components are). Using
  this we can define two algebraic pattern structures
  $\simp^{n,\op,\natural}$ and
  $\simp^{n,\op,\flat}$, where
  \[ \simp^{n,\op,\natural,\el} := (\simp^{\op,\natural,\el})^{n}\]
  consists of all objects $([i_{1}],\ldots,[i_{n}])$ with $i_{s}=0$ or
  $1$ for all $s$, while 
  \[ \simp^{n,\op,\flat,\el} := \{([1],\ldots,[1])\}.\] These are both
  special cases of products of algebraic patterns
  (Corollary~\ref{cor:PattLim}). Segal
  $\simp^{n,\op,\natural}$-spaces are \emph{$n$-uple Segal spaces},
  which describe internal \icats{} in internal \icats{} in \ldots{} in \icats{}. A special
  class of these was first introduced by Barwick~\cite{BarwickThesis} as a model for
  $(\infty,n)$-categories. On the other hand, the Dunn--Lurie
  additivity theorem \cite[Theorem 5.1.2.2]{ha} implies that Segal
  $\simp^{n,\op,\flat}$-objects are equivalent to
  $\mathbb{E}_{n}$-algebras, \ie{} algebras for the
  little $n$-disc operad.
\end{example}

\begin{example}\label{ex:Theta}
  Let $\bbTheta_{n}$ be defined inductively by $\bbTheta_{0}:= *$
  and $\bbTheta_{n} := \simp \wr \bbTheta_{n-1}$, where for any
  category $\mathbf{C}$ the wreath product $\simp \wr \mathbf{C}$ has
  objects $[n](C_{1},\ldots,C_{n})$ with $C_{i} \in \mathbf{C}$, and
  morphisms $[n](C_{1},\ldots,C_{n}) \to [m](C'_{1},\ldots,C'_{m})$
  given by morphisms $\phi \colon [n] \to [m]$ in $\simp$ together
  with maps $\psi_{ij} \colon C_{i} \to C_{j}$ in $\mathbf{C}$
  whenever $\phi(i-1)< j \leq \phi(i)$. (This category was first
  considered in unpublished work of Joyal; the ``wreath product''
  definition is due to Berger~\cite{Berger}.) Then $\bbTheta_{n}$ has
  an inductively defined factorization system (first defined in
  \cite[Lemma 1.11]{BergerNerve}): the morphism above is
  \emph{inert} (or \emph{active}) if $\phi$ is inert (active) and each
  $\psi_{ij}$ is inert (active). We can again use this to define two
  algebraic patterns. To do so we need some notation: We
  inductively define objects $C_{0},\ldots,C_{n}$ in $\bbTheta_{n}$ by
  $C_{0}:= [0]()$ and $C_{n} := [1](C_{n-1})$, starting with $C_{0}$
  being the unique object of $\bbTheta_{0}$. Then
  \begin{itemize}
  \item $\bbTheta_{n}^{\op,\natural}$ is defined by taking
    $\bbTheta_{n}^{\op,\natural,\el}$ to contain the objects
    $C_{0},\ldots,C_{n}$; we can depict this category as
    \[ C_{n} \rightrightarrows C_{n-1} \rightrightarrows  \cdots
      \rightrightarrows  C_{0}.\]
  \item $\bbTheta_{n}^{\op,\flat}$ is defined by taking
    $\bbTheta_{n}^{\op,\flat,\el}$ to contain the single
    object $C_{n}$.
  \end{itemize}
  Segal $\bbTheta_{n}^{\op,\natural}$-spaces are then precisely Rezk's
  \emph{$\bbTheta_{n}$-spaces} \cite{RezkThetan}, which describe the
  algebraic structure of
  $(\infty,n)$-categories. On the other hand, Segal
  $\bbTheta_{n}^{\op,\flat}$-objects are again equivalent to
  $\mathbb{E}_{n}$-algebras --- this follows from \cite[Theorem
  8.12]{bar} together with the Dunn--Lurie additivity theorem.
\end{example}

\begin{example}\label{ex Leinstercat}
  All the examples considered so far are special cases of the
  following construction, due to Barwick: Suppose $\Phi$ is a \emph{perfect operator
    category} in the sense of \cite{bar}, and let
  $\Lambda(\Phi)$ be its Leinster category, which is the Kleisli
  category of a certain monad on $\Phi$. This has an active-inert
  factorization system by \cite[Lemma 7.3]{bar}, where the active
  morphisms are the free morphisms on morphisms of $\Phi$. Using this
  factorization system we can define two natural algebraic patterns:
  \begin{itemize}
  \item $\Lambda(\Phi)^{\flat}$ is defined by taking
    $\Lambda(\Phi)^{\flat,\el}$ to consist only of the terminal object
    $* \in \Phi$,
  \item $\Lambda(\Phi)^{\natural}$ is defined by taking
    $\Lambda(\Phi)^{\natural,\el}$ to contain all objects $E$ such that
    there is an inert map $* \intto E$ in $\Lambda(\Phi)$.
  \end{itemize}
  If $\mathbb{O}$ denotes the category of (possibly empty) ordered
  finite sets then $\Lambda(\mathbb{O}) \simeq \Dop$, while if
  $\mathbb{F}$ denotes the category of finite sets then
  $\Lambda(\mathbb{F}) \simeq \mathbb{F}_{*}$, and these pattern
  structures agree with those defined above. The same holds for
  $\bbTheta_{n}^{\op}$, which can be described as the Leinster
  category of a wreath product $\mathbb{O}^{\wr n}$ of operator
  categories.
\end{example}

\begin{example}\label{ex Omega}
  Let $\bbO$ be the \emph{dendroidal category} of Moerdijk and Weiss
  \cite[\S 3]{MoerdijkWeiss}; this can be defined as the category of free
  operads on trees. This has a natural active-inert
  factorization system, described for example in \cite{Kock} (where
  the inert maps are called ``free'' and the active ones
  ``boundary-preserving''). Using this we can define an algebraic
  pattern $\bbO^{\op,\natural}$ where $\bbO^{\op,\natural,\el}$
  consists of the \emph{corollas} $C_{n}$ (i.e. trees with one vertex) and the plain edge
  $\eta$. Segal $\bbO^{\op,\natural}$-spaces are the dendroidal Segal
  spaces introduced by Cisinski and Moerdijk~\cite{CisinskiMoerdijk2},
  which describe the algebraic structure of $\infty$-operads. The
  Segal condition says that the value of a Segal object at a tree
  decomposes as a limit over the corollas and edges of the tree. (We
  can also consider a pattern $\bbO^{\op,\flat}$ where the elementary
  objects are just the corollas; then Segal $\bbO^{\op,\flat}$-spaces
  describe \iopds{} with a single object.)
\end{example}

\begin{example}\label{ex DPhi}
  If $\Phi$ is an operator category, let $\simp_{\Phi}$ be the
  category defined in \cite[Definition 2.4]{bar}. This has 
  pairs $([m], f \colon [m] \to \Phi)$ as objects, and morphisms
  $([m], f) \to ([n], g)$ are given by morphisms
  $\phi \colon [m] \to [n]$ in $\simp$ together with certain natural
  transformations $\eta \colon f \to g \circ \phi$. We define a
  morphism $(\phi,\eta)\colon ([m], f) \to ([n], g)$ in $\simp_\Phi$
  to be \emph{inert} if $\phi$ is inert in $\simp$, and \emph{active}
  if $\phi$ is active and $\eta_i\colon f(i)\to g(\phi(i))$ is an
  isomorphism for every $0\leq i\leq m$. This gives an inert--active
  factorization system on $\Dop_{\Phi}$, and we define an algebraic
  pattern $\simp_{\Phi}^{\op,\natural}$ by taking the elementary
  objects to be $([0], *)$ and $([1], I \to *)$ (where $*$ denotes the
  terminal object). Then Segal $\simp_{\Phi}^{\op,\natural}$-spaces
  are precisely the Segal $\Phi$-operads of \cite[\S 2]{bar}, which
  describe \emph{$\Phi$-\iopds{}}. (When $\Phi$ is $\mathbb{F}$ these
  agree with \iopds{} in the
  sense of Lurie by \cite[Theorem
  10.16]{bar}, and with dendroidal Segal spaces by \cite[Theorem
  1.1]{ChuHaugsengHeuts}.)
\end{example}

\begin{example}\label{ex Gamma}
  Let $\bbGamma$ be the category of acyclic connected finite directed graphs
  defined by Hackney, Robertson, and Yau in \cite{HRYProperad}. Then
  $\bbGamma^{\op}$ has an inert--active factorization system described
  in \cite[2.4.14]{Kock_Properads} (where the active maps are called
  ``refinements'' and the inert maps are called ``convex open
  inclusions''). Using this we can define an algebraic pattern
  structure $\bbGamma^{\op,\natural}$ by taking the elementary objects
  to be the elementary graphs with at most one vertex.
  Segal
  $\bbGamma^{\op,\natural}$-spaces are equivalent to the model
  of \emph{$\infty$-properads} as ``graphical spaces'' satisfying a
  Segal condition that is briefly discussed in \cite{PropLect}; this
  is presumably equivalent (after imposing a completeness condition)
  to the model of $\infty$-properads as certain presheaves of sets on
  $\bbGamma$ constructed in \cite{HRYProperad}. 
\end{example}

\begin{example}\label{ex Xi}
  Let $\bbXi$ denote the category of unrooted trees defined in
  \cite{HRYCyclic}. Then $\bbXi^{\op}$ has an inert--active
  factorization system, described in \cite[\S 4]{HRYCyclic}, and using
  this we can give $\bbXi^{\op}$ an algebraic pattern structure
  $\bbXi^{\op,\natural}$ where the elementary objects are the stars
  and the plain edge.  Segal $\bbXi^{\op,\natural}$-spaces are then
  precisely the model for cyclic \iopds{} considered by Hackney,
  Robertson, and Yau~\cite{HRYCyclic}.
\end{example}

\begin{example}
  Let $\mathbf{U}$ denote the category of connected graphs defined in
  \cite{HRYModular}. Then $\mathbf{U}^{\op}$ has an inert--active
  factorization system, described in \cite[\S 2.1]{HRYModular}, and we
  can use this to equip $\mathbf{U}^{\op}$ with an algebraic pattern
  structure $\mathbf{U}^{\op,\natural}$ where the elementary objects
  are the stars and the plain edge. We can also consider an algebraic
  pattern $\mathbf{U}^{\op,\flat}$ where the elementary objects are
  just the stars; Segal $\mathbf{U}^{\op,\flat}$-objects are then the
  \emph{Segal modular operads} defined by Hackney, Robertson, and
  Yau~\cite{HRYModular}.
\end{example}

\begin{remark}
  Below in \S\ref{sec:WSF} we will define (weak) Segal fibrations over
  an algebraic pattern, which give general classes of examples of
  algebraic patterns. As a special case, we will see that every
  \iopd{} $\mathcal{O}$ in the sense of Lurie~\cite{ha} has an
  algebraic pattern structure $\mathcal{O}^{\flat}$ such that a Segal
  $\mathcal{O}^{\flat}$-object in an \icat{} $\mathcal{C}$ with finite
  products is precisely an $\mathcal{O}$-monoid in $\mathcal{C}$.
\end{remark}

\section{Morphisms of Algebraic Patterns}\label{sec:pattmor}
In this section we define morphisms of algebraic patterns, and
consider when they are compatible with Segal objects. We then discuss
some examples of such morphisms.

\begin{defn}
Let $\mathcal{O}$ and $\mathcal{P}$ be algebraic patterns. A \emph{morphism of
algebraic patterns} from $\mathcal{O}$ to $\mathcal{P}$ is a functor $f
\colon \mathcal{O} \to \mathcal{P}$ such that $f$ preserves both
active and inert maps, and takes elementary object in $\mathcal{O}$ to
elementary objects in $\mathcal{P}$.
\end{defn}

In general, morphisms of algebraic patterns do not necessarily
interact well with Segal objects. We therefore isolate the class of
morphisms that preserve Segal objects under restriction:
\begin{defn}
A morphism of algebraic patterns $f \colon \mathcal{O} \to \mathcal{P}$
is called a \emph{Segal morphism} if it satisfies the following
condition:
\begin{itemize}
\item[($\ast$)] For all $X \in \mathcal{O}$ the induced functor
$\mathcal{O}^{\el}_{X/} \to \mathcal{P}^{\el}_{f(X)/}$ induces an
equivalence
\[ \lim_{\mathcal{P}^{\el}_{f(X)/}} F \isoto
\lim_{\mathcal{O}^{\el}_{X/}} F \circ f^{\el}\]
for every Segal $\mathcal{P}$-space $F \colon \mathcal{P} \to
\mathcal{S}$.
\end{itemize}
\end{defn}

\begin{remark}
  The condition depends only on the restriction of $F$ to
  $\mathcal{P}^{\el}$, so we could equivalently have  considered
  functors $\mathcal{P}^{\el} \to \mathcal{S}$ that occur as
  restrictions of Segal $\mathcal{P}$-spaces.
\end{remark}

\begin{remark}
  In practice, a morphism $f$ is a Segal morphism because the functor
  $\mathcal{O}^{\el}_{X/} \to \mathcal{P}^{\el}_{f(X)/}$ is coinitial,
  in which case we say that $f$ is a \emph{strong Segal
    morphism}. However, the more general definition allows for the
  following characterization:
\end{remark}

\begin{lemma}\label{lem:Segalmor}
  The following are equivalent for a morphism of algebraic patterns
  $f \colon \mathcal{O} \to \mathcal{P}$:
\begin{enumerate}[(1)]
\item $f$ is a Segal morphism.
\item The functor $f^{*} \colon \Fun(\mathcal{P},\mathcal{S}) \to
\Fun(\mathcal{O}, \mathcal{S})$ restricts to a functor
$\Seg_{\mathcal{P}}(\mathcal{S}) \to
\Seg_{\mathcal{O}}(\mathcal{S})$.
\item For every \icat{} $\mathcal{C}$, the functor
$f^{*} \colon \Fun(\mathcal{P},\mathcal{C}) \to \Fun(\mathcal{O},
\mathcal{C})$ restricts to a functor
$\Seg_{\mathcal{P}}(\mathcal{C}) \to
\Seg_{\mathcal{O}}(\mathcal{C})$.
\end{enumerate}
\end{lemma}
\begin{proof}
  It is immediate from the definition that (1) is equivalent to (2)
  and that (3) implies (2). It remains to check that (2) implies
  (3). Suppose $F \colon \mathcal{P} \to \mathcal{C}$ is a Segal
  $\mathcal{P}$-object; we need to show that $f^{*}F$ is a Segal
  $\mathcal{O}$-object, \ie{} that for all $X \in \mathcal{O}$ the
  natural map
  \[ \lim_{\mathcal{P}^{\el}_{f(X)/}} F \to
    \lim_{\mathcal{O}^{\el}_{X/}} F \circ f^{\el}\]
  is an equivalence in $\mathcal{C}$. Equivalently, we must show that
  for any $C \in \mathcal{C}$, the map of spaces
  \[ \lim_{\mathcal{P}^{\el}_{f(X)/}} \Map(C, F) \to
    \lim_{\mathcal{O}^{\el}_{X/}} \Map(C, F) \circ f^{\el}\] is an
  equivalence, which is true since $\Map(C, F)$ is a Segal
  $\mathcal{P}$-space.
\end{proof}

\begin{remark}
  One might feel that the Segal property is sufficiently fundamental
  that it should be included as part of the notion of a morphism of
  algebraic patterns. However, more general morphisms also turn out to
  be occasionally useful. For example, the identity functor of
  $\xF_{*}$ viewed as a functor
  $\xF_{*}^{\flat} \to \xF_{*}^{\natural}$ is a morphism of patterns,
  but is not a Segal morphism, and we will see later in
  \S\ref{sec:RKE} that it induces a functor from Segal
  $\xF_{*}^{\flat}$-objects to Segal $\xF_{*}^{\natural}$-object that
  can be viewed as a right Kan extension along $\id_{\xF_{*}}$.
\end{remark}

\begin{propn}\label{propn:f!adj}
Suppose $f \colon \mathcal{O} \to \mathcal{P}$ is a Segal morphism of
algebraic patterns, and $\mathcal{C}$ is a presentable \icat{}. Then
there is an adjunction
\[ L_{\Seg}f_{!} \colon \Seg_{\mathcal{O}}(\mathcal{C}) \rightleftarrows
\Seg_{\mathcal{P}}(\mathcal{C}) \cocolon f^{*} \]
where $L_{\Seg}$ is the localization functor left adjoint to the
inclusion $\Seg_{\mathcal{O}}(\mathcal{C}) \hookrightarrow
\Fun(\mathcal{O}, \mathcal{C})$, and $f_{!}$ is the functor of left
Kan extension along $f$.
\end{propn}
\begin{proof}
  Since $f^{*}$ restricts to a functor on Segal objects, for $F \in
  \Seg_{\mathcal{P}}(\mathcal{C})$ and $G \in
  \Seg_{\mathcal{O}}(\mathcal{C})$ we have
  a natural equivalence
  \[\Map_{\Seg_{\mathcal{O}}(\mathcal{C})}(L_{\Seg}f_{!}F, G) \simeq
    \Map_{\Fun(\mathcal{O}, \mathcal{C})}(f_{!}F, G) \simeq
    \Map_{\Fun(\mathcal{P}, \mathcal{C})}(F, f^{*}G) \simeq
    \Map_{\Seg_{\mathcal{P}}(\mathcal{C})}(F, f^{*}G),\]
  which implies the claim.
\end{proof}

\begin{remark}
  Below in \S\ref{sec:LKE} we will give conditions on a morphism $f$
  such that the left Kan extension functor $f_{!}$ preserves Segal
  objects, and so gives a left adjoint to $f^{*}$ without localizing.
\end{remark}

We now consider some examples of morphisms of patterns:
\begin{example}\label{ex:simptoF}
  There is a functor $|\blank| \colon \Dop \to \xF_{*}$ which takes an object $[n]$ to $|[n]|:= \angled{n}$ and a morphism $\alpha\colon [n]\to [m]$ in $\simp$ to $|\alpha|\colon |[m]|\to |[n]|$ given by 
\begin{equation*}
|\alpha|(i) = \begin{cases}
j &\text{if $\alpha(j-1)< j\leq \alpha(j)$}\\
0 &\text{otherwise}
\end{cases}
\end{equation*}
This functor gives a Segal morphism of algebraic patterns
$\simp^{\op,\natural} \to \xF_{*}^{\natural}$ as well as
$\simp^{\op,\flat} \to \xF_{*}^{\flat}$.
\end{example}

\begin{example}
  There is a functor $\tau_{n} \colon \simp^{n,\op} \to
  \bbTheta_{n}^{\op}$, defined inductively by setting $\tau_0:=\id$ and 
  \[\tau_{n}([i_{1}],\ldots,[i_{n}]) :=
    [i_{1}](\tau_{n-1}([i_{2}],\ldots,[i_{n}]), \ldots,
    \tau_{n-1}([i_{2}],\ldots,[i_{n}])).\] This functor gives a Segal
  morphism of algebraic patterns
  $\simp^{n,\op,\natural} \to \bbTheta_{n}^{\op,\natural}$ as well as
  $\simp^{n,\op,\flat} \to \bbTheta_{n}^{\op,\flat}$.
\end{example}

\begin{example}
  The previous examples are special cases of the following: Let $f
  \colon \Phi \to \Psi$ be an \emph{operator morphism} between perfect
  operator categories, as defined in \cite[Definition 1.10]{bar}.  
  As discussed in \cite[\S 7]{bar} this induces a functor $\Lambda(f)
  \colon \Lambda(\Phi) \to \Lambda(\Psi)$ between the corresponding
  Leinster categories, and it is easy to check that this preserves the
  inert and active morphisms. Since operator morphisms preserve terminal objects by definition, it follows from Example~\ref{ex Leinstercat} that $\Lambda(f)$ preserves elementary objects, and hence gives morphisms of algebraic
  patterns $\Lambda(\Phi)^{\natural} \to \Lambda(\Psi)^{\natural}$ and
  $\Lambda(\Phi)^{\flat} \to \Lambda(\Psi)^{\flat}$. The latter is
  evidently a Segal morphism, since
  \[ \Lambda(\Phi)^{\flat,\el}_{I/} \cong \{* \to I\} \cong \{* \to
    f(I)\} \cong \Lambda(\Psi)^{\flat,\el}_{f(I)/},\] where the seond
  isomorphism is part of the definition of an operator morphism.
\end{example}

\begin{example}
  Every operator category $\Phi$ has a unique operator morphism
  $|\blank| \colon \Phi \to \xF$, which gives a Segal morphism
  $\Lambda(\Phi)^{\flat} \to \xF_{*}^{\flat}$. This is also a Segal
  morphism $\Lambda(\Phi)^{\natural} \to \xF_{*}^{\natural}$ provided
  the category $\Lambda(\Phi)^{\natural,\el}_{I/}$ is weakly
  contractible for all $I \in \Phi$.
\end{example}

\begin{example}\label{ex:cyclicmor}
  By \cite[Definition 1.20]{HRYCyclic}, the category $\bbO$ of trees
  can be identified with a subcategory of the category $\bbXi$ of
  unrooted trees, and \cite[Definition 4.2]{HRYCyclic} and
  \cite[Remark 4.3]{HRYCyclic} show that this inclusion gives a
  morphism of algebraic patterns
  $\iota\colon \bbO^{\op,\natural}\to \bbXi^{\op,\natural}$. The
  description of morphisms in $\bbO^\op$ in \cite[Definition
  1.20]{HRYCyclic} implies that for every $X\in \bbO^{\op}$ and every
  $\alpha\in \bbXi_{\iota X/}^{\op, \el}$, the $\infty$-category
  $\bbO^{\op,\el}_{X/}\times_{\bbXi_{\iota X/}^{\op, \el}}
  (\bbXi_{\iota X/}^{\op, \el})_{/\alpha}$ has a terminal object. In
  particular, the functor $\bbO^{\op}_{X/}\to \bbXi^{\op}_{\iota X/}$
  is coinitial, and hence $\iota$ is a strong Segal morphism.  The resulting
  functor
  \[\iota^* \colon  \Seg_{\bbXi^{\op,\natural}}(\xS) \to \Seg_{\bbO^{\op,\natural}}(\xS)\]
  is the forgetful functor from
  cyclic $\infty$-operads to $\infty$-operads.
\end{example}

\section{The $\infty$-Category of Algebraic
  Patterns}\label{sec:pattcat}
In this section we construct the \icat{} of algebraic patterns, and
describe limits and filtered colimits in this \icat{}. As a first
step, we consider the \icat{} of \icats{} equipped with a factorization system:
\begin{defn}
We define $\name{Fact}$ to be the full subcategory of
$\Fun(\Lambda^{2}_{2}, \CatI)$ (where $\Lambda^{2}_{2}$ denotes the
category $0 \to 2 \from 1$) spanned by those cospans
\[ \mathcal{C}_{L} \to \mathcal{C} \from \mathcal{C}_{R} \] that
describe factorization systems, \ie{} those such that the functors
$\mathcal{C}_{L},\mathcal{C}_{R} \to \mathcal{C}$ are essentially
surjective subcategory inclusions, and
$\Fun_{L,R}(\Delta^{2},\mathcal{C}) \to \Fun(\Delta^{\{0,2\}},
\mathcal{C})$ is an equivalence, where the domain is defined as the
pullback
\nolabelcsquare{\Fun_{L,R}(\Delta^{2},\mathcal{C})}{\Fun(\Delta^{2},\mathcal{C})}{\Fun(\Delta^{1},
\mathcal{C}_{L}) \times \Fun(\Delta^{1},
\mathcal{C}_{R})}{\Fun(\Delta^{\{0,1\}}, \mathcal{C}) \times
\Fun(\Delta^{\{1,2\}}, \mathcal{C}).}
\end{defn}

\begin{propn}\label{propn:factlim}
  The \icat{} $\name{Fact}$ is closed under limits and filtered
  colimits in $\Fun(\Lambda^{2}_{2},\CatI)$. In particular, the
  \icat{} $\name{Fact}$ has limits and filtered colimits, and the
  forgetful functor to $\CatI$ preserves these.
\end{propn}

This will follow from the following observation:
\begin{lemma}\label{lem:subcatlim}
  In the \icat{} $\Fun(\Delta^{1},\CatI)$, the full subcategories of
  subcategory inclusions\footnote{Note that we use ``subcategory
    inclusion'' in the equivalence-invariant sense --- in other words,
    a subcategory in our sense must include all equivalences between
    its objects.}, essentially surjective subcategory inclusions, and
  full subcategory inclusions, are all closed under limits and
  filtered colimits.
\end{lemma}
\begin{proof}
  A functor $F \colon \mathcal{C} \to \mathcal{D}$ is a subcategory
  inclusion precisely when
  $\mathcal{C}^{\simeq} \to \mathcal{D}^{\simeq}$ is a monomorphism of
  spaces, and $\Map_{\mathcal{C}}(x,y) \to \Map_{\mathcal{D}}(Fx,Fy)$
  is a monomorphism of spaces for all $x,y \in \mathcal{C}$. A
  subcategory inclusion $F$ is essentially surjective if the map
  $\mathcal{C}^{\simeq} \to \mathcal{D}^{\simeq}$ is an equivalence,
  and a full subcategory inclusion if the maps
  $\Map_{\mathcal{C}}(x,y) \to \Map_{\mathcal{D}}(Fx,Fy)$ are
  equivalences for all $x,y \in \mathcal{C}$. Since mapping spaces and
  the underlying space of a limit (or filtered colimit) in $\CatI$ are
  computed as limits (or filtered colimits) of spaces, it suffices to
  observe that equivalences and monomorphisms are closed under limits
  and filtered colimits in $\mathcal{S}$. This is obvious for
  equivalences, and for monomorphisms it follows from the
  characterization of these by \cite[Lemma 5.5.6.15]{ht} as the
  morphisms $f \colon X \to Y$ such that the diagonal
  $X \to X \times_{Y}X$ is an equivalence, since filtered colimits
  commute with finite limits and limits commute.
\end{proof}

\begin{proof}[Proof of Proposition~\ref{propn:factlim}]
  It follows from Lemma~\ref{lem:subcatlim} that cospans of
  subcategory inclusions are closed under limits and filtered colimits
  in
  $\Fun(\Lambda^{2}_{2},\CatI)$. Since limits commute, the \icat{}
  $\Fun_{L,R}(\Delta^{2},\blank)$, viewed as a functor
  $\Fun(\Lambda^{2}_{2},\CatI) \to \CatI$, preserves limits, which
  implies that objects such that the natural map
  $\Fun_{L,R}(\Delta^{2},\blank) \to \Fun(\Delta^{\{0,2\}}, \blank)$ is an
  equivalence are also closed under limits. The same holds for
  filtered colimits, since the objects mapped out of in the definition
  of $\Fun_{L,R}(\Delta^{2},\blank)$ are compact, and filtered
  colimits commute with finite limits in $\CatI$.
\end{proof}

\begin{defn}\label{def AlgPatt}
  We now define the \icat{} $\name{AlgPatt}$ of algebraic patterns as
  the full subcategory of the fibre product
  $\name{Fact} \times_{\CatI} \Fun(\Delta^{1}, \CatI)$ (where the
  pullback is over $\name{ev}_{0} \colon \name{Fact} \to \CatI$ and
  $\name{ev}_{1} \colon \Fun(\Delta^{1}, \CatI) \to \CatI$) containing
  the objects
\[ \mathcal{C}' \to \mathcal{C}_{L} \to \mathcal{C} \from
\mathcal{C}_{R} \]
where $\mathcal{C}' \to \mathcal{C}_{L}$ is a full subcategory
inclusion.
\end{defn}

Applying Lemma~\ref{lem:subcatlim} again, now in the case of full
subcategory inclusions, we get:
\begin{cor}\label{cor:PattLim}
The full subcategory $\AlgPatt$ is closed under limits and
filtered colimits in
\[\Fun(\Lambda^{2}_{2}, \CatI) \times_{\CatI} \Fun(\Delta^{1},
\CatI).\] In particular, $\AlgPatt$ has limits and filtered
colimits, and
the forgetful functor to $\CatI$ preserves these. \qed
\end{cor}

\begin{remark}
  The \icat{} $\AlgPatt$ contains \emph{all} morphisms of algebraic
  patterns; restricting these to Segal morphisms gives a (wide)
  subcategory $\AlgPatt^{\Seg}$. However, note that Segal morphisms do
  not seem to be closed under filtered colimits or general pullbacks,
  though by Lemma~\ref{lem:Segalmor} and the next example they
  \emph{are} closed under finite products.
\end{remark}

\begin{example}
  For any pair of algebraic patterns $\mathcal{O}$, $\mathcal{P}$ we
  have a cartesian product pattern $\mathcal{O} \times
  \mathcal{P}$. For this we have an equivalence
  \[ \Seg_{\mathcal{O} \times \mathcal{P}}(\mathcal{C}) \simeq
    \Seg_{\mathcal{O}}(\Seg_{\mathcal{P}}(\mathcal{C}))\] for any
  $\mathcal{O} \times \mathcal{P}$-complete \icat{} $\mathcal{C}$. To
  see this, observe that a right Kan extension along
  $\mathcal{O}^{\el} \times \mathcal{P}^{\el}\to \mathcal{O}^{\xint} \times
  \mathcal{P}^{\xint}$ can be computed in two stages in two ways, by first
  doing the right Kan extension to either
  $\mathcal{O}^{\el} \times \mathcal{P}^{\xint}$ or $\mathcal{O}^{\xint} \times
  \mathcal{P}^{\el}$; this shows that
  $F \colon \mathcal{O} \times \mathcal{P} \to \mathcal{C}$ is a Segal
  object \IFF{} $F(O,\blank)$ is a $\mathcal{P}$-Segal object for all
  $O \in \mathcal{O}$ and $F(\blank,P)$ is an $\mathcal{O}$-Segal
  object for all $P \in \mathcal{P}$.
\end{example}

\begin{example}
  The pattern $\simp^{\op,\flat}$ can be described as the pullback
  $\simp^{\op,\natural} \times_{\xF_{*}^{\natural}}
  \xF_{*}^{\flat}$ using the map $\simp^{\op,\natural} \to
  \xF_{*}^{\natural}$ from Example~\ref{ex:simptoF} and the identity
  of $\xF_{*}$ viewed as a morphism of patterns $\xF_{*}^{\flat} \to
  \xF_{*}^{\natural}$. (Similarly, for the other pairs of patterns
  $\mathcal{O}^{\flat}, \mathcal{O}^{\natural}$ mentioned in
  \S\ref{sec:ex} the pattern $\mathcal{O}^{\flat}$ is the pullback
  $\mathcal{O}^{\natural}\times_{\xF_{*}^{\natural}} \xF_{*}^{\flat}$
  for a morphism of patterns $\mathcal{O}^{\natural} \to \xF_{*}^{\natural}$.)
\end{example}

\begin{example}
  Let $\bbTheta^{\op,\natural}$ be the colimit
  $\colim_{n\geq 0} \bbTheta_{n}^{\op,\natural}$ induced by the
  sequence of natural inclusions
  $\bbTheta_{n}^{\op,\natural} \hookrightarrow
  \bbTheta_{n+1}^{\op,\natural}$, $n\geq 0$, where
  $\bbTheta_{n}^{\op,\natural}$ is the algebraic pattern defined in
  Example~\ref{ex:Theta}. The underlying category $\bbTheta$ is
  equivalent to that introduced by Joyal~\cite{JoyalTheta} to give a
  definition of weak higher categories.  It is easy to see that in
  this case we have an equivalence
  \[\Seg_{\bbTheta^{\op,\natural}}(\xS)\simeq \lim_{n\geq 0}
    \Seg_{\bbTheta^{\op,\natural}_{n}}(\xS),\] so that Segal
  $\bbTheta^{\op,\natural}$-spaces model $(\infty,\infty)$-categories (in
  the inductive sense). In
  particular, the canonical functor
  $\Seg_{\bbTheta^{\op,\natural}}(\xS)\to
  \Seg_{\bbTheta_{n}^{\op,\natural}}(\xS)$ gives the underlying
  $(\infty, n)$-category of an $(\infty,\infty)$-category.
\end{example}

\section{Right Kan Extensions and Segal Objects}\label{sec:RKE}
Our goal in this section is to give a sufficient criterion on a
morphism of algebraic patterns $f \colon \mathcal{O} \to \mathcal{P}$
such that right Kan extension along $f$ preserves Segal objects.

\begin{defn}
We say that a morphism $f \colon \mathcal{O} \to \mathcal{P}$ of
algebraic patterns has \emph{unique lifting of active morphisms} if
for every active morphism $\phi \colon P \to f(O)$ in $\mathcal{P}$,
the \igpd{} of lifts of $\phi$ to an active morphism $O' \to O$ in
$\mathcal{O}$ is contractible. More precisely, the fibre
$(\mathcal{O}^{\act}_{/O})^{\simeq}_{\phi}$ of the morphism
\[ (\mathcal{O}^{\act}_{/O})^{\simeq} \to
  (\mathcal{P}^{\act}_{/f(O)})^{ \simeq}\] at $\phi$ is contractible.
Equivalently, $f$ has unique lifting of active morphisms if this
morphism of \igpds{} is an equivalence for all $O \in \mathcal{O}$.
\end{defn}

\begin{lemma}\label{lem:uniqactcoinit}
A morphism of algebraic patterns $f \colon \mathcal{O} \to \mathcal{P}$
has unique lifting of active morphisms \IFF{} it satisfies the following
condition:
\begin{itemize}
\item[($\ast$)] For all $P \in \mathcal{P}$ the functor
\[ \mathcal{O}^{\xint}_{P/} \to \mathcal{O}_{P/}\]
is coinitial.
\end{itemize}
\end{lemma}
\begin{proof}
By \cite[Theorem 4.1.3.1]{ht}, the functor $\mathcal{O}^{\xint}_{P/}
\to \mathcal{O}_{P/}$ is coinitial \IFF{} for every morphism $\phi
\colon P \to
f(O)$ in $\mathcal{P}$, the \icat{}
$(\mathcal{O}^{\xint}_{P/})_{/\phi}$ is weakly contractible. This
\icat{} has objects pairs
\[ \left(O' \xto{\alpha} O,
\begin{tikzcd}
{} & P \arrow[inert]{dl}[above left]{\iota} \arrow{dr}{\phi} \\
f(O') \arrow{rr}{f(\alpha)} & & f(O)
\end{tikzcd}
\right),
\]
where $\iota$ is inert. The morphism $\alpha$ has an essentially
unique inert--active factorization, and since $f$ is compatible with
this factorization we see that the full subcategory of objects where
$\alpha$ is active is cofinal. By uniqueness of factorizations a
morphism in this subcategory is required to be an equivalence, hence
this is an \igpd{}, and so $(\ast)$ is equivalent to this
$\infty$-groupoid being contractible. But an object in this
subcategory gives an inert--active factorization of $\phi$, and we see
that it is equivalent to the \igpd{} of lifts of the active part of
$\phi$ to an active morphism in $\mathcal{O}$.
\end{proof}

\begin{propn}\label{propn:RKESeg}
Suppose $f \colon \mathcal{O} \to \mathcal{P}$ is a morphism of
algebraic patterns that has unique lifting of
active morphisms and $\mathcal{C}$ is an $\mathcal{O}$- and
$\mathcal{P}$-complete \icat{} such that the pointwise right Kan
extension
\[ f_{*} \colon \Fun(\mathcal{O}, \mathcal{C}) \to \Fun(\mathcal{P},
\mathcal{C}) \]
exists. Then $f_{*}$ restricts to a functor
\[ f_{*} \colon \Seg_{\mathcal{O}}(\mathcal{C}) \to
  \Seg_{\mathcal{P}}(\mathcal{C}).\]
\end{propn}

\begin{remark}
We emphasize that the condition of unique lifting of active morphisms is
far from a \emph{necessary} one. Indeed, the functor $f_{*}$ will
preserve Segal objects \IFF{} its left adjoint $f^{*}$ preserves
Segal equivalences. In \cite{ChuHaugsengHeuts} the latter condition
was checked for a certain morphism
$\tau \colon \DF^{1,\op} \to \bbOmega^{\op}$, which clearly does
\emph{not} have unique lifting of active morphisms.
\end{remark}

\begin{proof}[Proof of Proposition~\ref{propn:RKESeg}]
By Lemma~\ref{lem:uniqactcoinit}, the condition that $f$ has unique
lifting of active morphisms implies that for any functor
$F \colon \mathcal{O} \to \mathcal{C}$, the Beck--Chevalley
transformation
\[ (f_{*}F)|_{\mathcal{P}^{\xint}} \to
f^{\xint}_{*}(F|_{\mathcal{O}^{\xint}})\]
is an equivalence. If $F$ is a Segal $\mathcal{O}$-object, then
$F|_{\mathcal{O}^{\xint}} \simeq
i_{\mathcal{O},*}F|_{\mathcal{O}^{\el}}$, where $i_{\mathcal{O}}$ is
the inclusion $\mathcal{O}^{\el} \hookrightarrow
\mathcal{O}^{\xint}$, so in this case we have
$(f_{*}F)|_{\mathcal{P}^{\xint}} \simeq
f^{\xint}_{*}i_{\mathcal{O},*}F|_{\mathcal{O}^{\el}}$. By 
naturality of right Kan extensions in
the commutative square
\[
\begin{tikzcd}
\mathcal{O}^{\el} \arrow[d, swap, "i_{\mathcal{O}}"] \arrow{r}{f^{\el}} &
\mathcal{P}^{\el} \arrow{d}{i_{\mathcal{P}}} \\
\mathcal{O}^{\xint} \arrow{r}{f^{\xint}} & \mathcal{P}^{\xint}
\end{tikzcd}
\]
this can in turn be identified with
$i_{\mathcal{P},*}f^{\el}_{*}F|_{\mathcal{O}^{\el}}$. Moreover,
since $\mathcal{P}^{\el}$ is a full subcategory of
$\mathcal{P}^{\xint}$, we have
\[ f^{\el}_{*}F|_{\mathcal{O}^{\el}} \simeq
i_{\mathcal{P}}^{*}i_{\mathcal{P},*}f^{\el}_{*}F|_{\mathcal{O}^{\el}} \simeq
i_{\mathcal{P}}^{*}f^{\xint}_{*}i_{\mathcal{O},*}F|_{\mathcal{O}^{\el}}
\simeq i_{\mathcal{P}}^{*}f^{\xint}_{*}
F|_{\mathcal{O}^{\xint}}.\]
Combining these equivalences, we see that
$(f_{*}F)|_{\mathcal{P}^{\xint}} \simeq i_{\mathcal{P},*}(i_{\mathcal{P}}^* f^{\xint}_*F|_{\xxO^\xint})\simeq 
i_{\mathcal{P},*}(f_{*}F)|_{\mathcal{P}^{\xel}}$, where the second
equivalence is given by $i_\xxP^* f_*^{\xint}(F|_{\xxO^{\xint}})\simeq
i^*_\xxP (f_*F)|_{\xxP^{\xint}}\simeq (f_*F)|_{\xxP^\el}$. Hence $f_{*}F$ is
a Segal $\mathcal{P}$-object.
\end{proof}

\begin{remark}\label{rmk:RKESegadj}
  If $f$ in Proposition~\ref{propn:RKESeg} is moreover a Segal
  morphism, we get an adjunction
\[ f^{*} \colon \Seg_{\mathcal{P}}(\mathcal{C}) \rightleftarrows
  \Seg_{\mathcal{O}}(\mathcal{C}) \cocolon f_{*}\]
by restricting the adjunction $f^{*}\dashv f_{*}$ on functor \icats{}.
\end{remark}

\begin{example}\label{ex:twopattsamefact}
Suppose we have two categorical patterns $\mathfrak{O}_{1}$ and
$\mathfrak{O}_{2}$ with the same underlying \icat{} $\mathcal{O}$ and
the same inert--active factorization system, and
$\mathfrak{O}_{1}^{\el}$ is a full subcategory of
$\mathfrak{O}_{2}^{\el}$. Then the identity functor of $\mathcal{O}$
gives a morphism of algebraic patterns
$\mathfrak{O}_{1} \to \mathfrak{O}_{2}$ for which unique
lifting of active morphisms holds trivially. In this case, this just
means that the Segal condition for $\mathfrak{O}_{1}$ is stronger
than that for $\mathfrak{O}_{2}$. For example, this holds for the
identity morphism of $\xF_{*}$ viewed as a morphism $\xF_{*}^{\flat}
\to \xF_{*}^{\natural}$. On the other hand, the identity
functor would typically \emph{not} be a
Segal morphism.
\end{example}

\begin{example}
  The inclusion $i \colon \{[0]\} \to \simp^{\op,\natural}$ clearly
  has unique lifting of active morphisms, since the only active
  morphism to $[0]$ in $\simp^{\op}$ is the identity. In this case,
  the right Kan extension functor
\[ i_{*} \colon \mathcal{C} \simeq \Fun(\{[0]\}, \mathcal{C}) \to
\Fun(\simp^{\op}, \mathcal{C})\] takes an object
$C \in \mathcal{C}$ to the simplicial object $i_{*}C$ given by
$(i_{*}C)_{n} \simeq \prod_{i=0}^{n} C$, with face maps corresponding
to projections and degeneracies given by diagonal maps. This clearly
satisfies the Segal condition. More generally, the inclusion
$\bbTheta_{n-1}^{\op,\natural} \hookrightarrow
\bbTheta_{n}^{\op,\natural}$ has unique lifting of active
morphisms for all $n \geq 1$.
\end{example}

\begin{example}
  Let $\iota \colon \bbO^{\op,\natural} \to \bbXi^{\op,\natural}$ be
  the Segal morphism of Example~\ref{ex:cyclicmor}. Since the active morphisms in $\bbXi^{\op}$
  are the boundary-preserving ones, it is easy to see that $\iota$ has unique
  lifting of active morphisms. Then 
  Proposition~\ref{propn:RKESeg} and Remark~\ref{rmk:RKESegadj} give an adjunction
\[ \iota^* \colon \Seg_{\bbXi^{\op}}(\xS)\rightleftarrows
  \Seg_{\bbO^{\op,\natural}}(\xS) \cocolon \iota_*,\]
where $\iota_{*}$ is a right adjoint to the forgetful functor $\iota^{*}$ from
cyclic $\infty$-operads to $\infty$-operads. According to
\cite[\S 2.15]{DCHRightInd} the analogue of this right adjoint for
ordinary cyclic operads was first considered in the unpublished thesis
of J.~Templeton.
\end{example}

\section{Left Kan Extensions and Segal Objects}\label{sec:LKE}
In this section we will give conditions under which left Kan extension
along a morphism $f$ preserves Segal objects in $\mathcal{C}$. In
contrast to the case of right Kan extensions, this requires strong
assumptions on both $f$ and the target \icat{} $\mathcal{C}$. Part of
the condition is a uniqueness requirement on lifts of inert morphisms,
which we consider first:
\begin{defn}
A morphism of algebraic patterns $f \colon \mathcal{O} \to
\mathcal{P}$ is said to have \emph{unique lifting of inert
morphisms} if for every inert morphism $f(O) \to P$ the \igpd{} of
lifts to inert morphisms $O \to O'$ is contractible. More precisely, the fibre
$(\mathcal{O}^{\xint}_{O/})^{\simeq}_{\phi}$ of the morphism
\[ (\mathcal{O}^{\xint}_{O/})^{\simeq} \to
(\mathcal{P}^{\xint}_{f(O)/})^{ \simeq}\] at $\phi$ is
contractible. Equivalently, $f$ has unique lifting of inert morphisms if this morphism of \igpds{} is an
equivalence for all $O \in \mathcal{O}$.
\end{defn}

\begin{lemma}\label{lem:uniqintcoinit}
A morphism of algebraic patterns
$f \colon \mathcal{O} \to \mathcal{P}$ has unique lifting of inert
morphisms \IFF{} it satisfies the following condition:
\begin{itemize}
\item[($\ast$)] For all $P \in \mathcal{P}$ the functor
\[ \mathcal{O}^{\act}_{/P} \to \mathcal{O}_{/P}\] is cofinal.
\end{itemize}
\end{lemma}
\begin{proof}
This follows by the same argument as for Lemma~\ref{lem:uniqactcoinit}, with the roles
of active and inert morphisms reversed.
\end{proof}

Unique lifting of inert morphisms allows us to functorially transport
active morphisms along inert morphisms, in the following sense:
\begin{propn}\label{propn:stronginertlift}
Suppose $f \colon \mathcal{O} \to \mathcal{P}$ has unique
lifting of inert morphisms. Let
\[\mathcal{X} \subseteq \mathcal{O} \times_{\mathcal{P}}
  \mathcal{P}^{\Delta^{1}}\]
be the full subcategory of the fibre product over evaluation at $0$,
with objects those pairs \[(O, f(O) \xto{\phi} P)\]
where $\phi$ is active. Then the projection
$\mathcal{X} \to \mathcal{P}$ given by evaluation at $1 \in
\Delta^{1}$ is a cocartesian fibration,
and a morphism
\[ \left(
    \begin{tikzcd}
      O \arrow[d, swap, "\omega"] \\
      O'
    \end{tikzcd},
\begin{tikzcd}
f(O) \arrow[active]{r} \arrow[d, swap, "f(\omega)"] & P \arrow{d} \\
f(O') \arrow[active]{r} & P'
\end{tikzcd}
\right) \]
is cocartesian \IFF{} $\omega$ is inert.
\end{propn}
\begin{proof}
We first show that such a morphism with $\omega$ inert is
cocartesian. This means that given a morphism $O \to X$ in
$\mathcal{O}$ and a commutative diagram
\[
\begin{tikzcd}
f(O) \arrow{rr} \arrow[active]{dd} \arrow[inert]{dr}[below left]{f(\omega)} & & f(X) \arrow[active]{dd} \\
& f(O') \arrow[dashed]{ur}  \\
P \arrow{rr} \arrow{dr} & & Q, \\
& P' \arrow{ur} \arrow[lactive,crossing over]{uu}
\end{tikzcd}
\]
there exists a unique lift $O' \to X$ making the diagram
commute. 

The morphism $O \to X$ has a unique inert--active factorization as $O
\intto O'' \actto X$. Since $f$ is compatible with the factorization
system, we see that the unique inert--active factorization of $f(O) \to
Q$ is $f(O) \intto f(O'') \actto f(X) \actto Q$.

On the other hand, the inert--active factorization of $f(O') \to Q$
gives another factorization
$f(O) \intto f(O') \intto Q' \actto Q$, where by uniqueness we must
have $Q' \simeq f(O'')$. Since $f$ has unique lifts of inert
morphisms, the map $f(O') \intto f(O'')$ lifts to a unique morphism $O'
\intto O''$, and moreover by uniqueness the composite $O \intto O' \intto O''$ must be the
inert map $O \intto O''$ arising from the factorization of $O \to X$.

Thus, there are unique diagrams
\[
\begin{tikzcd}
O \arrow[inert]{r} \arrow[bend left=30]{rrr} & O' \arrow[inert]{r} & O'' \arrow[active]{r} & X,
\end{tikzcd}
\]
\[
\begin{tikzcd}
f(O) \arrow[bend left=30]{rrr} \arrow[active]{d} \arrow[inert]{r} & f(O')\arrow[active]{d} \arrow[inert]{r}  & f(O'')
\arrow[active]{d} \arrow[active]{r}  & f(X) \arrow[active]{dl}\\
P \arrow{r} & P' \arrow{r} & Q,
\end{tikzcd}
\]
which give the required unique factorization (since any other
factorization through $(O', f(O') \actto P')$ must induce these by
uniqueness of inert--active factorizations).

We next check that $\mathcal{X} \to \mathcal{P}$ is a cocartesian
fibration. This amounts to showing that cocartesian morphisms exist,
and by the first part of the proof it suffices to check that
given $(O, f(O) \overset{\phi}{\actto} P)$ with $\phi$ active and a
morphism $P \to P'$, there exists a morphism
\[ \left(
    \begin{tikzcd}
      O \arrow[d, swap, tail, "\omega"] \\
      O'
    \end{tikzcd}
    ,
\begin{tikzcd}
f(O) \arrow[active]{r} \arrow[d, swap, tail, "f(\omega)"] & P \arrow{d} \\
f(O') \arrow[active]{r} & P'
\end{tikzcd}
\right) \]
with $\omega$ inert. This again follows from unique lifting
of inert morphisms, which ensures that the inert--active factorization
of $f(O) \actto P \to P'$ gives such a diagram.

It remains to show that $\omega$ must be inert for any cocartesian
morphism. Since cocartesian morphisms are unique when they exist, this
follows from the existence of the cocartesian morphisms we just described.
\end{proof}

Straightening this cocartesian fibration, we get:
\begin{cor}\label{cor:Oactftr}
Suppose $f \colon \mathcal{O} \to \mathcal{P}$ has unique
lifting of inert morphisms. Then there is a functor
$\mathcal{P} \to \CatI$ that takes $P$ to
$\mathcal{O}^{\act}_{/P}$. The functor
$\mathcal{O}^{\act}_{/P} \to \mathcal{O}^{\act}_{/P'}$ assigned to
a morphism $P \to P'$ takes a pair
$(O, f(O) \actto P)$ to $(O', f(O') \actto P')$ where $f(O) \intto
f(O') \actto P'$ is the inert--active factorization of $f(O) \actto P \to
P'$. \qed  
\end{cor}

\begin{remark}\label{rmk:actcocart}
  Let $\mathcal{O}$ be an algebraic pattern, and write
  $\Fun(\Delta^{1}, \mathcal{O})_{\act}$ for the full subcategory of
  $\Fun(\Delta^{1}, \mathcal{O})$ spanned by the active morphisms.
  As a simple special case of the previous result (taking $f$ to be
  $\id_{\mathcal{O}}$) we see that
  \[ \name{ev}_{1} \colon \Fun(\Delta^{1}, \mathcal{O})_{\act} \to
    \mathcal{O}\]
  is a cocartesian fibration. This corresponds to a functor $\mathcal{O}
  \to \CatI$ that takes $O$ to $\mathcal{O}^{\act}_{/O}$ and a
  morphism $\phi \colon O \to O'$ to a functor
  $\mathcal{O}^{\act}_{/O} \to \mathcal{O}^{\act}_{/O'}$ that takes $X
  \actto O$ to $X' \actto O'$, where $X \intto X' \actto O'$ is the
  inert--active factorization of the composite $X \actto O \to O'$.
\end{remark}

\begin{remark}
Suppose $f \colon \mathcal{O} \to \mathcal{P}$ has unique lifting of
inert morphisms, and let $\mathcal{X}^{\xint} \to \mathcal{P}^{\xint}$
be the pullback of the cocartesian fibration $\mathcal{X} \to
\mathcal{P}$ of Proposition~\ref{propn:stronginertlift} to the subcategory $\mathcal{P}^{\xint}$. Then for every active morphism $\phi \colon f(O) \actto P$
in $\mathcal{P}$ we can define a functor
$\mathcal{P}^{\xint}_{P/} \to \mathcal{O}^{\xint}_{O/}$ as the
composite
\[ \mathcal{P}^{\xint}_{P/} \to \mathcal{X}^{\xint}_{(O,\phi)/} \to
  \mathcal{O}^{\xint}_{O/}\] where the first functor takes
$\alpha \colon P \intto P'$ to the cocartesian morphism
$(O,\phi) \to (\alpha_{!}O,\alpha_{!}\phi)$ for the cocartesian
fibration $\mathcal{X}$ (where $\alpha_{!}\phi$ is the active part of
the map $\alpha \circ \phi$), and the second is induced by the
forgetful functor $\mathcal{X} \to \mathcal{O}$. In particular, we can
restrict to $\mathcal{P}^{\el}_{P/}$ and compose with the functor
$\mathcal{O}^{\xint,\op}_{O/} \to \mathcal{O}^{\op}
\xto{\mathcal{O}^{\el}_{\blank/}} \CatI$ to get a functor
$\mathcal{P}^{\el,\op}_{P/} \intto \CatI$ that takes
$\alpha \colon P \to E$ to $\mathcal{O}^{\el}_{\alpha_{!}O/}$. We
write $\mathcal{O}^{\el}(\phi) \to \mathcal{P}^{\el}_{P/}$ for the
corresponding cartesian fibration.
\end{remark}

Using this functoriality we can now state the conditions we require of
a morphism of algebraic patterns:
\begin{defn}\label{defn:ext}
A morphism of algebraic patterns $f \colon \mathcal{O} \to
\mathcal{P}$ is \emph{extendable} if the following
conditions are satisfied:
\begin{enumerate}[(1)]
\item The morphism $f$ has unique lifting of inert morphisms.
\item For $P \in \mathcal{P}$, let $\mathcal{L}_{P}$ denote the 
limit of the composite functor
$\epsilon_{P} \colon \mathcal{P}^{\el}_{P/} \to \mathcal{P}^{\xint} \to \CatI$ taking $E$
to $\mathcal{O}^{\act}_{/E}$ (where the second functor is that of
Corollary~\ref{cor:Oactftr}).
Then the canonical functor
\[\mathcal{O}^{\act}_{/P} \to \mathcal{L}_{P}\]
is cofinal.
\item For every active morphism $\phi \colon f(O) \actto P$, the
canonical functor \[\mathcal{O}^{\el}(\phi) \to
\mathcal{O}^{\el}_{O/}\]
induces an equivalence
\[ \lim_{\mathcal{O}^{\el}_{O/}}F \isoto
\lim_{\mathcal{O}^{\el}(\phi)} F\]
for every functor $F \colon \mathcal{O}^{\el} \to \mathcal{S}$.
\end{enumerate}
\end{defn}

\begin{remark}
We have used the limit in condition (2) as this seems the most
natural choice in Definition~\ref{defn:limdistrcolim}; we could also
have used the lax limit instead, provided the same change is made
in Definition~\ref{defn:limdistrcolim}. In the cases of interest the
lax limit actually agrees with the usual limit, as it will either be
a finite product or a limit of $\infty$-groupoids, so the distinction turns
out not to matter in practice.
\end{remark}

\begin{remark}
In practice, condition (3) holds because the map
$\mathcal{O}^{\el}(\phi) \to \mathcal{O}^{\el}_{O/}$ is coinitial.
\end{remark}

\begin{remark}\label{rem genSegalcond}
Condition (3) implies that for a functor $\Phi \colon
\mathcal{O}^{\el} \to \mathcal{C}$, we have an equivalence
\[
\lim_{E \in
\mathcal{O}^{\el}_{O/}} \Phi(E) \simeq \lim_{\alpha \in
\mathcal{P}^{\el}_{P/}} \lim_{E \in
\mathcal{O}^{\el}_{\alpha_{!}O/}} \Phi(E)\]
whenever either limit exists in $\mathcal{C}$.
If $\Phi$ is a Segal $\mathcal{O}$-object, this implies that the
following ``relative Segal condition'' holds:
\[ \Phi(O) \simeq \lim_{\alpha \in
\mathcal{P}^{\el}_{P/}} \Phi(\alpha_{!}O).\]
\end{remark}

We now turn to the requirements we must make of our target category,
for which we need the following notion:
\begin{defn}\label{defn:limdistrcolim}
Consider a functor $K \colon \mathcal{I} \to \CatI$ with
corresponding cocartesian fibration
$\pi \colon \mathcal{K} \to \mathcal{I}$.  Let $\mathcal{L}$ be the
limit of $K$, which we can identify with the \icat{} of cocartesian sections
$\Fun_{\mathcal{I}}^{\name{cocart}}(\mathcal{I}, \mathcal{K})$. We then have a
functor $p \colon \mathcal{I} \times \mathcal{L} \to \mathcal{K}$
adjoint to the forgetful functor
$\Fun^{\name{cocart}}_{\mathcal{I}}(\mathcal{I}, \mathcal{K}) \to \Fun(\mathcal{I},
\mathcal{K})$;
the composite $\pi \circ p$ is moreover the projection
$\mathcal{L} \times \mathcal{I} \to \mathcal{I}$. This gives a
commutative diagram
\[
\begin{tikzcd}
\mathcal{L} \times \mathcal{I} \arrow{r}{p} \arrow[d, swap,
"\text{pr}_{1}"]  \arrow[rr,bend left, "\text{pr}_{2}"] & \mathcal{K}
\arrow{r}{\pi} & \mathcal{I} \arrow{d}{\iota} \\
\mathcal{L} \arrow{rr}{\lambda} & & *,
\end{tikzcd}
\]
which for any \icat{} $\mathcal{C}$ (with appropriate limits and
colimits)
determines an equivalence of functors between functor \icats{}
\[ p^{*}\pi^{*}\iota^{*} \simeq \text{pr}_{1}^{*} \lambda^{*}.\]
This induces a mate transformation
\[ \lambda^{*}\iota_{*} \to \text{pr}_{1,*}\text{pr}_{2}^{*} \simeq
  \text{pr}_{1,*} p^{*}\pi^{*},\]
and this is an equivalence: for $\Phi \colon \mathcal{I} \to
\mathcal{C}$, $\lambda^{*}\iota_{*}\Phi$ is the constant functor with
value $\lim_{\mathcal{I}} \Phi$ while the right Kan extension
$\text{pr}_{1,*}$ takes limits over $\mathcal{I}$ fibrewise so that 
$\text{pr}_{1,*}\text{pr}_{2}^{*}\Phi$ is also the constant functor
with value $\lim_{\mathcal{I}}\Phi$. From this equivalence we in turn obtain, by moving adjoints around, a natural
transformation
\[ \lambda_{!} \text{pr}_{1,*} p^{*} \to \lambda_{!} \text{pr}_{1,*}
  p^{*}\pi^{*}\pi_{!} \simeq \lambda_{!}\lambda^{*}\iota_{*}\pi_{!}
  \to \iota_{*} \pi_{!}.\]
For a functor $F \colon \mathcal{K} \to \mathcal{C}$ we can
interpret this as a natural morphism
\[ \colim_{\mathcal{L}} \lim_{\mathcal{I}} p^{*}F \to
\lim_{i \in \mathcal{I}} \colim_{\mathcal{K}_{i}}
F|_{\mathcal{K}_{i}}.\]
We say that \emph{$\mathcal{I}$-limits distribute over $K$-colimits}
in $\mathcal{C}$
if this morphism is an equivalence for any functor $F$.
\end{defn}

\begin{defn}\label{defn:adm}
Let $f \colon \mathcal{O} \to \mathcal{P}$ be an extendable morphism
of algebraic patterns.  We say that an \icat{} $\mathcal{C}$ is
\emph{$f$-admissible} if $\mathcal{C}$ is $\mathcal{O}$- and
$\mathcal{P}$-complete, the pointwise left Kan extension
$f_{!}  \colon \Fun(\mathcal{O},\mathcal{C}) \to
\Fun(\mathcal{P},\mathcal{C})$ exists, and
$\mathcal{P}^{\el}_{P/}$-limits distribute
over $\epsilon_{P}$-colimits for all $P \in \mathcal{P}$, where
$\epsilon_{P}$ is the functor from Definition~\ref{defn:ext}(2). In
other words, if $\mathcal{C}$ is $f$-admissible then for every
$P \in \mathcal{P}$ and every functor $\Phi$, the natural map
\[ \colim_{(O_{E})_{E \in \mathcal{P}^{\el}_{P/}} \in \mathcal{L}_{P}}  
\lim_{E \in \mathcal{P}^{\el}_{P/}} \Phi(O_{E}) \to \lim_{E \in
\mathcal{P}^{\el}_{P/}} \colim_{O_{E} \in \mathcal{O}^{\act}_{/E}} \Phi(O_{E})\]
is an equivalence.
\end{defn}

Having made these definitions, we can now state our result on left Kan
extensions:
\begin{propn}\label{propn:LKESeg}
Suppose $f \colon \mathcal{O} \to \mathcal{P}$ is an extendable
morphism of algebraic patterns, and $\mathcal{C}$ is an
$f$-admissible \icat{}. Then left Kan extension along $f$ restricts
to a functor
\[ f_{!} \colon \Seg_{\mathcal{O}}(\mathcal{C}) \to
\Seg_{\mathcal{P}}(\mathcal{C}),\]
given by $f_{!}\Phi(P) \simeq \colim_{O \in \mathcal{O}^{\act}_{/P}}
\Phi(O)$.
\end{propn}
\begin{proof}
Given $\Phi \in \Seg_{\mathcal{O}}(\mathcal{C})$, we must show that
$f_{!}\Phi$ is a Segal object, \ie{} that the natural map
\[ (f_{!}\Phi)(P) \to \lim_{E \in \mathcal{P}^{\el}_{P/}} (f_{!}\Phi)(E) \]
is an equivalence. We have a sequence of equivalences
\[
\begin{array}{r@{\ }>{\displaystyle}ll}
f_{!}\Phi(P) & \simeq \colim_{O \in \mathcal{O}_{/P}} \Phi(O) \\
& \simeq \colim_{O \in \mathcal{O}^{\act}_{/P}} \Phi(O) &
\text{(by \ref{lem:uniqintcoinit})}\\
& \simeq \colim_{O \in \mathcal{O}^{\act}_{/P}} \lim_{E \in \mathcal{P}^{\el}_{P/}}\Phi(O_{E}) &
\text{(by \ref{defn:ext}(3))}\\
& \simeq \colim_{(O_{E})_{E} \in \mathcal{L}_{p}} \lim_{E \in \mathcal{P}^{\el}_{P/}}\Phi(O_{E}) &
\text{(by \ref{defn:ext}(2))}\\
& \simeq \lim_{E \in \mathcal{P}^{\el}_{P/}} \colim_{O_{E} \in
\mathcal{O}^{\act}_{/E}} \Phi(O_{E}) &
\text{(by \ref{defn:adm})} \\
& \simeq \lim_{E \in \mathcal{P}^{\el}_{P/}} (f_{!}\Phi)(E), & 
\end{array}
\]
which completes the proof.
\end{proof}

Having identified conditions under which $f_{!}$ preserves Segal
objects, we now turn to the question of when these conditions
hold. For extendability, we will see some general classes of examples
below in \S\ref{sec:WSF}; here, we will discuss two classes of
examples where $f$-admissibility holds. The starting point is the
following examples of distributivity of limits over colimits:

\begin{defn}
  We say an \icat{} $\mathcal{C}$ is $\times$-admissible if it has
  finite products and the cartesian product preserves colimits in each variable.
\end{defn}

\begin{lemma}
  Suppose $\mathcal{C}$ is $\times$-admissible. Then finite products
  distribute over all colimits in $\mathcal{C}$.
\end{lemma}
\begin{proof}
  For any functors $F_{i} \colon \mathcal{I}_{i} \to \mathcal{C}$
  ($i = 1,\ldots,n$) whose colimits exist we have
\[ \colim_{\mathcal{I}_{1} \times \cdots \times \mathcal{I}_{n}} F_{1}
\times \cdots \times F_{n} \simeq \colim_{\mathcal{I}_{1}} \cdots
\colim_{\mathcal{I}_{n}} F_{1} \times \cdots \times F_{n} \simeq
\colim_{\mathcal{I}_{1}} F_{1} \times \cdots \times
\colim_{\mathcal{I}_{n}} F_{n}.\qedhere\]
\end{proof}

\begin{propn}\label{propn:spccolimdist}
Let $\mathcal{C}$ be a presentable \icat{} and write
$t \colon \mathcal{S} \to \mathcal{C}$ for the unique
colimit-preserving functor taking $*$ to the terminal object $*_{\mathcal{C}}$ of
$\mathcal{C}$. Consider a functor
$K \colon \mathcal{I} \to \mathcal{S}$ and suppose the following
conditions hold:
\begin{enumerate}[(1)]
\item $t$ preserves $\mathcal{I}$-limits.
\item The functor $\mathcal{C}_{/t(S)} \to \lim_{S}
\mathcal{C} \simeq \Fun(S, \mathcal{C})$ induced by taking pullbacks
along $*_{\mathcal{C}} \simeq t(*) \to t(S)$, is an equivalence for $S = \lim_{\mathcal{I}} K(i)$ and $S =
K_{i}$ for all $i \in \mathcal{I}$.
\end{enumerate}
Then $\mathcal{I}$-limits distribute over $K$-colimits in $\mathcal{C}$.
\end{propn}
\begin{proof}
Condition (2) implies that we have a commutative diagram of right
adjoints
\[
\begin{tikzcd}
\mathcal{C}_{/t(S)} \arrow{rr}{\sim} & &  \Fun(S,
\mathcal{C})\\
& \mathcal{C} \arrow{ul}{\blank \times t(S)} \arrow{ur}[below right]{\text{const}}
\end{tikzcd}
\]
Passing to left adjoints, we get the commutative triangle
\[
\begin{tikzcd}
\Fun(S,\mathcal{C}) \arrow{rr}{\sim} \arrow{dr}[below left]{\text{colim}} & &
\mathcal{C}_{/t(S)} \arrow{dl}{\text{src}}\\
& \mathcal{C},
\end{tikzcd}
\]
from which we see that under the equivalence of (2) the
colimit of a diagram $S \to \mathcal{C}$ is given by the source of
the corresponding morphism to $t(S)$. Given $F \colon S \to
\mathcal{C}$, it follows that we have pullback squares
\[
\begin{tikzcd}
F(s) \arrow{r} \arrow{d} & \colim_{S} F \arrow{d} \\
*_{\mathcal{C}} \arrow{r}{t(s)} & t(S)
\end{tikzcd}
\]
for $s \in S$.

Now consider a functor $F \colon \mathcal{K} \to \mathcal{C}$, where
$\mathcal{K} \to \mathcal{I}$ is the left fibration corresponding to
$K$. We have a commutative square
\[
\begin{tikzcd}
\colim_{L} \lim_{\mathcal{I}} F \arrow{r} \arrow{d} &
\lim_{i \in \mathcal{I}} \colim_{K_{i}} F \arrow{d} \\
\colim_{L} \lim_{\mathcal{I}} *_{\mathcal{C}} \arrow{r} & \lim_{i \in \mathcal{I}}
\colim_{K_{i}} *_{\mathcal{C}},
\end{tikzcd}
\]
where $L := \lim_{\mathcal{I}} K(i)$.
Here the bottom horizontal map can be identified with the natural map
\[ t(L) \simeq \colim_{L} *_{\mathcal{C}} \to \lim_{i \in \mathcal{I}}
\colim_{K_{i}} *_{\mathcal{C}} \simeq \lim_{i \in \mathcal{I}} t(K_{i}).\]
This is an equivalence by assumption (1). The equivalence of
assumption (2) then implies that the top horizontal map is an
equivalence \IFF{} it induces an equivalence on fibres over each map
$t(l) \colon *_{\mathcal{C}} \to t(L)$ for $l \in L$. Using the pullback squares
above and the fact that limits commute, we see that the map on fibres
at $(k_{i})_{i} \in L$ is the identity 
\[ \lim_{\mathcal{I}} F(k_{i}) \to \lim_{\mathcal{I}} F(k_{i}). \qedhere\]
\end{proof}

This argument applies to $\mathcal{C}$ being $\mathcal{S}$, or more
generally any $\infty$-topos, giving:
\begin{cor}\label{cor:Sadm}
Given any functor $K \colon \mathcal{I} \to \mathcal{S}$ we have
that:
\begin{enumerate}[(i)]
\item $\mathcal{I}$-limits distribute over $K$-colimits in
$\mathcal{S}$,
\item $\mathcal{I}$-limits distribute over $K$-colimits in any
$\infty$-topos provided $\mathcal{I}$ is a finite \icat{}.
\end{enumerate}
\end{cor}
\begin{proof}
  Condition (2) of Proposition~\ref{propn:spccolimdist} holds in $\infty$-topoi by descent, \cite[Theorem
  6.1.3.9]{ht}, while condition (1) holds for finite limits since $t$
  is the left adjoint of a geometric morphism by \cite[Proposition
  6.3.4.1]{ht} and so preserves finite limits. In the case of
  $\mathcal{S}$ the finiteness condition is unnecessary since $t$ is
  an equivalence and so preserves all limits.
\end{proof}

\begin{cor}\label{cor:lccadm}
  Let $f \colon \mathcal{O} \to \mathcal{P}$ be an extendable morphism
  of algebraic patterns such that $\mathcal{P}^{\el}_{P/}$ is a finite
  set for all $P \in \mathcal{P}$. Suppose $\mathcal{C}$ is a $\times$-admissible \icat{},
  and assume the pointwise left Kan extension
  $f_{!} \colon \Fun(\mathcal{O}, \mathcal{C}) \to \Fun(\mathcal{P},
  \mathcal{C})$ exists. Then $\mathcal{C}$ is
  $f$-admissible, and the left Kan extension along $f$ restricts to a
  functor
  \[ \pushQED{\qed}
    f_{!} \colon \Seg_{\mathcal{O}}(\mathcal{C}) \to
    \Seg_{\mathcal{P}}(\mathcal{C}).\qedhere
  \popQED\]
\end{cor}

\begin{remark}
  The assumption of $\times$-admissibility can be slightly weakened:
  It is enough to assume that the cartesian product in $\mathcal{C}$
  preserves colimits of shape $\mathcal{O}^{\act}_{/E}$ in each
  variable for all $E \in \mathcal{P}^{\el}$.
\end{remark}

\begin{cor}\label{cor:adm}
  Suppose $\mathcal{X}$ is an $\infty$-topos, and
  $f \colon \mathcal{O} \to \mathcal{P}$ is an extendable morphism of
  algebraic patterns such that
  \begin{enumerate}[(1)]
  \item $\mathcal{O}^{\act}_{/E}$ is an \igpd{} for all
    $E \in \mathcal{P}^{\el}$,
  \item the \icat{} $\mathcal{P}^{\el}_{P/}$ is finite for all
    $P \in \mathcal{P}$ (or arbitrary if $\mathcal{X}$ is the
    $\infty$-topos $\mathcal{S}$).
  \end{enumerate}
  Then $\mathcal{X}$ is $f$-admissible, and the
  left Kan extension restricts to a functor
  \[ \pushQED{\qed}
    f_{!} \colon \Seg_{\mathcal{O}}(\mathcal{X}) \to
    \Seg_{\mathcal{P}}(\mathcal{X}).\qedhere
  \popQED\]
\end{cor}

\section{Free Segal Objects}\label{sec:free}
Suppose $\mathcal{O}$ is an algebraic pattern, and $\mathcal{C}$ an
$\mathcal{O}$-complete \icat{}. Restricting Segal objects to the subcategory
$\mathcal{O}^{\el}$ gives a functor
\[U_{\mathcal{O}}\colon \Seg_{\mathcal{O}}(\mathcal{C}) \to
\Fun(\mathcal{O}^{\el},\mathcal{C}).\]
We think of \emph{free} Segal $\mathcal{O}$-objects as being given by
a left adjoint $F_{\mathcal{O}}$ to this functor, when this exists.

The subcategory $\mathcal{O}^{\xint}$ has a canonical pattern
structure restricted from $\mathcal{O}$ (so only equivalences are
active morphisms and the elementary objects are still those of
$\mathcal{O}^{\el}$), and using this the inclusion
$j_{\mathcal{O}} \colon \mathcal{O}^{\xint} \to \mathcal{O}$ is a
Segal morphism. The \icat{} $\Seg_{\mathcal{O}^{\xint}}(\mathcal{C})$
is by definition the full subcategory of
$\Fun(\mathcal{O}^{\xint}, \mathcal{C})$ spanned by the functors that are
right Kan extensions along the fully faithful inclusion
$i_{\mathcal{O}} \colon \mathcal{O}^{\el} \to \mathcal{O}^{\xint}$,
which means that the restriction functor
$\Seg_{\mathcal{O}^{\xint}}(\mathcal{C}) \to \Fun(\mathcal{O}^{\el},
\mathcal{C})$ is an equivalence. 
The functor
$U_{\mathcal{O}}$ thus factors as the composite
\[ \Seg_{\mathcal{O}}(\mathcal{C}) \xto{j_{\mathcal{O}}^{*}}
  \Seg_{\mathcal{O}^{\xint}}(\mathcal{C}) \xto{i_{\mathcal{O}}^{*}}
  \Fun(\mathcal{O}^{\el}, \mathcal{C}),\] where the second functor is
an equivalence with inverse the right Kan extension functor
$i_{\mathcal{O},*}$. If $\mathcal{C}$ is presentable, using
Proposition~\ref{propn:f!adj} this means the left adjoint
$F_{\mathcal{O}}$ is given by
\[ \Fun(\mathcal{O}^{\el}, \mathcal{C}) \xto{i_{\mathcal{O},*}}
  \Seg_{\mathcal{O}^{\xint}}(\mathcal{C}) \xto{L_{\Seg}f_{!}}
  \Seg_{\mathcal{O}}(\mathcal{C}).\] In this section we will first
show that this adjunction is monadic, and then specialize the results
of the previous section to $j_{\mathcal{O}}$ to get conditions under
which the free Segal objects are described by a formula in terms of
limits and colimits.

Monadicity is a special case of the following observation:
\begin{propn}\label{propn:esssurjSeg}
Suppose $f \colon \mathcal{O} \to \mathcal{P}$ is an essentially
surjective Segal morphism and $\mathcal{C}$ is a presentable
\icat{}. Then:
\begin{enumerate}[(i)]
\item A functor $F \colon \mathcal{P} \to \mathcal{C}$ is a Segal
object \IFF{} $f^{*}F$ is a Segal $\mathcal{O}$-object.
\item The adjunction
\[ Lf_{!} \colon \Seg_{\mathcal{O}}(\mathcal{C}) \rightleftarrows
\Seg_{\mathcal{P}}(\mathcal{C}) \cocolon f^{*} \]
is monadic.
\end{enumerate}
\end{propn}
\begin{proof}
We first prove (i). One direction amounts to $f$ being a Segal
morphism, which is true by assumption. To prove the non-trivial direction, observe that for
$\Phi \colon \mathcal{O} \to \mathcal{C}$ we have for every
$O \in \mathcal{O}$ canonical morphisms
\[ \Phi(f(O)) \to \lim_{E \in \mathcal{P}^{\el}_{f(O)/}} \Phi(E) \to \lim_{E'
\in \mathcal{O}^{\el}_{O/}} \Phi(f(E')).\]
Here the second morphism is an equivalence since $f$ is a Segal
morphism, and if $f^{*}\Phi$ is a Segal $\mathcal{O}$-object then the
composite morphism is an equivalence. Thus the first morphism is an
equivalence, and so $\Phi$ satisfies the Segal condition at every object
of $\mathcal{P}$ in the image of $f$; since $f$ is essentially
surjective this completes the proof.

Using the monadicity theorem for \icats{} \cite[Theorem
4.7.3.5]{ha}, to prove (ii) it suffices to show that $f^{*}$ detects
equivalences, that $\Seg_{\mathcal{P}}(\mathcal{C})$ has colimits of
$f^{*}$-split simplicial objects, and these colimits are preserved
by $f^{*}$. Since $f$ is essentially surjective it is immediate that
$f^{*}$ detects equivalences. Consider an $f^{*}$-split simplicial
object $p \colon \Dop \to \Seg_{\mathcal{P}}(\mathcal{C})$. Let
$\overline{p} \colon (\Dop)^{\triangleright} \to \Fun(\mathcal{P},
\mathcal{C})$
denote the colimit of $p$ in $\Fun(\mathcal{P}, \mathcal{C})$. Since
$f^{*}$, viewed as a functor
$\Fun(\mathcal{P}, \mathcal{C}) \to \Fun(\mathcal{O}, \mathcal{C})$,
is a left adjoint, $f^{*}\overline{p}$ is the colimit of $f^{*}p$ in
$\Fun(\mathcal{O}, \mathcal{C})$. On the other hand, since $f^{*}p$
extends to a split simplicial diagram, and all functors preserve
colimits of split simplicial diagrams, we see that the colimit of
$f^{*}p$ in $\Seg_{\mathcal{O}}(\mathcal{C})$ is also the colimit in
$\Fun(\mathcal{O}, \mathcal{C})$. In particular,
$f^{*}\overline{p}(\infty)$ lies in $\Seg_{\mathcal{O}}(\mathcal{C})$. By
(i) this implies that $\overline{p}(\infty)$ is in
$\Seg_{\mathcal{P}}(\mathcal{C})$. This completes the proof, since
the colimit of $p$ in $\Seg_{\mathcal{P}}(\mathcal{C})$ is the
localization of $\overline{p}(\infty)$, which is already local.
\end{proof}

Applying this to $j_{\mathcal{O}}$, we get:
\begin{cor}\label{cor:freemonadic}
  Let $\mathcal{O}$ be an algebraic pattern and $\mathcal{C}$ a
  presentable \icat{}. Then the free-forgetful adjunction
  \[ F_{\mathcal{O}} \colon \Fun(\mathcal{O}^{\el}, \mathcal{C}) \simeq
    \Seg_{\mathcal{O}^{\xint}}(\mathcal{C}) \rightleftarrows
    \Seg_{\mathcal{O}}(\mathcal{C}) \cocolon U_{\mathcal{O}} \] is
  monadic. \qed
\end{cor}

Now we apply the results of the previous section to $j_{\mathcal{O}}$
to understand when the free algebras are simply given by the left Kan
extension $j_{\mathcal{O},!}$. It is convenient to first introduce
some notation:
\begin{notation}
  Let $\mathcal{O}$ be an algebraic pattern. For $O \in \mathcal{O}$
  we write $\name{Act}_{\mathcal{O}}(O)$ for the $\infty$-groupoid of
  active morphisms to $O$ in $\mathcal{O}$; this is equivalent to
  $(\mathcal{O}^{\xint})^{\act}_{/O}$ since the only active morphisms
  in $\mathcal{O}^{\xint}$ are the equivalences.
\end{notation}

\begin{remark}
  By Remark~\ref{rmk:actcocart} the \icats{} $\mathcal{O}^{\act}_{/O}$
  are functorial in $O \in \mathcal{O}$. Passing to the underlying
  \igpds{} this means the $\infty$-groupoids $\Act_{\mathcal{O}}(O)$ are functorial in $O \in
  \mathcal{O}$, via the factorization system.
\end{remark}

\begin{defn}\label{defn:pattext}
  We say an algebraic pattern $\mathcal{O}$ is \emph{extendable} if
  the inclusion
  $j_{\mathcal{O}}\colon \mathcal{O}^{\xint} \to \mathcal{O}$ is
  extendable in the sense of Definition~\ref{defn:ext}. This is
  equivalent to the following pair of conditions:
  \begin{enumerate}[(1)]
  \item The morphism
    \[ \Act_{\mathcal{O}}(O) \to \lim_{E \in \mathcal{O}^{\el}_{O/}}
      \Act_{\mathcal{O}}(E) \]
    is an equivalence for all $O \in \mathcal{O}$. In other words,
    $\Act_{\mathcal{O}}$ is a Segal $\mathcal{O}$-space.
  \item For every active map $O \xto{\phi} O'$ in $\mathcal{O}$, the
    canonical functor $\mathcal{O}^{\el}(\phi) \to
    \mathcal{O}^{\el}_{O/}$ induces an equivalence on limits
    \[ \lim_{\mathcal{O}^{\el}_{O/}} F \to
      \lim_{\mathcal{O}^{\el}(\phi)} F \]
    for every functor $F \colon \mathcal{O}^{\el} \to \mathcal{S}$.
  \end{enumerate}
\end{defn}

\begin{remark}
Condition (2) implies that
\[ \lim_{E \in \mathcal{O}^{\el}_{O/}} \Phi(E) \to \lim_{\alpha \in
\mathcal{O}^{\el}_{O'/}} \lim_{E \in
\mathcal{O}^{\el}_{\alpha_{!}O/}} \Phi(E)\]
is an equivalence for any functor $\Phi \colon \mathcal{O}^{\el} \to
\mathcal{C}$, provided either limit exists, and $O \to \alpha_{!}O \to E$ is the
inert--active factorization of $O \to O' \xto{\alpha} E$. This in particular
implies the following ``generalized Segal condition'': If $\Phi$ is
a Segal object, then for any active morphism $\phi \colon O \to O'$,
we have
\[\Phi(O) \simeq
\lim_{\alpha \in \mathcal{O}^{\el}_{O'/}} \Phi(\alpha_{!}O).\]
\end{remark}

\begin{remark}
In practice, condition (2) holds because the map
$\mathcal{O}^{\el}(\phi) \to \mathcal{O}^{\el}_{O/}$ is
coinitial. However, with the more general formulation we get the
following characterization of the extendable patterns:
\end{remark}

\begin{propn}\label{prop:extendable}
The following are equivalent for an algebraic pattern $\mathcal{O}$:
\begin{enumerate}[(1)]
\item $\mathcal{O}$ is extendable.
\item $j_{\mathcal{O},!} \colon \Fun(\mathcal{O}^{\xint},
\mathcal{S}) \to \Fun(\mathcal{O}, \mathcal{S})$ restricts to a
functor $\Seg_{\mathcal{O}^{\xint}}(\mathcal{S}) \to \Seg_{\mathcal{O}}(\mathcal{S})$.
\end{enumerate}
\end{propn}
\begin{proof}
  Suppose (1) holds. Since $\Act_{\mathcal{O}}(O)$ is an
  $\infty$-groupoid for all $O$,
  the \icat{} $\mathcal{S}$ is $j_{\mathcal{O}}$-admissible by
  Corollary~\ref{cor:adm}, and so (2) follows from
  Proposition~\ref{propn:LKESeg}.

  We now show that
  (2) implies the two conditions in Definition~\ref{defn:pattext}. To
  prove condition (1), consider the terminal object $* \in
  \Fun(\mathcal{O}^{\xint},\mathcal{S})$. For this we have
  \[ j_{\mathcal{O},!}*(O) \simeq \colim_{\Act_{\mathcal{O}}(O)} *
    \simeq \Act_{\mathcal{O}}(O),\]
  so since $*$ is a Segal $\mathcal{O}^{\xint}$-space, assumption (2)
  implies that $\Act_{\mathcal{O}}(\blank)$ is a Segal
  $\mathcal{O}$-space. To prove condition (2), consider $F \colon
  \mathcal{O}^{\el} \to \mathcal{S}$ and its right Kan extension $F'
  := i_{\mathcal{O},*}F$, which is a Segal
  $\mathcal{O}^{\xint}$-space. Then $j_{\mathcal{O},!}F'$ is a Segal
  $\mathcal{O}$-space, which means that in the commutative square
  \[
    \begin{tikzcd}
      \colim_{X \in \Act_{\mathcal{O}}(O)} F'(X) \arrow{r} \arrow{d} &
      \lim_{E \in \mathcal{O}^{\el}_{O/}} \colim_{Y \in
        \Act_{\mathcal{O}}(E)} F'(Y) \arrow{d} \\
      \Act_{\mathcal{O}}(O) \arrow{r}{\sim} & \lim_{E \in
        \mathcal{O}^{\el}_{O/}} \Act_{\mathcal{O}}(E),
    \end{tikzcd}
    \]
 the top horizontal morphism is an equivalence. Hence we get an
 equivalence on fibres at each active morphism $(\phi \colon X \to O)
 \in \Act_{\mathcal{O}}(O)$, which we can identify with the natural
 map
 \[ F'(X) \isoto \lim_{\alpha \in \mathcal{O}^{\el}_{O/}}
   F'(\alpha_{!}X).\]
 Using the description of $F'$ as a right Kan extension we get
 \[ \lim_{\mathcal{O}^{\el}_{X/}} F \isoto \lim_{\alpha \in
     \mathcal{O}^{\el}_{O/}} \lim_{\mathcal{O}^{\el}_{\alpha_{!}X/}}F
   \simeq \lim_{\mathcal{O}^{\el}(\phi)} F,\]
 as required. 
\end{proof}

\begin{defn}
We say an \icat{} $\mathcal{C}$ is \emph{$\mathcal{O}$-admissible} if
$\mathcal{O}^{\el}_{O/}$-limits distribute over colimits indexed by
the functor $\mathcal{O}^{\el}_{O/} \to \mathcal{S}$ taking $E$ to
$\Act_{\mathcal{O}}(E)$ for all $O \in \mathcal{O}$.
\end{defn}

From Corollaries~\ref{cor:lccadm} and \ref{cor:adm} we get:
\begin{example}\label{ex:Oadm}
Let $\mathcal{O}$ be an extendable algebraic pattern. Then:
\begin{enumerate}[(i)]
\item $\mathcal{S}$ is $\mathcal{O}$-admissible.
\item Any $\infty$-topos is $\mathcal{O}$-admissible if the \icats{}
$\mathcal{O}^{\el}_{O/}$ are all finite.
\item Any $\times$-admissible \icat{} is $\mathcal{O}$-admissible if
  the \icats{} $\mathcal{O}^{\el}_{O/}$ are all finite sets.
\end{enumerate}
\end{example}

\begin{cor}\label{cor:freealg}
Let $\mathcal{O}$ be an extendable algebraic pattern and
$\mathcal{C}$ an $\mathcal{O}$-admissible \icat{}. Then left Kan
extension along $j_{\mathcal{O}} \colon \mathcal{O}^{\xint} \to \mathcal{O}$
restricts to a functor
\[ j_{\mathcal{O},!} \colon \Seg_{\mathcal{O}^{\xint}}(\mathcal{C}) \to
\Seg_{\mathcal{O}}(\mathcal{C}),\]
left adjoint to the restriction $j_{\mathcal{O}}^{*} \colon
\Seg_{\mathcal{O}}(\mathcal{C}) \to
\Seg_{\mathcal{O}^{\xint}}(\mathcal{C})$. This functor is given by 
\[ j_{\mathcal{O},!}\Phi(O) \simeq \colim_{O' \in
    \Act_{\mathcal{O}}(O)} \Phi(O').\]
\end{cor}

Combining this with the equivalence
$\Seg_{\mathcal{O}^{\xint}}(\mathcal{C}) \simeq
\Fun(\mathcal{O}^{\el},\mathcal{C})$ given by right Kan extension
along $i_{\mathcal{O}}$, we can reformulate this as:
\begin{cor}
Let $\mathcal{O}$ be an extendable algebraic pattern and
$\mathcal{C}$ an $\mathcal{O}$-admissible \icat{}. Then the
restriction
\[ U_{\mathcal{O}}\colon \Seg_{\mathcal{O}}(\mathcal{C}) \to \Fun(\mathcal{O}^{\el},
\mathcal{C})\]
has a left adjoint $F_{\mathcal{O}}$, which is given by
\[ F_{\mathcal{O}}(\Phi)(O) \simeq
  j_{\mathcal{O},!}i_{\mathcal{O},*}\Phi(O) \simeq \colim_{O' \in
    \name{Act}_{\mathcal{O}}(O)} \lim_{E \in \mathcal{O}^{\el}_{O'/}}
  \Phi(E).\]
\end{cor}

We end this section with some examples of extendable patterns:
\begin{example}
The algebraic patterns $\xF_{*}^{\flat}$ and $\xF_{*}^{\natural}$
are extendable. In the former case, we recover the familiar formula for
free commutative monoids:
\[ U_{\xF_{*}^{\flat}}F_{\xF_{*}^{\flat}}(X) \simeq \coprod_{n = 0}^{\infty} X^{\times
n}_{h\Sigma_{n}}.\]
In the latter case, we get
\[ U_{\xF_{*}^{\natural}}F_{\xF_{*}^{\natural}}(X \to Y) \simeq
  \coprod_{n = 0}^{\infty} X^{\times_{Y} n}_{h\Sigma_{n}} \to Y,\]
which describes a free commutative monoid on $X \to Y$ in the slice
over $Y$.
\end{example}

\begin{example}
The algebraic patterns $\simp^{\op,\flat}$ and $\simp^{\op,\natural}$
are extendable. In the former case, we get the expected formula for
free associative monoids:
\[ U_{\simp^{\op,\flat}}F_{\simp^{\op,\flat}}(X) \simeq
\coprod_{n=0}^{\infty} X^{\times n},\]
while in the latter case we get the formula for free \icats{}:
\[ U_{\simp^{\op,\natural}}F_{\simp^{\op,\natural}}\left(
\begin{tikzcd}
{} & X \arrow{dl} \arrow{dr} \\
Y & & Y
\end{tikzcd} \right)
\simeq
\left(
\begin{tikzcd}
{} & \coprod_{n=0}^{\infty} X\times_{Y}\cdots \times_{Y}X \arrow{dl} \arrow{dr} \\
Y & & Y
\end{tikzcd} \right).
\]
\end{example}

\begin{example}
More generally, the algebraic pattern $\bbTheta_{n}^{\op,\natural}$
is extendable for every $n$; the
conditions are checked in \cite{thetan}, giving a formula for free
$(\infty,n)$-categories. (On the other hand, the pattern
$\bbTheta_{n}^{\op,\flat}$ is \emph{not} extendable for $n > 1$.)
\end{example}

\begin{example}
The algebraic pattern $\bbOmega^{\op,\natural}$ is extendable; the
conditions are checked in \cite[\S 5.3]{AnalMnd}, giving a formula
for free $\infty$-operads. (On the other hand, the pattern
$\simp_{\xF}^{\op,\natural}$ is \emph{not} extendable.)
\end{example}

\section{Segal Fibrations and Weak Segal  Fibrations}\label{sec:WSF}
In this section we first consider \emph{Segal fibrations} over an
algebraic pattern, which are the cocartesian fibrations corresponding
to Segal objects in $\CatI$, and then generalize these to the class of
\emph{weak Segal fibrations}; for the pattern $\xF_{*}^{\flat}$, these
objects are respectively symmetric monoidal \icats{} and symmetric
\iopds{} in the sense of \cite{ha}. Our main goal is to show that
extendability can be lifted from a base pattern to morphisms between
(weak) Segal fibrations. Combined with our previous results this
allows us, for example, to reproduce (in the cartesian setting) Lurie's formula for operadic Kan
extensions  along morphisms of symmetric
\iopds{}.

\begin{defn}
  Let $\mathcal{O}$ be an algebraic pattern. A \emph{Segal
    $\mathcal{O}$-fibration} is a cocartesian fibration
  $\mathcal{E} \to \mathcal{O}$ whose corresponding functor
  $\mathcal{O} \to \CatI$ is a Segal $\mathcal{O}$-\icat{}.
\end{defn}

\begin{examples}\ 
  \begin{enumerate}[(i)]
  \item A Segal $\xF_{*}^{\flat}$-fibration is a symmetric monoidal
    \icat{}.
  \item A Segal $\simp^{\op,\flat}$-fibration is a monoidal \icat{},
    and a Segal $\simp^{\op,\natural}$-fibration is a double \icat{}.
  \item Segal $\simp^{n,\op,\flat}$-fibrations and Segal
    $\bbTheta_{n}^{\op,\flat}$-fibrations both describe
    $\mathbb{E}_{n}$-monoidal \icats{}.
  \end{enumerate}
\end{examples}

\begin{defn}\label{defn:SFpattern}
  Suppose $\mathcal{O}$ is an algebraic pattern, and
  $\pi \colon \mathcal{E} \to \mathcal{O}$ is a Segal
  $\mathcal{O}$-fibration. We say a morphism in $\mathcal{E}$ is
  \emph{inert} if it is cocartesian and lies over an inert morphism in
  $\mathcal{O}$, and \emph{active} if it lies over an active morphism
  in $\mathcal{O}$; moreover, we say an object of $\mathcal{E}$ is
  \emph{elementary} if it lies over an elementary object of
  $\mathcal{O}$.
\end{defn}

\begin{lemma}\label{lem:SFSegmor}
  Equipped with this data, $\mathcal{E}$ is an algebraic pattern, and
  $\pi \colon \mathcal{E} \to \mathcal{O}$ is a Segal morphism.
\end{lemma}
\begin{proof}
  The inert and active morphisms form a factorization system by
  \cite[Proposition 2.1.2.5]{ha}, so we have defined an algebraic
  pattern structure on $\mathcal{E}$. To see that $\pi$ is a Segal
  morphism it suffices to show that for $\overline{X} \in \mathcal{E}_{X}$ the induced
  functor
  \[ \mathcal{E}^{\el}_{\overline{X}/} \to \mathcal{O}^{\el}_{X/} \] is
  coinitial. But this functor is an equivalence since for
  each inert morphism $X \to E$ with $E$ elementary there is a unique
  cocartesian morphism with source $\overline{X}$ lying over it.
\end{proof}

We now show that we can lift extendability along Segal fibrations:
\begin{propn}\label{propn:SegFibExt}
Consider a commutative square
\[
\begin{tikzcd}
\mathcal{E} \arrow{r}{F} \arrow[d, swap, "p"] & \mathcal{F} \arrow{d}{q}
\\
\mathcal{O} \arrow{r}{f} & \mathcal{P},
\end{tikzcd}
\]
where $f$ is an extendable morphism
of algebraic patterns,  $p \colon \mathcal{E} \to
\mathcal{O}$ and $q \colon \mathcal{F} \to \mathcal{P}$ are Segal
fibrations, and $F$ preserves cocartesian morphisms. Then $F$ is
extendable. Moreover, if $\mathcal{C}$ is $f$-admissible and either
\begin{enumerate}[(i)]
\item $\mathcal{P}^{\el}_{P/}$-limits distribute over
$\eta$-colimits in $\mathcal{C}$ for all functors $\eta \colon
\mathcal{P}^{\el}_{P/} \to \CatI$ and all $P \in \mathcal{P}$, or
\item $p$ and $q$ are left fibrations, and $\mathcal{P}^{\el}_{/P}$-limits distribute over
$\eta$-colimits in $\mathcal{C}$ for all functors $\eta \colon
\mathcal{P}^{\el}_{/P} \to \mathcal{S}$ and all $P \in \mathcal{P}$,
\end{enumerate}
then $\mathcal{C}$ is $F$-admissible.
\end{propn}
\begin{proof}
  It is immediate from the definitions that $F$ preserves inert and
  active morphisms. We now observe that $F$ has unique lifting of inert morphisms.
  Given $\overline{O} \in \mathcal{E}$ lying over $O \in \mathcal{O}$, and an
  inert morphism $\overline{\epsilon} \colon F(\overline{O}) \intto \overline{P}$ in $\mathcal{F}$,
  lying over $\epsilon \colon f(O) \intto P$ in $\mathcal{P}$, there
  exists a unique inert morphism $\gamma \colon O \intto O'$ such that
  $f(\gamma) \simeq \epsilon$, since $f$ is extendable. Since inert
  morphisms in $\mathcal{E}$ are cocartesian, there exists a unique
  inert morphism $\overline{\gamma} \colon \overline{O} \intto \overline{O}'$ lying over
  $\gamma$. Moreover, as $F$ preserves cocartesian morphisms, the
  morphism $F(\overline{\gamma})$ is the unique inert
  morphism over $\epsilon$ with source $F(\overline{O})$, \ie{}
  $F(\overline{\gamma}) \simeq \overline{\epsilon}$, and since cocartesian
  morphisms are unique, $\overline{\gamma}$ is the unique inert morphism
  that maps to $\overline{\epsilon}$.

  For every active morphism $\overline \phi\colon F(\overline O)\actto \overline P$ lying over $\phi\colon f(O)\actto P$, equivalences of the type
  $\mathcal{E}^{\el}_{\overline{X}/} \simeq \mathcal{O}^{\el}_{X/}$ imply that $\xE^\el(\overline \phi)\to \xE^\el_{\overline O/}$ is equivalent to $\xxO^\el(\phi)\to \xxO^\el_{O/}$, hence condition (3) in
  Definition~\ref{defn:ext} follows immediately from $f$ being
  extendable.
  It remains to prove condition (2). For $\overline{P} \in \mathcal{F}$ lying
over $P \in \mathcal{P}$ and $\overline{\epsilon} \colon \overline{P} \intto \overline{P}'$ an
inert morphism in $\mathcal{F}$ lying over
$\epsilon \colon P \intto P'$, we have a functor
\[ \mathcal{E}^{\act}_{/\overline{P}} \to \mathcal{E}^{\act}_{/\overline{P}'},\]
which fits in a commutative square
\[
\begin{tikzcd}
\mathcal{E}^{\act}_{/\overline{P}} \arrow{r} \arrow{d} &
\mathcal{E}^{\act}_{/\overline{P}'} \arrow{d} \\
\mathcal{O}^{\act}_{/P} \arrow{r} & \mathcal{O}^{\act}_{/P'}.      
\end{tikzcd}
\]
We claim that here the vertical functors are cocartesian fibrations,
and the top horizontal functor preserves cocartesian morphisms.
The functor
\[ \mathcal{E}_{/\overline{P}} := \mathcal{E} \times_{\mathcal{F}}
\mathcal{F}_{/\overline{P}} \to \mathcal{O} \times_{\mathcal{P}}
\mathcal{P}_{/P} =: \mathcal{O}_{/P} \] is a fibre product of
cocartesian fibrations along functors that preserve cocartesian
morphisms, hence it is again a cocartesian fibration. We can write
$\mathcal{E}^{\act}_{/\overline{P}}$ as a pullback
$\mathcal{E}_{/\overline{P}} \times_{\mathcal{O}_{/P}}
\mathcal{O}^{\act}_{/P}$, hence
$\mathcal{E}^{\act}_{/\overline{P}} \to \mathcal{O}^{\act}_{/P}$ is a pullback
of a cocartesian fibration and so is itself cocartesian. Moreover, a
morphism in $\mathcal{E}^{\act}_{/\overline{P}}$ is cocartesian \IFF{} its
image in $\mathcal{E}$ is cocartesian (since the functor
$\mathcal{F}_{/\overline{P}} \to \mathcal{F}$ detects cocartesian morphisms, by
\cite[Proposition 2.4.3.2]{ht}). Since inert morphisms are
cocartesian, this implies that the top horizontal functor preserves cocartesian morphisms
by the 3-for-2 property of cocartesian morphisms (\cite[Proposition
2.4.1.7]{ht}).

For $\overline{P} \in \mathcal{F}$ we therefore have a commutative square
\[
\begin{tikzcd}
\mathcal{E}^{\act}_{/\overline{P}} \arrow{r} \arrow{d} & \lim_{\alpha
\colon P \intto E \in
\mathcal{P}^{\el}_{P/}} \mathcal{E}^{\act}_{/\alpha_{!}\overline{P}}
\arrow{d} \\
\mathcal{O}^{\act}_{/P} \arrow{r} & \lim_{\alpha \colon P \intto E\in
\mathcal{P}^{\el}_{P/}} \mathcal{O}^{\act}_{/E}.
\end{tikzcd}
\]
where the vertical functors are cocartesian fibrations, the top
horizontal functor preserves cocartesian morphisms, and the bottom
horizontal functor is cofinal, since $f$ is extendable. Our goal is
to show that the top horizontal functor is cofinal. Since pullbacks
of cofinal functors along cocartesian fibrations are cofinal by
\cite[Proposition 4.1.2.15]{ht}, it suffices to show that the
square is cartesian, which in this situation is equivalent to the
functor on fibres being an equivalence.

Since the fibration $\mathcal{E}^\act_{/\overline{P}} \to \mathcal{O}^\act_{/P}$
is a fibre product, its fibre at $(O, \psi \colon f(O) \actto P)$ is the
fibre product
$\mathcal{E}_{O} \times_{\mathcal{F}_{f(O)}}
(\mathcal{F}^\act_{/\overline{P}})_{\psi}$; since $\mathcal{F}$ is
cocartesian over $\mathcal{P}$, we can use the cocartesian pushforward
over $\psi$ to identify this with a fibre product
$\mathcal{E}_{O} \times_{\mathcal{F}_{P}}
\mathcal{F}_{P/\overline{P}}$ over the composite functor
$\mathcal{E}_{O} \to \mathcal{F}_{f(O)} \xto{\psi_{!}}
\mathcal{F}_{P}$.

If $\psi$ is active, then as $f$ is extendable and $\xE\to \xxO$ is a Segal fibration
we have an equivalence
\[ \mathcal{E}_{O} \isoto \lim_{\alpha \in \mathcal{P}^{\el}_{P/}}
\mathcal{E}_{\alpha_{!}O}\]
by Remark~\ref{rem genSegalcond}.
Putting this together with the equivalence $\mathcal{F}_{P/\overline{P}} \isoto
\lim_{\alpha \in \mathcal{P}^{\el}_{P/}}
\mathcal{F}_{E/\alpha_{!}\overline{P}}$ (and similarly for $\mathcal{F}_{P}$) we get
\[ (\mathcal{E}^\act_{/\overline{P}})_{(O, \psi)} \isoto \lim_{\alpha \in
\mathcal{P}^{\el}_{P/}}
(\mathcal{E}^\act_{/\alpha_!\overline{P}})_{(E,\psi_{\alpha})},\]
\ie{} the functor  we get on fibres is indeed an equivalence, which
completes the proof that $F$ is extendable.

For admissibility, observe that since $\lim_{\alpha \in
\mathcal{P}^{\el}_{P/}} \mathcal{E}^{\act}_{/\alpha_{!}\overline{P}} \to
\lim_{\alpha \in \mathcal{P}^{\el}_{P/}} \mathcal{O}^{\act}_{/E}$ is
a cocartesian fibration, if we compute the colimit of a functor $\Phi$ over its source in
two stages using the left Kan extension along this functor, we get
\[ \colim_{\lim_{\alpha \in
\mathcal{P}^{\el}_{P/}} \mathcal{E}^{\act}_{/\alpha_{!}\overline{P}}} \Phi
\simeq \colim_{(\omega_{\alpha}) \in \lim_{\alpha \in \mathcal{P}^{\el}_{P/}}
\mathcal{O}^{\act}_{/E}} 
\colim_{\lim_{\alpha \in
\mathcal{P}^{\el}_{P/}}
(\mathcal{E}^{\act}_{/\alpha_{!}\overline{P}})_{\omega_{\alpha}}} \Phi,\]
from which we see that $F$-admissibility follows from
$f$-admissibility plus either (i) or (ii).
\end{proof}

\begin{defn}\label{defn:WSF}
  Let $\mathcal{O}$ be an algebraic pattern. A \emph{weak Segal
    $\mathcal{O}$-fibration} is a functor $p \colon \mathcal{E} \to
  \mathcal{O}$ such that:
  \begin{enumerate}
  \item For every object $\overline{X}$ in $\mathcal{E}$ lying over
    $X \in \mathcal{O}$ and every inert morphism $i \colon X \to Y$ in
    $\mathcal{O}$ there exists a $p$-cocartesian morphism
    $\overline{\imath} \colon \overline{X} \to \overline{Y}$ lying over $i$.
  \item For every object $X \in \mathcal{O}$, the functor
    \[ \mathcal{E}_{X} \to \lim_{E \in \mathcal{O}^{\el}_{X/}}
      \mathcal{E}_{E}, \]
    induced by the cocartesian morphisms over inert maps, is an
    equivalence.
  \item Given $\overline{X}$ in $\mathcal{E}_{X}$, choose a cocartesian lift
    $\xi \colon (\mathcal{O}^{\el}_{X/})^{\triangleleft}
    \to \mathcal{E}$ of the diagram of inert morphisms from $X$ in
    $\mathcal{O}$, taking $-\infty$ to $\overline{X}$. Then for any $Y \in
    \mathcal{O}$ and $\overline{Y} \in \mathcal{E}_{Y}$, the
    commutative square
    \[
      \begin{tikzcd}
        \Map_{\mathcal{E}}(\overline{Y}, \overline{X}) \arrow{r} \arrow{d} & \lim_{E
          \in \mathcal{O}^{\el}_{X/}} \Map_{\mathcal{E}}(\overline{Y},
        \xi(E)) \arrow{d} \\
        \Map_{\mathcal{O}}(Y, X) \arrow{r} & \lim_{E
          \in \mathcal{O}^{\el}_{X/}} \Map_{\mathcal{O}}(Y, E)
      \end{tikzcd}
    \]
    is cartesian.
  \end{enumerate}
\end{defn}

\begin{remark}\label{rmk:WSFMapeq} 
  Condition (3) in the definition can be rephrased as: For every map
  $\phi \colon Y \to X$ in $\mathcal{O}$, the natural map
  \[ \Map_{\mathcal{E}}^{\phi}(\overline{Y},\overline{X}) \to \lim_{\alpha
\colon X \intto E \in
\mathcal{O}^{\el}_{X/}}\Map_{\mathcal{E}}^{\alpha\phi}(\overline{Y},
\alpha_{!}\overline{X})\]
is an equivalence, where
$\Map_{\mathcal{E}}^{\phi}(\overline{Y},\overline{X})$ denotes the fibre at
$\phi$ of $\Map_{\mathcal{E}}(\overline{Y}, \overline{X}) \to
\Map_{\mathcal{O}}(Y,X)$. If $\phi$ is active, let $Y
\overset{\alpha_{Y}}{\intto} Y_{\alpha} \overset{\phi_{\alpha}}{\actto} E$
denote the inert--active factorization of $Y \overset{\phi}{\actto} X
\overset{\alpha}{\intto} E$, then combining this equivalence with the cocartesian morphisms $\overline{Y} \intto
\alpha_{Y,!}\overline{Y}$ over $\alpha_{Y}$ we obtain an equivalence
\[ \Map_{\mathcal{E}}^{\phi}(\overline{Y},\overline{X}) \simeq \lim_{\alpha
\colon X \intto E \in \mathcal{O}^{\el}_{X/}}
\Map_{\mathcal{E}}^{\phi_{\alpha}}(\alpha_{Y,!}\overline{Y},
\alpha_{!}\overline{X}).\]
\end{remark}

\begin{examples}\label{ex:weakSegal}\ 
\begin{enumerate}[(i)]
\item A weak Segal $\xF_{*}^{\flat}$-fibration is a symmetric \iopd{},
and a weak Segal $\xF_{*}^{\natural}$-fibration is a generalized \iopd{}, in
the sense of \cite{ha}.
\item A weak Segal $\simp^{\op,\flat}$-fibration is a non-symmetric
\iopd{}, and a weak Segal $\simp^{\op,\natural}$-fibration is a
generalized non-symmetric
\iopd{}, as considered in \cite{enriched}.
\item If $\Phi$ is a perfect operator category and
$\Lambda(\Phi)$ is its Leinster category, then a weak Segal
$\Lambda(\Phi)^{\flat}$-fibration is a $\Phi$-\iopd{}, in the
sense of \cite{bar}, and
weak Segal $\Lambda(\Phi)^{\natural}$-fibrations are the
natural extension of generalized \iopds{} to \emph{generalized
$\Phi$-\iopds{}}.
\item Weak Segal $\bbTheta_{n}^{\op,\natural}$-fibrations can be
  viewed as an \icatl{} analogue of the \emph{$n$-operads}
  of Batanin~\cite{BataninNcat}.
\end{enumerate}
\end{examples}

\begin{defn}\label{defn:WSFpattern}
Suppose $\mathcal{O}$ is an algebraic pattern, and
$\pi \colon \mathcal{E} \to \mathcal{O}$ is a weak Segal
$\mathcal{O}$-fibration. We say a morphism in $\mathcal{E}$ is
\emph{inert} if it is cocartesian and lies over an inert morphism in
$\mathcal{O}$, and \emph{active} if it lies over an active morphism
in $\mathcal{O}$; moreover, we say an object of $\mathcal{E}$ is
\emph{elementary} if it lies over an elementary object of
$\mathcal{O}$.
\end{defn}

\begin{lemma}\label{lem:WSFSegmor}
Equipped with this data, $\mathcal{E}$ is an algebraic pattern, and
$\pi \colon \mathcal{E} \to \mathcal{O}$ is a Segal morphism.
\end{lemma}
\begin{proof}
  As Lemma~\ref{lem:SFSegmor}.
\end{proof}

\begin{remark}
A cocartesian fibration $\mathcal{E} \to \mathcal{O}$ is a Segal
fibration \IFF{} it is a weak Segal fibration.
\end{remark}

\begin{remark}
Suppose $\mathcal{E} \to \mathcal{O}$ and $\mathcal{F} \to
\mathcal{O}$ are weak Segal fibrations. Then a morphism $\mathcal{E}
\to \mathcal{F}$ over $\mathcal{O}$ is a Segal morphism \IFF{} it
preserves inert morphisms.
\end{remark}

\begin{remark}
  Let $\Cat_{\infty/\mathcal{O}}^{\name{WSF}}$ denote the subcategory
  of $\Cat_{\infty/\mathcal{O}}$ whose objects are the weak Segal
  fibrations and whose morphisms are those that preserve inert
  morphisms. This \icat{} is described by a \emph{categorical pattern}
  in the sense of \cite[\S B]{ha}, and so arises from a combinatorial
  model category by \cite[Theorem B.0.20]{ha}. It follows that
  $\Cat_{\infty/\mathcal{O}}^{\name{WSF}}$ is a presentable \icat{}.
\end{remark}

For weak Segal fibrations we can prove a weaker version of
Proposition~\ref{propn:SegFibExt}; for this we need the following
consequence of extendability, which we learned from Roman
Kositsyn:
\begin{lemma}\label{lem:stronglyext}
  Let $\mathcal{O}$ be an extendable pattern. Then
  the functor
  $\mathcal{O} \to \CatI$ taking $O$ to $\mathcal{O}^{\act}_{/O}$ from
  Remark~\ref{rmk:actcocart} is a Segal $\mathcal{O}$-\icat{}. In
  particular, for any active maps $\phi \colon X \actto O$, $\psi
  \colon Y \actto O$ in $\mathcal{O}$, the morphism of
  mapping spaces
  \[ \Map_{\mathcal{O}^{\act}_{/O}}(\phi, \psi) \to \lim_{\alpha
      \colon O \intto E \in
      \mathcal{O}^{\el}_{O/}}
    \Map_{\mathcal{O}^{\act}_{/E}}(\phi_{\alpha}, \psi_{\alpha}) \]
  is an equivalence.
\end{lemma}
\begin{proof}
  We must show that for any $O \in \mathcal{O}$, the functor
  \[ \mathcal{O}^{\act}_{/O} \to \lim_{E \in \mathcal{O}^{\el}_{O/}}
    \mathcal{O}^{\act}_{/E} \] is an equivalence; to see this it
  suffices to check that it is an equivalence on underlying \igpds{}
  and is fully faithful. The map on underlying \igpds{} is the map
  $\Act_{\mathcal{O}}(O) \to \lim_{E \in \mathcal{O}^{\el}_{O/}}
  \Act_{\mathcal{O}}(E)$, which is an equivalence by assumption since
  $\mathcal{O}$ is extendable. Given active maps
  $\phi \colon X \actto O$, $\psi \colon Y \actto O$, the morphism of
  mapping spaces
  \[ \Map_{\mathcal{O}^{\act}_{/O}}(\phi, \psi) \to \lim_{\alpha
      \colon O \intto E \in
      \mathcal{O}^{\el}_{O/}}
    \Map_{\mathcal{O}^{\act}_{/E}}(\phi_{\alpha}, \psi_{\alpha}) \]
 fits in a commutative cube
  \[
    \begin{tikzcd}[column sep=tiny,row sep=small]
      \Map_{\mathcal{O}^{\act}_{/O}}(\phi, \psi)\arrow{rr} \arrow{dd}
      \arrow{dr} & &
      \Act_{\mathcal{O}}(Y) \arrow{dd} \arrow{dr} \\
      & \lim_{\alpha
      \colon \! O \intto E \in
      \mathcal{O}^{\el}_{O/}}
    \Map_{\mathcal{O}^{\act}_{/E}}(\phi_{\alpha}, \psi_{\alpha})
    \arrow[crossing over]{rr}  & & \lim_{\alpha
      \colon \! O \intto E \in
      \mathcal{O}^{\el}_{O/}} \Act_{\mathcal{O}}(\alpha_{!}Y)
    \arrow{dd} \\
      \{\phi\} \arrow{rr} \arrow{dr} & & \Act_{\mathcal{O}}(X)
      \arrow{dr} \\
       & \lim_{\alpha
      \colon \! O \intto E \in
      \mathcal{O}^{\el}_{O/}} \{\phi_{\alpha}\} \arrow[crossing
    over,leftarrow]{uu} \arrow{rr} & & \lim_{\alpha
      \colon \! O \intto E \in
      \mathcal{O}^{\el}_{O/}} \Act_{\mathcal{O}}(\alpha_{!}X),
    \end{tikzcd}
  \]
  where the back and front faces are cartesian. Since $\mathcal{O}$ is
  extendable, we can apply the ``extended Segal condition'' of
  Remark~\ref{rem genSegalcond} to $\Act_{\mathcal{O}}(\blank)$ and
  conclude the horizontal morphisms in the right-hand square are
  equivalences. It follows that the map on fibres in the left square
  is also an equivalence, as required.
\end{proof}

Using this we can prove the following key observation:
\begin{propn}\label{propn:WSFlim}
Suppose $\mathcal{O}$ is an extendable algebraic pattern. Consider a
commutative triangle
\[
\begin{tikzcd}
\mathcal{E} \arrow{rr}{f} \arrow{dr}[below left]{p} & & \mathcal{F}
\arrow{dl}{q} \\
& \mathcal{O},
\end{tikzcd}
\]
where $p$ and $q$ are weak Segal fibrations and $f$ preserves inert
morphisms. Then for any $F \in \mathcal{F}$ the functor
\[ \mathcal{E}^{\act}_{/F} \to \lim_{\alpha \in \mathcal{O}^{\el}_{q(F)/}}
\mathcal{E}^{\act}_{/\alpha_{!}F} \]
is an equivalence.
\end{propn}
\begin{proof}
For any active morphisms
$\phi \colon Y \actto X$, $\psi \colon X \actto q(F)$ in $\mathcal{O}$ and $\alpha\in \xxO^\el_{q(F)/}$ the inert-active factorization gives a commutative diagram
\[
\begin{tikzcd}
Y\arrow[r, rightsquigarrow, "\phi"] \arrow[d, swap, tail, "\alpha_Y"]& X \arrow[d, tail, "\alpha_X"] \arrow[r, rightsquigarrow] & q(F)\arrow[d, tail, "\alpha"]\\
Y_\alpha \arrow[r, rightsquigarrow, "\phi_\alpha"] & X_\alpha \arrow[r, rightsquigarrow] & E.
\end{tikzcd}
\]
By combining Remark~\ref{rem genSegalcond} (the
``generalized Segal condition'') with the argument of
Remark~\ref{rmk:WSFMapeq} we then get an equivalence
\[ \Map^{\phi}_{\mathcal{E}}(\overline{Y}, \overline{X}) \isoto \lim_{\alpha
 \in \mathcal{O}^{\el}_{q(F)/}}
\Map^{\phi_{\alpha}}_{\mathcal{E}}(\alpha_{Y,!}\overline{Y},
\alpha_{X,!}\overline{X}).\]
Thus in the commutative square
\[
  \begin{tikzcd}
    \Map_{\mathcal{E}^{\act}_{/F}}(\overline{Y}, \overline{X}) \arrow{r} \arrow{d} & \lim_{\alpha
 \in \mathcal{O}^{\el}_{q(F)/}} \Map_{\mathcal{E}^{\act}_{/\alpha_{!}F}}(\alpha_{Y,!}\overline{Y},
\alpha_{X,!}\overline{X}) \arrow{d} \\
\Map_{\mathcal{O}^{\act}_{/qF}}(Y, X) \arrow{r} & \lim_{\alpha
 \in \mathcal{O}^{\el}_{q(F)/}}
\Map_{\mathcal{O}^{\act}_{/E}}(Y_\alpha, X_\alpha),
  \end{tikzcd}
\]
the map on fibres is an equivalence for all $\phi \colon Y \actto X$,
which means the square is cartesian. The bottom horizontal morphism is
an equivalence by Lemma~\ref{lem:stronglyext} since $\mathcal{O}$ is  extendable. Hence we see that the functor $\mathcal{E}^{\act}_{/F} \to \lim_{\alpha \in \mathcal{O}^{\el}_{q(F)/}}
\mathcal{E}^{\act}_{/\alpha_{!}F}$ induces equivalences on
mapping spaces, and so is fully faithful.
To see that this functor is also
essentially surjective, consider the commutative square of \igpds{}
\[
\begin{tikzcd}
(\mathcal{E}^{\act}_{/F})^{\simeq}  \arrow{r} \arrow{d} &  \lim_{\alpha \in \mathcal{O}^{\el}_{q(F)/}}
(\mathcal{E}^{\act}_{/\alpha_{!}F})^{\simeq} \arrow{d} \\
(\mathcal{O}^{\act}_{/q(F)})^{\simeq} \arrow{r} & \lim_{\alpha
\in \mathcal{O}^{\el}_{q(F)/}} (\mathcal{O}^{\act}_{/E})^{\simeq};      
\end{tikzcd}
\]
we want to show that the top horizontal morphism is an equivalence.
The bottom horizontal morphism is an equivalence by assumption, since
$\mathcal{O}$ is extendable; it therefore suffices to show the map
on fibres over $\phi \colon O \to q(F)$ is an equivalence. The fibre
$(\mathcal{E}^{\act}_{/F})^{\simeq}_{\phi}$ we can identify with
$\mathcal{E}_{O}^{\simeq} \times_{\mathcal{F}_{O}^{\simeq}}
(\mathcal{F}_{/F})_{\phi}^{\simeq}$. By condition (2) in
Definition~\ref{defn:WSF} we have an equivalence
$\mathcal{E}_{O}^{\simeq} \simeq \lim_{\alpha \in
\mathcal{O}^{\el}_{O/}} \mathcal{E}_{E}^{\simeq}$, and similarly
for $\mathcal{F}$. Moreover, condition (3) implies that in the
commutative square
\[
\begin{tikzcd}
(\mathcal{F}_{/F})^{\simeq}_{\phi} \arrow{r} \arrow{d} &
\lim_{\alpha \in \mathcal{O}^{\el}_{O/}}
(\mathcal{F}_{/\alpha_{!}F})^{\simeq}_{\phi_{\alpha}} \arrow{d} \\
\mathcal{F}_{O}^{\simeq} \arrow{r}{\sim} &   \lim_{\alpha
\in \mathcal{O}^{\el}_{O/}} \mathcal{F}_{E}^{\simeq},
\end{tikzcd}
\]
the map on fibres over each object of
$\mathcal{F}_{O}^{\simeq}$ is an equivalence, hence the
top horizontal morphism is an equivalence. Since limits commute, it
follows that we have an equivalence
\[ (\mathcal{E}^{\act}_{/F})^{\simeq}_{\phi} \to \lim_{\alpha
\in \mathcal{O}^{\el}_{O/}}
(\mathcal{E}^{\act}_{/\alpha_{!}F})^{\simeq}_{\phi_{\alpha}},\]
which completes the proof.
\end{proof}

\begin{cor}\label{cor:WSFext}
Suppose $\mathcal{O}$ is an extendable algebraic pattern. Then any
morphism between weak Segal fibrations over $\mathcal{O}$ that
preserves inert morphisms is
extendable.
\end{cor}

\begin{proof}
Suppose $\mathcal{E}$ and $\mathcal{F}$ are weak Segal fibrations
over $\mathcal{O}$. Then any morphism of algebraic patterns
$f \colon \mathcal{E} \to \mathcal{F}$ over $\mathcal{O}$ has unique
lifting of inert morphisms, as an inert morphism is uniquely
determined by its source and its image in $\mathcal{O}$. Moreover,
$f$ satisfies condition (2) in Definition~\ref{defn:ext} by
Proposition~\ref{propn:WSFlim}, and condition (3) reduces to the
extendability of $\mathcal{O}$.
\end{proof}

\begin{cor}
  Suppose $\mathcal{O}$ is an extendable algebraic pattern, and
  $\mathcal{E} \to \mathcal{O}$ is a weak Segal fibration. Then
  $\mathcal{E}$ is extendable.
\end{cor}
\begin{proof}
  The restriction $\mathcal{E}^{\xint} \to \mathcal{O}$ is also
  a weak Segal fibration, hence we can apply
  Corollary~\ref{cor:WSFext} to the inclusion $\mathcal{E}^{\xint} \to
  \mathcal{E}$.
\end{proof}

\begin{example}\label{ex:symopdext}
  The pattern $\xF_{*}^{\flat}$ is extendable. 
  Our previous
  results therefore specialize to tell us that any morphism
  $f \colon \mathcal{O} \to \mathcal{P}$ of symmetric \iopds{} is
  extendable. If $\mathcal{C}$ is a cocomplete $\times$-admissible
  \icat{}, we conclude that left Kan extension along $f$ restricts to
  a functor
  $f_{!} \colon \Seg_{\mathcal{O}}(\mathcal{C}) \to
  \Seg_{\mathcal{P}}(\mathcal{C})$, given by the formula
  \[ (f_{!}F)(P) \simeq \colim_{O \in \mathcal{O}^{\act}_{/P}} F(O).\]
  Note that this agrees with the formula for operadic left Kan
  extensions from \cite[\S 3.1.2]{ha}, though in our case the target
  must be a cartesian symmetric monoidal \icat{}.
\end{example}

\begin{example}\label{ex:symopdext2}
  Let us spell out the description of free Segal $\mathcal{O}$-objects
  for a symmetric \iopd{} $\mathcal{O} \to \xF_{*}$ in a bit more
  detail. We can identify $\mathcal{O}^{\el}$ with the
  $\infty$-groupoid $\mathcal{O}_{\angled{1}}^{\simeq}$, and for
  $X \in \mathcal{O}_{\angled{1}}$ the space $\Act_{\mathcal{O}}(X)$
  decomposes as $\coprod_{n = 0}^{\infty} \Act_{\mathcal{O}}(X)_{n}$,
  where $\Act_{\mathcal{O}}(X)_{n}$ is the space of morphisms to $X$
  in $\mathcal{O}$ lying over the unique active morphism
  $\angled{n} \to \angled{1}$ in $\xF_{*}$. If $\mathcal{C}$ is a
  cocomplete $\times$-admissible \icat{}, then for $F\in \xFun(\xxO^\simeq_{\angled{1}}, \xcc)$ our formula for the
  free Segal $\mathcal{O}$-object monad $T_{\mathcal{O}}$ gives:
  \[ (T_{\mathcal{O}}F)(X) \simeq \coprod_{n = 0}^{\infty}
    \colim_{(Y_{1},\ldots,Y_{n}) \in \Act_{\mathcal{O}}(X)_{n}}
    F(Y_{1}) \times \cdots \times F(Y_{n}).\]
  If
  $\mathcal{O}_{\angled{1}}^{\simeq}$ is contractible, we
  can identify the space $\mathcal{O}(n)$ of $n$-ary operations with
  the fibre of $\Act_{\mathcal{O}}(X) \to \Act_{\xF_{*}}(\angled{1})
  \simeq B\Sigma_{n}$, and so rewrite this as the familiar formula
  \[ T_{\mathcal{O}}C \simeq \coprod_{n = 0}^{\infty}
    \colim_{B\Sigma_{n}} \colim_{\mathcal{O}(n)} C \times \cdots
    \times C \simeq \coprod_{n = 0}^{\infty}
    \left(\mathcal{O}(n) \times C^{\times n}\right)_{h\Sigma_{n}}\]
  for $C \in \mathcal{C} \simeq \Fun(\mathcal{O}_{\angled{1}}^{\simeq},
  \mathcal{C})$.
\end{example}

\begin{remark}
  Our description of free algebras differs from what Lurie calls
  ``free algebras'' in \cite[Section 3.1.3]{ha}, because Lurie defines
  these to be given by operadic Kan extension along the inclusion
  $\mathcal{O} \times_{\xF_{*}} \xF_{*}^{\xint} \to \mathcal{O}$ where
  the source is the subcategory containing \emph{all} morphisms in
  $\mathcal{O}$ lying over inert morphisms in $\xF_{*}$, not just the
  cocartesian ones. Lurie's construction amounts to specifying the
  unary operations in advance and freely adding the $n$-ary operations
  for $n>1$, while our version adds all the operations freely.
\end{remark}

\begin{example}
  The pattern $\simp^{\op,\flat}$ is also extendable.
  The analogues of Examples~\ref{ex:symopdext} and \ref{ex:symopdext2}
  hence also hold for non-symmetric \iopds{}.
\end{example}

\begin{example}\label{ex:genopdext}
  The patterns $\xF_{*}^{\natural}$ and $\simp^{\op,\natural}$ are
  also extendable.  Hence any morphism of \emph{generalized} symmetric
  or non-symmetric \iopds{} is extendable.
\end{example}

\begin{remark}\label{rmk:geniopdext}
Suppose
\[
\begin{tikzcd}
\mathcal{O} \arrow{rr}{f} \arrow{dr} & & \mathcal{P} \arrow{dl}
\\
& \xF_{*}
\end{tikzcd}
\]
is a morphism of generalized symmetric \iopd{}s. Then the previous
example does \emph{not} say that we can compute free Segal
$\mathcal{P}^{\flat}$-objects on Segal
$\mathcal{O}^{\flat}$-objects, as $f_{!}$ generally will not
restrict to a functor between these. In the definition of
extendability, condition (1) is still automatic (as the inert
morphisms in $\xF_{*}^{\natural}$ and $\xF_{*}^{\flat}$ are the
same), while condition (3) reduces to $\xF_{*}^{\flat}$ being
extendable. Thus the morphism $f^{\flat} \colon \mathcal{O}^{\flat}
\to \mathcal{P}^{\flat}$ is extendable \IFF{} for all $P$ over
$\angled{n}$ in $\xF_{*}$ the functor
\[ \mathcal{O}^{\act}_{/P} \to \prod_{i = 1}^{n}
\mathcal{O}^{\act}_{/\rho_{i,!}P} \]
is cofinal, where $\rho_{i} \colon \angled{n} \to \angled{1}$ is as
in the introduction.
\end{remark}

\section{Polynomial Monads from Patterns}
\label{sec:poly}
In this section we introduce the notion of \emph{polynomial monad} on
an \icat{} of presheaves, and prove that the free Segal
$\mathcal{O}$-space monad for an extendable pattern $\mathcal{O}$ is
polynomial. Moreover, we show that this is compatible with Segal
morphisms of algebraic patterns, yielding a functor
\[ \mathfrak{M} \colon \AlgPatt_{\name{ext}}^{\Seg} \to \name{PolyMnd} \] between the
subcategory of $\AlgPatt$ consisting of extendable patterns and Segal
morphisms, and an \icat{} of polynomial monads. We start by
introducing some terminology:

\begin{defn}
  A natural transformation $\phi \colon F \to G$ is \emph{cartesian}
  if the naturality squares
  \csquare{F(x)}{F(y)}{G(x)}{G(y)}{F(f)}{\phi_x}{\phi_y}{G(f)}
  are all cartesian.
\end{defn}

\begin{defn}
  A functor $F \colon \mathcal{C} \to \mathcal{D}$ is a \emph{local
    right adjoint} if for every $c \in \mathcal{C}$ the induced
  functor $\mathcal{C}_{/c} \to \mathcal{D}_{/Fc}$ is a right
  adjoint.
\end{defn}

\begin{lemma}
  If $\mathcal{C}$ and $\mathcal{D}$ are presentable \icats{}, then
  the following are equivalent for a functor $F \colon \mathcal{C} \to
  \mathcal{D}$:
  \begin{enumerate}[(1)]
  \item $F$ is accessible and preserves weakly contractible limits.
  \item $F$ is a local right adjoint.
  \item The functor $F_{/*} \colon \mathcal{C} \to \mathcal{D}_{/F(*)}$ has a left
    adjoint.
  \end{enumerate}
\end{lemma}
\begin{proof}
  The equivalence of (1) and (2) was proved as \cite[Proposition
  2.2.8]{AnalMnd}.  Since (3) is a special case of (2), it remains to
  prove that (3) implies (1). By the adjoint functor theorem
  \cite[Corollary 5.5.2.9]{ht}, it follows from (3) that $F_{/*}$ is
  accessible and preserves limits. The forgetful functor
  $\mathcal{D}_{/F(*)} \to \mathcal{D}$ preserves and creates all
  colimits, as well as weakly contractible limits, by \cite[Lemma
  2.2.7]{AnalMnd}, so this implies that $F$ itself is accessible and
  preserves weakly contractible limits.
\end{proof}

\begin{defn}
  A monad $T$ is \emph{cartesian} if its multiplication and unit are
  cartesian natural transformations, and  is \emph{polynomial}
  if it is cartesian and the underlying endofunctor is a
  local right adjoint.
\end{defn}

\begin{remark}
  For ordinary categories, our notion of polynomial monads is the same
  as the \emph{strongly cartesian} monads considered in
  \cite{BergerMelliesWeber}. For monads on \icats{} of the form
  $\mathcal{S}_{/X}$ for $X \in \mathcal{S}$, we recover the
  polynomial monads studied in \cite{AnalMnd} (see Theorem 2.2.3
  there), which is our reason for adopting this terminology.
\end{remark}

\begin{propn}\label{prop:polymonad}
  If $\mathcal{O}$ is an extendable algebraic pattern, then the free
  Segal $\mathcal{O}$-space monad $T_{\mathcal{O}}$ on
  $\Fun(\mathcal{O}^{\el}, \mathcal{S})$ is a polynomial monad.
\end{propn}
\begin{proof}
  Since $\Seg_{\mathcal{O}^{(\xint)}}(\mathcal{S})$ is an accessible
  localization of $\Fun(\mathcal{O}^{(\xint)}, \mathcal{S})$, the
  inclusions $\Seg_{\mathcal{O}^{(\xint)}}(\mathcal{S})
  \hookrightarrow \Fun(\mathcal{O}^{(\xint)}, \mathcal{S})$ are
  accessible and preserve limits. The endofunctor $T_{\mathcal{O}}$ of
  $\Seg_{\mathcal{O}^{\xint}}(\mathcal{S})$ factors as a composite
  \[ \Seg_{\mathcal{O}^{\xint}}(\mathcal{S}) \hookrightarrow
    \Fun(\mathcal{O}^{\xint}, \mathcal{S}) \xto{j_{\mathcal{O},!}}
    \Fun(\mathcal{O}, \mathcal{S}) \xto{j_{\mathcal{O}}^{*}}
    \Fun(\mathcal{O}^{\xint}, \mathcal{S}),\] where the composite
  lands in the subcategory
  $\Seg_{\mathcal{O}^{\xint}}(\mathcal{S})$. To see that
  $T_{\mathcal{O}}$ is a local right adjoint it suffices to show that
  the three functors in this composition are accessible and
  preserve weakly contractible limits. All three functors 
  are clearly accessible and except for $j_{\mathcal{O},!}$ they preserve
  limits. It therefore remains to show that $j_{\mathcal{O},!}$
  preserves weakly contractible limits. By
  Lemma~\ref{lem:uniqintcoinit} for $O \in \mathcal{O}$ and
  $F \in \Fun(\mathcal{O}^{\xint}, \mathcal{S})$, the value of
  $j_{\mathcal{O},!}F$ at $O$ is
  $\colim_{X \in \Act_{\mathcal{O}}(O)} F(X)$. Since
  $\Act_{\mathcal{O}}(O)=(\xxO^\xint)^{\xact}_{/O}$ is an
  $\infty$-groupoid, this factors through the forgetful functor
  $\mathcal{S}_{/\Act_{\mathcal{O}}(O)} \to \mathcal{S}$, which
  detects weakly contractible limits by \cite[Lemma
  2.2.7]{AnalMnd}. It therefore suffices to show that the functor
  $\Fun(\mathcal{O}^{\xint}, \mathcal{S}) \to
  \mathcal{S}_{/\Act_{\mathcal{O}}(O)}$ taking $F$ to
  $\colim_{X \in \Act_{\mathcal{O}}(O)} F(X) \to
  \Act_{\mathcal{O}}(O)$ preserves weakly contractible limits. But
  this factors as restriction along
  $\Act_{\mathcal{O}}(O) \to \mathcal{O}^{\xint}$, which certainly
  preserves limits, followed by the colimit functor
  $\Fun(\Act_{\mathcal{O}}(O), \mathcal{S}) \to
  \mathcal{S}_{/\Act_{\mathcal{O}}(O)}$, which is an equivalence.

  Next, we show that the multiplication transformation
  $T_{\mathcal{O}}^{2} \to T_{\mathcal{O}}$ is cartesian. For $O \in
  \mathcal{O}$, we have an equivalence
  \[ (T_{\mathcal{O}}^{2}F)(O) \simeq \colim_{X \in
      \Act_{\mathcal{O}}(O)} (T_{\mathcal{O}}F)(O) \simeq \colim_{X \in
      \Act_{\mathcal{O}}(O)} \colim_{Y \in \Act_{\mathcal{O}}(X)}
    F(Y) \simeq \colim_{(Y \actto X \actto O) \in \Act^{2}_{\mathcal{O}}(O)} F(Y),\]
  where $\Act^{2}_{\mathcal{O}}(O) \to \Act_{\mathcal{O}}(O)$ is the
  left fibration for the functor taking $X \actto O$ to
  $\Act_{\mathcal{O}}(X)$. We then have an identification
  \[ \Act^{2}_{\mathcal{O}}(O) \simeq \{Y \overset{g}{\actto} X \overset{f}{\actto} O :
    \text{$f,g$ active}\}\]
  under which the multiplication transformation 
  $T_{\mathcal{O}}^{2}F(X) \to T_{\mathcal{O}}F(X)$ is the morphism induced on
  colimits by the map $\Act^{2}_{\mathcal{O}}(O) \to
  \Act_{\mathcal{O}}(O)$ given by composition of active
  morphisms. Given $F \to G$, we want to show that the square
  \nolabelcsquare{\colim_{(Y \actto X \actto O) \in \Act^2_{\mathcal{O}}(O)}
    F(Y)}{\colim_{(Y \actto X \actto O) \in \Act^2_{\mathcal{O}}(O)}
    G(Y)}{\colim_{(Y \actto O) \in \Act_{\mathcal{O}}(O)}
    F(Y)}{\colim_{(Y \actto O) \in \Act_{\mathcal{O}}(O)} G(Y)}
  is cartesian. To see this it suffices to show that the square on
  fibres over $(Y \overset{f}{\actto} O) \in \Act_{\mathcal{O}}(O)$ is
  cartesian. The fibre $(T_{\mathcal{O}}^{2}F(X))_{f}$ we can identify
  with the colimit over the fibre
  \[ \Act_{\mathcal{O}}^{2}(O)_{f} \simeq \left\{
      \begin{tikzcd}
        {} & X \arrow{dr} \\
        Y \arrow{ur} \arrow{rr}{f} & & O
      \end{tikzcd}
    \right \}\]
  of the \emph{constant} functor with value $F(Y)$. The square of
  fibres is therefore
  \nolabelcsquare{\Act_{\mathcal{O}}^{2}(O)_{f} \times
    F(Y)}{\Act_{\mathcal{O}}^{2}(O)_{f} \times G(Y)}{F(Y)}{G(Y),}
  which is indeed cartesian.

  The value of the unit transformation $F(O) \to T_{\mathcal{O}}F(O)$
  is similarly induced by the map $\{\id_{O}\} \to
  \Act_{\mathcal{O}}(O)$. To see that the unit transformation is
  cartesian we must show that for $F \to G$ the square
  \nolabelcsquare{F(O)}{G(O)}{\colim_{(Y \actto O) \in \Act_{\mathcal{O}}(O)}
    F(Y)}{\colim_{(Y \actto O) \in \Act_{\mathcal{O}}(O)} G(Y)}
  is cartesian.
  It again suffices to consider the square of fibres over $(X \overset{f}{\actto} O) \in \Act_{\mathcal{O}}(O)$. The fibre of $\{\id_{O}\} \to
  \Act_{\mathcal{O}}(O)$ at $f$ is the space
 \[P_{f}:= \Map_{\Act_{\mathcal{O}}(O)}(\id_{O},f)\] of paths from
  $\id_{\mathcal{O}}$ to $f$  in $\Act_{\mathcal{O}}(O)$ (which is empty if $\id_{O}$
  and $f$ are not equivalent), and the square of fibres is
  \nolabelcsquare{P_f \times F(O)}{P_f \times G(O)}{F(X)}{G(X),}
  which is cartesian as required.
\end{proof}

\begin{remark}
  We can regard polynomial monads as being the monads in an
  $(\infty,2)$-category whose objects are presheaf \icats{}, whose
  morphisms are local right adjoints, and whose 2-morphisms are
  cartesian transformations. The natural morphisms between polynomial
  monads are then the lax morphisms of monads in this
  $(\infty,2)$-category. If $T$ is a polynomial monad on
  $\mathcal{S}^{\mathcal{I}}$ and $S$ is a polynomial monad on
  $\mathcal{S}^{\mathcal{J}}$, then by the results of \cite{adjmnd}
  these correspond to commutative squares
  \[
    \begin{tikzcd} \Alg_{S}(\mathcal{S}^{\mathcal{J}}) \arrow{r}{\Phi}
\arrow[d, swap, "U_{S}"] & \Alg_{T}(\mathcal{S}^{\mathcal{I}})
\arrow{d}{U_{T}} \\ \mathcal{S}^{\mathcal{J}} \arrow{r}{f^{*}}&
\mathcal{S}^{\mathcal{I}},
    \end{tikzcd}
  \] for some functor $f \colon \mathcal{I} \to \mathcal{J}$, such
  that the mate transformation
  \[ F_{T}f^{*} \to \Phi F_{S}\] is cartesian. Noting the
  contravariance here, this motivates the
  following definition of an \icat{} of polynomial monads:
\end{remark}

\begin{definition}
  Consider the pullback
  \[ \Fun(\Delta^{1}, \LCatI) \times_{\LCatI} \CatI\] along
  $\name{ev}_{1} \colon \Fun(\Delta^{1}, \LCatI) \to \LCatI$ and
  $\mathcal{S}^{(\blank)} \colon \CatI^{\op}\to \LCatI$. We write
  $\name{PolyMnd}^{\op}$ for the subcategory of this pullback
  whose objects
  are the monadic right adjoints of polynomial monads, and whose
  morphisms are commutative squares whose mate transformations are
  cartesian.
 \end{definition}

 \begin{remark}
   Note that since $U_{T}$ detects pullbacks, the mate transformation
   above is cartesian \IFF{} the transformation
  \[ Tf^{*} \to f^{*}S\] obtained by composing with $U_{T}$ is
cartesian.
\end{remark}

Next, we observe that any Segal morphism between extendable
patterns gives a morphism of polynomial monads:
\begin{propn}\label{propn:cartmate}
  Suppose $f \colon \mathcal{O} \to \mathcal{P}$ is a Segal morphism
  between extendable patterns. Then the mate transformation
  \[j_{\mathcal{O},!}f^{\xint,*} \to f^{*}j_{\mathcal{P},!}\] of
  functors
  $\Seg_{\mathcal{P}^{\xint}}(\mathcal{S}) \to
  \Seg_{\mathcal{O}}(\mathcal{S})$ is cartesian.
\end{propn}
\begin{proof}
We have to show that for every morphism $\Phi\to \Psi$ the commutative square
\[
\begin{tikzcd}
j_{\mathcal{O},!}f^{\xint,*}\Phi \arrow{r} \arrow{d} &
f^{*}j_{\mathcal{P},!}\Phi \arrow{d} \\
j_{\mathcal{O},!}f^{\xint,*}\Psi \arrow{r} & f^{*}j_{\mathcal{P},!}\Psi
\end{tikzcd}
\]
is cartesian in $\xSeg_\xxO(\xS)$. Since
$\Seg_{\mathcal{P}^{\xint}}(\mathcal{S})$ has a terminal object it
suffices to consider $\Psi \simeq *$, in which case we obtain the commutative square
  \[
    \begin{tikzcd}
      \colim_{X \in \Act_{\mathcal{O}}(E)} \Phi(fX) \arrow{r}
      \arrow{d} & \colim_{Y \in \Act_{\mathcal{P}}(f(E))} \Phi(Y)
      \arrow{d} \\
      \Act_{\mathcal{O}}(E) \arrow{r} & \Act_{\mathcal{P}}(f(E))
    \end{tikzcd}
  \]
  after evaluating at an object $E \in \mathcal{O}^{\el}$. To show
  that this square is cartesian, it now suffices to observe that for
  every point $(X \to E) \in \Act_{\mathcal{O}}(E)$, the map on fibres
  is the identity $\Phi(fX) \to \Phi(fX)$.
\end{proof}

\begin{defn}
  We let $\name{AlgPatt}^{\Seg}_{\name{ext}}$ denote the subcategory of
  $\name{AlgPatt}$ whose objects are the extendable patterns and whose
  morphisms are the Segal morphisms.
\end{defn}

\begin{cor}\label{cor:patttopolymnd}
  The functor $\AlgPatt^{\Seg} \to \Fun(\Delta^{1}, \CatI)^{\op}$ taking a
  pattern $\mathcal{O}$ to the monadic right adjoint $U_{\mathcal{O}} \colon \Seg_{\mathcal{O}}(\mathcal{S}) \to
  \Fun(\mathcal{O}^{\el}, \mathcal{\mathcal{S}})$ restricts to a functor
  $ \mathfrak{M} \colon \AlgPatt_{\name{ext}}^{\Seg} \to \PolyMnd.$ \qed
\end{cor}

\section{Generic Morphisms and the Nerve Theorem}\label{sec:nerve}
In the previous section we saw that the free Segal space monad for any
extendable pattern was a polynomial monad. Our next goal is to extract
an extendable pattern from any polynomial monad. As a first step
towards this, in this section we prove an \icatl{} version of Weber's
nerve theorem \cite{WeberFamilial}; our proof was particularly
inspired by that of Berger, Melli\`es, and
Weber~\cite{BergerMelliesWeber}.

We begin by defining \emph{generic morphisms} with respect to a
local right adjoint functor, and extend some basic observations about
them from \cite{Weber} to the \icatl{} setting.
\begin{defn}
  Suppose $F \colon \mathcal{C} \to \mathcal{D}$ is a local right
  adjoint functor between presentable \icats{}. Let
  $L_{*}\colon \mathcal{D}_{/F(*)} \to \mathcal{C}$ be the left
  adjoint to $F_{/*} \colon \mathcal{C} \to \mathcal{D}_{/F(*)}$; we
  will abusively write $L_{*}D$ for the value of $L_{*}$ at an object
  $D \to F(*)$. For any morphism $D \xto{\phi} F(C)$ in $\mathcal{D}$,
  we can view $\phi$ as a morphism in $\mathcal{D}_{/F(*)}$ via the
  map $F(q) \colon F(C) \to F(*)$, where $q$ is the unique morphism
  $C \to *$. We say $\phi$ is $F$-\emph{generic} (or just
  \emph{generic} if $F$ is clear from context) if the adjoint morphism
  \[L_{*}D \simeq L_{*}(F(q)\circ \phi) \to C\]
  is an equivalence. (In other words, the generic morphisms are
  precisely the unit morphisms $D \to F_{/*}L_{*}D$.)
\end{defn}

\begin{remark}\label{rmk:genericfill}
  Using the universal property of the left adjoint, we can rephrase
  this definition purely in terms of $F$ as follows: $\phi \colon D \to F(B)$ is $F$-generic if for every
  commutative square
  \[
    \begin{tikzcd}
      D \arrow{r}{\psi} \arrow[d, swap, "\phi"] & F(A) \arrow{d} {F(\alpha)} \\
      F(B) \arrow{r}{F(\beta)} \arrow[dotted]{ur} & F(*)
    \end{tikzcd}
  \]
  there exists a unique morphism $\gamma \colon B \to A$ such
  that $F(\gamma) \circ \phi \simeq \psi$ and the equivalence in the
  square arises by combining this with the canonical equivalence $F(\alpha) \circ
  F(\gamma) \simeq F(\alpha\gamma) \simeq F(\beta)$ induced by $*$
  being terminal. This is the version of the
  definition considered in \cite{Weber}.
\end{remark}

\begin{lemma}\label{lem:genericfillX}
  Let
  $\phi \colon D \to F(B)$ be an $F$-generic morphism. Then given a
  commutative square
  \[
    \begin{tikzcd}
      D \arrow{r}{\psi} \arrow[d, swap, "\phi"] & F(A) \arrow{d} {F(\alpha)} \\
      F(B) \arrow{r}{F(\beta)} \arrow[dotted]{ur} & F(X),
    \end{tikzcd}
  \]
  there exists a unique commutative triangle
  \[
    \begin{tikzcd}
      A \arrow{rr}{\gamma} \arrow[dr, swap, "\alpha"] & & B \arrow{dl}{\beta}\\
       & X
    \end{tikzcd}
  \]
  such that $F(\gamma) \circ \phi \simeq \psi$ and the equivalence in the
  square arises by combining this with the equivalence $F(\alpha) \circ
  F(\gamma) \simeq F(\alpha\gamma) \simeq F(\beta)$ given by applying
  $F$ to the triangle.
\end{lemma}
\begin{proof}
	The existence of a unique filler in the original square is equivalent to the existence of such a filler in the adjoint square
  \[
    \begin{tikzcd}
      L_{X}D \arrow{r} \arrow{d} & A \arrow{d}{\alpha} \\
      B \arrow{r}{\beta} & X.
    \end{tikzcd}
    \]
  Since $F$ preserves pullbacks, if $\xi$ denotes the unique morphism
  $X \to *$ we have a commutative square of right adjoints
  \[
    \begin{tikzcd}
      \mathcal{C} \arrow{r}{F_{/*}} \arrow[d, swap, "\xi^{*}"] &
      \mathcal{D}_{/F(*)} \arrow{d}{F(\xi)^{*}} \\
      \mathcal{C}_{/X} \arrow{r}{F_{/X}} & \mathcal{D}_{/F(X)}.
    \end{tikzcd}
  \]
  This induces a corresponding square of left adjoints
  \[
    \begin{tikzcd}
      \mathcal{D}_{/F(X)} \arrow{r}{L_{X}} \arrow[d, swap, "F(\xi)_{!}"] &
      \mathcal{C}_{/X} \arrow{d}{\xi_{!}} \\
      \mathcal{D}_{/F(*)} \arrow{r}{L_{*}} & \mathcal{C}.
    \end{tikzcd}
  \]
  Thus $\xi_{!}L_{X} \simeq L_{*}F(\xi)_{!}$; since $\xi_{!}$ detects
  equivalences, we see that for
  $D \xto{\phi} F(B) \xto{F(\beta)} F(X)$ the adjoint morphism
  $L_{X}D \to B$ over $X$ is equivalent to $L_{*}X \to B$ computed
  using the morphism $F(B) \to F(*)$ that is the image of $B \to *$,
  as this is the composite
  $F(B) \xto{F(\beta)} F(X) \xto{F(\xi)} F(*)$. Since $\phi$ is generic,
  it therefore follows that the map $L_{X}D \to B$ is also an
  equivalence, hence the unique filler arises from the composite $B
  \simeq L_{X}D \to A$.
\end{proof}

\begin{remark}\label{rmk:genericfact}
  For any morphism $\phi \colon D \to F(C)$, if $\psi \colon L_{*}D
  \to C$ is the adjoint morphism, 
  we can write $\phi$ as a
  composite
  \[ D \xto{\eta_{D}} F(L_{*}D) \xto{F(\psi)} F(C),\]
  where $\eta_{D}$ is the unit of the adjunction $L_{*}\dashv F_{/*}$.
  This is the \emph{unique} factorization of $\phi$ as a generic
  morphism followed by a morphism in the image of $F$; we will often
  refer to this as the \emph{generic--free factorization} of
  $\phi$.
\end{remark}

\begin{lemma}[Cf.~{\cite[Proposition 5.10]{Weber}}]\label{lem:cartnattr}
  Suppose $F,G \colon \mathcal{C} \to \mathcal{D}$ are local right
  adjoint functors between presentable \icats{} and $\phi \colon F \to
  G$ is a cartesian natural transformation. Then a morphism
  $f \colon D \to F(C)$ is
  $F$-generic \IFF{} the composite $D \to F(C) \to G(C)$ is $G$-generic.
\end{lemma}
\begin{proof}
  Since $\phi$ is a cartesian transformation, we have natural
  cartesian squares
  \[
    \begin{tikzcd}
      F(X) \arrow{r} \arrow{d} & G(X) \arrow{d}\\
      F(*) \arrow{r}{\phi(*)} & G(*).
    \end{tikzcd}
  \]
  This means we can write $F_{/*}$ as the composite
  \[ \mathcal{C} \xto{G_{/*}} \mathcal{D}_{/G(*)} \xto{\phi(*)^{*}}
    \mathcal{D}_{/F(*)}.\]
  But then the left adjoint $L_{*,F}$ of $F_{/*}$ is the composite
  \[ \mathcal{D}_{/F(*)} \xto{\phi(*)_!} \mathcal{D}_{/G(*)}
    \xto{L_{*,G}} \mathcal{C},\]
  where $L_{*,G}$ denotes the left adjoint to $G_{/*}$. Given $f \colon
  D \to F(C)$, this means the adjoint morphism $L_{*,F}D \to C$ is the
  same as the adjoint morphism $L_{*,G}D \to C$ for the composite $D
  \to F(C) \to G(C)$.
\end{proof}

\begin{lemma}[Cf.~{\cite[Lemma 5.14]{Weber}}]\label{lem:lracomp}
  Suppose $F \colon \mathcal{C} \to \mathcal{D}$ and $G \colon
  \mathcal{D} \to \mathcal{E}$ are local right adjoint functors
  between presentable \icats{}. If $f \colon D \to F(C)$ is
  $F$-generic and $g \colon E
  \to G(D)$ is $G$-generic, then the composite
  \[ E \xto{g} G(D) \xto{G(f)} GF(C) \]
  is $GF$-generic.
\end{lemma}
\begin{proof}
  The functor $(GF)_{/*}$ factors in two steps as
  \[ \mathcal{C} \xto{F_{/*}} \mathcal{D}_{/F(*)} \xto{G_{/F(*)}}
    \mathcal{E}_{/GF(*)}.\]
  The left adjoint is therefore also computed in two steps; to find
  the morphism adjoint to $G(f)g$ we first get the
  commutative diagram
  \[
    \begin{tikzcd}
      L_{*,G}E \arrow{r}{\sim} \arrow{dr} & D \arrow{d} \arrow{r}{f} &
      F(C) \arrow{dl} \\
      & F(*),
    \end{tikzcd}
  \]
  and then $L_{*,F}L_{*,G}E \isoto L_{*,F}D \isoto C$, which is an
  equivalence as required.
\end{proof}

\begin{defn}\label{def:Wint}
  Suppose $\mathcal{I}$ is a small \icat{} and $T$ is a polynomial
  monad on the functor \icat{} $\mathcal{S}^{\mathcal{I}}$.  We define
  $\mathcal{U}(T)^{\op}$ to be the full subcategory of
  $\mathcal{S}^\xI$ spanned by the objects $X$ that admit a generic
  morphism $I \to TX$
  with $I \in \mathcal{I}^{\op}$ (regarded as an
  object of $\mathcal{S}^{\mathcal{I}}$ through the Yoneda
  embedding). We write $\xW(T)^\op$ for the full subcategory of
  $\xAlg_T(\xS^\xI)$ spanned by the free $T$-algebras on the objects
  of $\mathcal{U}(T)$.
\end{defn}

\begin{remark}
  From the definition of generic morphisms it follows that we can
  equivalently describe the objects of $\mathcal{U}(T)^{\op}$
  as those of the form $L_*I$ for some
  $I\in \xI^\op$ and some morphism $I\to T*$ in
  $\mathcal{S}^{\mathcal{I}}$. 
\end{remark}

\begin{lemma}\label{lem:unitmultgen}
  Let $T$ be a polynomial monad on $\mathcal{S}^{\mathcal{I}}$.
  \begin{enumerate}[(i)]
  \item For any object $X \in \mathcal{S}^{\mathcal{I}}$, the unit map
    $X \to T(X)$ is generic.
  \item If $X \xto{\phi} T(Y)$ and $Y \xto{\psi} T(Z)$ are generic morphisms, then
    the composite
    \[ X \xto{\phi} TY \xto{T\psi} T^{2}Z \xto{\mu_{Z}} TZ \]
    is generic, where $\mu$ denotes the multiplication transformation
    of the monad.
  \end{enumerate}
\end{lemma}
\begin{proof}
  Since $T$ is a polynomial monad, the unit transformation $\id \to T$
  is cartesian and so by Lemma~\ref{lem:cartnattr} the unit map $X \to
  TX$ is generic for all $X$ (since an $\id$-generic map is precisely
  an equivalence).

  The composite $X \xto{\phi} TY \xto{T\psi} T^{2}Z$ is $T^{2}$-generic by
  Lemma~\ref{lem:lracomp}, and as the multiplication $\mu$ is a
  cartesian transformation this implies the composite of this with
  $\mu_{Z} \colon T^{2}Z \to TZ$ is $T$-generic by 
  Lemma~\ref{lem:cartnattr}.
\end{proof}

\begin{propn}\label{propn:WTint}
  Let $T$ be a polynomial monad on $\mathcal{S}^{\mathcal{I}}$.
  \begin{enumerate}[(i)]
  \item The full subcategory $\mathcal{U}(T)^{\op}$ contains
    $\mathcal{I}^{\op}$.
  \item For any generic morphism $X \to TY$ with $X \in
    \mathcal{U}(T)^{\op}$, the object $Y$ also lies in $\mathcal{U}(T)^{\op}$.
 \end{enumerate}
\end{propn}
\begin{proof}
  The unit map $I \to TI$ is generic by
  Lemma~\ref{lem:unitmultgen}(i). Hence $I \to TI \to T*$ is a
  generic--free factorization, where the second map is the image under
  $T$ of the unique map $I \to *$. This shows that $I$ is in
  $\mathcal{U}(T)^{\op}$, which proves (i).

  To prove (ii), observe that since $X$ is in
  $\mathcal{U}(T)^{\op}$, we have a generic morphism $I \to TX$
  with $I$ in $\mathcal{I}^{\op}$. Then by
  Lemma~\ref{lem:unitmultgen}(ii) the composite
  \[ I \to TX \to T^{2}Y \xto{\mu_{Y}} TY \]
  is also generic, which means that $Y$ is also
  in $\mathcal{U}(T)^{\op}$.
\end{proof}

\begin{remark}\label{rem:W}
  Note that the functor $\mathcal{U}(T) \to \mathcal{W}(T)$ need not
  exhibit $\mathcal{U}(T)$ as a subcategory of $\mathcal{W}(T)$.
\end{remark}

Our goal is now to show that the algebras for the polynomial monad $T$
can be described in terms of the \icats{} $\mathcal{U}(T)$ and
$\mathcal{W}(T)$ --- this is the content of the nerve theorem. The
next proposition gives the key input needed to prove this.

\begin{notation}\label{no nu}
  Given a functor $j \colon \mathcal{A}^{\op} \to
  \mathcal{S}^{\mathcal{I}}$, we let
  \[\nu_{\mathcal{A}}\colon \mathcal{S}^{\mathcal{I}} \to 
    \Fun((\mathcal{S}^{\mathcal{I}})^{\op}, \mathcal{S}) \xto{j^{*}}
    \mathcal{S}^{\mathcal{A}}\] denote the composition of the Yoneda
  embedding and $j^*$. Thus $\nu_\xA$ takes
  $\Phi \colon \mathcal{I} \to \mathcal{S}$ to
  $\Map_{\mathcal{S}^{\mathcal{I}}}(j(\blank), \Phi)$.
\end{notation}

\begin{propn}\label{propn:WTintnerve}
  Let $T$ be a polynomial monad on $\mathcal{S}^{\mathcal{I}}$.
  \begin{enumerate}[(i)]
  \item The functor $\nu_{\mathcal{U}(T)} \colon
    \mathcal{S}^{\mathcal{I}} \to
    \mathcal{S}^{\mathcal{U}(T)}$ is fully faithful, and
    given by right Kan extension along the inclusion $\mathcal{I}
    \hookrightarrow \mathcal{U}(T)$.
  \item For every $\Phi \in \mathcal{S}^{\mathcal{I}}$, the diagram
    \[ (\mathcal{U}(T)^{\op})_{/\Phi}^{\triangleright} \to
      \mathcal{S}^{\mathcal{I}}\]
    is a colimit diagram.
  \item For every $\Phi$ in $\mathcal{S}^{\mathcal{I}}$ the composite
    diagram
    \[ (\mathcal{U}(T)^{\op})_{/\Phi}^{\triangleright} \to
      \mathcal{S}^{\mathcal{I}} \xto{T} \mathcal{S}^{\mathcal{I}}
      \xto{\nu_{\mathcal{U}(T)}}
        \mathcal{S}^{\mathcal{U}(T)}\]
      is a colimit diagram. (In other words, the colimit diagram in
      (ii) is preserved by the functor $\nu_{\mathcal{U}(T)}T$.)
  \end{enumerate}
\end{propn}

The proof uses the following technical observation:
\begin{lemma}\label{lem:nerveRKE}
  Suppose $j \colon \mathcal{A}^{\op} \hookrightarrow
  \mathcal{S}^{\mathcal{I}}$ is a full subcategory of a presheaf
  \icat{} $\mathcal{S}^{\mathcal{I}}$ such that $\mathcal{I}^{\op}$
  (viewed as a full subcategory of $\mathcal{S}^{\mathcal{I}}$ via the
  Yoneda embedding) is contained in $\mathcal{A}^{\op}$, so that we have
  a fully faithful functor $i \colon \mathcal{I} \to \mathcal{A}$. Then:
  \begin{enumerate}[(i)]
  \item $\nu_{\mathcal{A}}$ is equivalent to the functor $i_{*} \colon
    \mathcal{S}^{\mathcal{I}}\to \mathcal{S}^{\mathcal{A}}$ given by
    right Kan extension along $i$.
  \item $\nu_{\mathcal{A}}$ is fully faithful.
  \item For every $\Phi$ in $\mathcal{S}^{\mathcal{I}}$, the diagram
    \[ (\mathcal{A}^{\op})_{/\Phi}^{\triangleright} \to \mathcal{S}^{\mathcal{I}}\]
    is a colimit diagram, and this colimit is preserved by $\nu_{\mathcal{A}}$.
  \end{enumerate}
\end{lemma}
\begin{proof}
  For any $\Phi \in \mathcal{S}^{\mathcal{I}}$, the diagram
  $(\mathcal{I}^{\op})_{/\Phi}^{\triangleright} \to
  \mathcal{S}^{\mathcal{I}}$ is a colimit, so we have a natural equivalence
  \[ \nu_{\mathcal{A}}\Phi(a) \simeq \Map(j(a), \Phi) \simeq
    \Map(\colim_{x \in (\mathcal{I}^{\op})_{/j(a)}} y(x), \Phi) \simeq
    \lim_{x \in \mathcal{I}_{a/}} \Phi(x) \simeq (i_{*}\Phi)(a).\]
  This proves (i).
  Since $i\colon \mathcal{I} \to \mathcal{A}$ is fully faithful, it
  follows that $i_{*}$ is also fully faithful, which proves (ii). To
  prove (iii), since $\nu_{\mathcal{A}}$ is fully faithful it suffices
  to show that the composite
  \[ (\mathcal{A}^{\op})_{/\Phi}^{\triangleright} \to
    \mathcal{S}^{\mathcal{I}} \xto{\nu_{\mathcal{A}}}
    \mathcal{S}^{\mathcal{A}}\] is a colimit diagram. But this is now
  a Yoneda cocone for $\mathcal{A}^{\op}$, which is always a colimit
  in $\mathcal{S}^{\mathcal{A}}$.
\end{proof}

\begin{proof}[Proof of Proposition~\ref{propn:WTintnerve}]
  (i) and (ii) follow from Proposition~\ref{propn:WTint}(i) and
  Lemma~\ref{lem:nerveRKE}. To prove (iii), since colimits in functor
  categories are computed objectwise, it suffices to show that for
  every $X \in \mathcal{U}(T)$ and $\Phi \in \mathcal{S}^{\mathcal{I}}$, the morphism
  \[ \colim_{Y \in (\mathcal{U}(T)^{\op})_{/\Phi}}
    \Map_{\mathcal{S}^{\mathcal{I}}}(X, TY) \to
    \Map_{\mathcal{S}^{\mathcal{I}}}(X, T\Phi)\] is an
  equivalence. Let
  $\mathcal{E} \to (\mathcal{U}(T)^{\op})_{/\Phi}$ be the left
  fibration for the functor
  $(\mathcal{U}(T)^{\op})_{/\Phi} \to \mathcal{S}$ taking $Y$
  to $\Map_{\mathcal{S}^{\mathcal{I}}}(X, TY)$; then we have a
  pullback square
  \[
    \begin{tikzcd}
      \mathcal{E} \arrow{rr} \arrow{d} & &
      \mathcal{S}^{\mathcal{I}}_{X/} \arrow{d} \\
      (\mathcal{U}(T)^{\op})_{/\Phi} \arrow{r}&
      \mathcal{S}^{\mathcal{I}} \arrow{r}{T} & \mathcal{S}^{\mathcal{I}},\\
    \end{tikzcd}
  \]
  so that an object of $\mathcal{E}$ is a pair
  $(Y \to \Phi, X \to TY)$. By \cite[Proposition 3.3.4.5]{ht}, the
  space
  $\colim_{Y \in (\mathcal{U}(T)^{\op})_{/\Phi}}
  \Map_{\mathcal{S}^{\mathcal{I}}}(X, TY)$ is equivalent to the space
  $\| \mathcal{E} \|$ obtained by inverting all morphisms in
  $\mathcal{E}$, and the morphism we are interested in is the map of
  spaces induced by the functor of \icats{}
  $\mathcal{E} \to \Map_{\mathcal{S}^{\mathcal{I}}}(X, T\Phi)$ taking
  $(Y \xto{\alpha} \Phi, X \to TY)$ to the composite
  $X \to TY \xto{T\alpha} T\Phi$. By \cite[Proposition 4.1.1.3]{ht} a
  morphism of spaces that arises from a cofinal functor of \icats{} is
  an equivalence, so it suffices to show that the functor
  $\mathcal{E} \to \Map_{\mathcal{S}^{\mathcal{I}}}(X, T\Phi)$ is
  cofinal.  Since every functor to an $\infty$-groupoid is a cartesian
  fibration, to prove this we may apply \cite[Lemma 4.1.3.2]{ht},
  which says that a cartesian fibration with weakly contractible
  fibres is cofinal. It thus suffices to check that the fibres
  $\mathcal{E}_{\phi}$ at $\phi \colon X \to T\Phi$ are weakly
  contractible. But the fibre $\mathcal{E}_{\phi}$ is the \icat{} of
  factorizations of $\phi$ of the form $X \to TY \xto{T\alpha}
  T\Phi$. Since $T$ is a local right adjoint, this \icat{} has an
  initial object, corresponding to the generic-free factorization
  $X \to TY \to T\Phi$, as $Y$ also lies in $\mathcal{U}(T)$
  by Proposition~\ref{propn:WTint}(ii); hence $\mathcal{E}_{\phi}$ is
  indeed weakly contractible, as required.
\end{proof}

\begin{thm}[Nerve Theorem]\label{thm:nervecartsq}
  Suppose $T$ is a polynomial monad on $\mathcal{S}^{\mathcal{I}}$,
  and let $j_{T}$ denote the restriction of $F_{T}^{\op}$ to a functor
  $\mathcal{U}(T) \to \mathcal{W}(T)$. Then the commutative
  square
  \[
    \begin{tikzcd}
      \Alg_{T}(\mathcal{S}^{\mathcal{I}})
      \arrow{r}{\nu_{\mathcal{W}(T)}} \arrow[d, swap, "U_{T}"] &
      \Fun(\mathcal{W}(T), \mathcal{S}) \arrow{d}{j_{T}^{*}} \\
      \mathcal{S}^{\mathcal{I}} \arrow{r}{\nu_{\mathcal{U}(T)}} &
      \Fun(\mathcal{U}(T), \mathcal{S})
    \end{tikzcd}
  \]
  is cartesian, and the mate transformation
  \[ j_{T,!}\nu_{\mathcal{U}(T)} \to
    \nu_{\mathcal{W}(T)}F_{T} \] is an equivalence. In particular,
  $\nu_{\xW(T)}\colon \Alg_{T}(\mathcal{S}^{\mathcal{I}}) \to
  \Fun(\mathcal{W}(T),\mathcal{S})$ is fully faithful, and the left
  adjoint $j_{T,!}$ restricts to $F_T$.
\end{thm}
\begin{proof}
  We want to apply \cite[Proposition 5.3.5]{AnalMnd} to conclude that
  the square is cartesian. All the requirements for this are clearly
  satisfied, with one exception: We must show that the mate
  transformation
  \[ j_{T,!}\nu_{\mathcal{U}(T)} \to
    \nu_{\mathcal{W}(T)}F_{T} \] is an equivalence, \ie{} is given by
  an equivalence when evaluated at every object $\Phi\in \xS^\xI$. We
  first consider the case of
  $X \in \mathcal{U}(T)^{\op} \subseteq
  \mathcal{S}^{\mathcal{I}}$. Then $\nu_{\mathcal{U}(T)}X$ is
  the presheaf on $\mathcal{U}(T)$ represented by $X$, hence
  $j_{T,!}\nu_{\mathcal{U}(T)}X$ is represented by
  $j_{T}X \simeq F_{T}X$, and so
  $j_{T,!}\nu_{\mathcal{U}(T)}X \isoto
  \nu_{\mathcal{W}(T)}F_{T}X$, as required.

  Now let $\Phi \in \mathcal{S}^{\mathcal{I}}$ be a general object. Since $j_{T}^{*}$ detects equivalences, it
  suffices to show that the evaluation of the transformation
  \[ j_{T}^{*}j_{T,!}\nu_{\mathcal{U}(T)} \to
  j_{T}^{*}\nu_{\mathcal{W}(T)}F_{T} \simeq
  \nu_{\mathcal{U}(T)}T\]
  at $\Phi$ is an equivalence. We know from
  Lemma~\ref{lem:nerveRKE}(iii) and Proposition~\ref{propn:WTintnerve}(iii) that
  $\Phi$ is the colimit of the diagram
  $(\mathcal{U}(T)^{\op})_{/\Phi} \to
  \mathcal{S}^{\mathcal{I}}$ taking $X \to \Phi$ to $X$, and this
  colimit is preserved by the functors $\nu_{\mathcal{U}(T)}$
  and $\nu_{\mathcal{U}(T)}T$. Since $j_{T}^{*}j_{T,!}$
  preserves colimits (being itself a left adjoint), we have a
  commutative square
  \[ 
    \begin{tikzcd}
      \colim_{X \in (\mathcal{U}(T)^{\op})_{/\Phi}}
      j_{T}^{*}j_{T,!}\nu_{\mathcal{U}(T)}X \arrow[d, swap, "\wr"]
      \arrow{r} &       \colim_{X \in
        (\mathcal{U}(T)^{\op})_{/\Phi}}
      \nu_{\mathcal{U}(T)}TX \arrow{d}{\wr} \\
      j_{T}^{*}j_{T,!}\nu_{\mathcal{U}(T)}\Phi \arrow{r} & \nu_{\mathcal{U}(T)}T\Phi,
    \end{tikzcd}
  \]
  where the vertical morphisms are equivalences. Moreover, the top
  horizontal morphism is an equivalence, since it is the colimit of
  equivalences $j_{T}^{*}j_{T,!}\nu_{\mathcal{U}(T)}X \isoto
  \nu_{\mathcal{U}(T)}TX$ for $X \in
  \mathcal{U}(T)^{\op}$. The bottom horizontal morphism is
  therefore also an equivalence, which completes the proof.
\end{proof}

\begin{cor}\label{cor:AlgTpres}
  $\Alg_{T}(\mathcal{S}^{\mathcal{I}})$ is equivalent to the full
  subcategory of $\Fun(\mathcal{W}(T), \mathcal{S})$ spanned by
  functors that are local with respect to the morphisms
  \[j_{T,!}(\colim_{I \in (\mathcal{I}_{X/})^{\op}} y(I)) \to j_{T,!}y(X)\]
  for $X \in \mathcal{U}(T)$. In particular
  $\Alg_{T}(\mathcal{S})^{\mathcal{I}}$ is an accessible localization
  of $Fun(\mathcal{W}(T), \mathcal{S})$ and so a presentable \icat{}. \qed
\end{cor}

We now want to show that the \icats{} $\mathcal{U}(T)$ and
$\mathcal{W}(T)$ are compatible with morphisms of polynomial monads.

\begin{propn}\label{propn:Wmor}
  Let $T$ be a polynomial monad on $\mathcal{S}^{\mathcal{I}}$ and $S$
  a polynomial monad on $\mathcal{S}^{\mathcal{J}}$, and suppose we
  have a commutative square
  \[
    \begin{tikzcd}
      \Alg_{S}(\mathcal{S}^{\mathcal{J}}) \arrow{r}{\Phi} \arrow[d, swap, "U_{S}"] &
      \Alg_{T}(\mathcal{S}^{\mathcal{I}}) \arrow{d}{U_{T}}
      \\
      \mathcal{S}^{\mathcal{J}} \arrow{r}{f^{*}} & \mathcal{S}^{\mathcal{I}},
    \end{tikzcd}
  \]
  such that the mate transformation $F_{T}f^{*} \to \Phi F_{S}$ is
  cartesian. 
  \begin{enumerate}[(i)]
  \item If $X \to TY$ is $T$-generic, then the composite
    $f_{!}X \to f_{!}TY \to Sf_{!}Y$ is $S$-generic, where $f_{!}
    \colon \mathcal{S}^{\mathcal{I}} \to \mathcal{S}^{\mathcal{J}}$ is
    the left adjoint to $f^{*}$, given by left Kan extension along
    $f$, and the natural transformation $f_{!}T \to Sf_{!}$ is
    obtained from the mate by applying $U_{T}$ and moving adjoints
    around.
  \item The functor $f_{!}$ restricts to a functor
    $\mathcal{U}(T)^{\op} \to \mathcal{U}(S)^{\op}$.
  \item The functor $\Phi \colon \Alg_{T}(\mathcal{S}^{\mathcal{I}})
    \to \Alg_{S}(\mathcal{S}^{\mathcal{J}})$ has a left adjoint $\Psi$.
  \item The functor $\Psi$ restricts to a functor
    $\mathcal{W}(S)^{\op} \to \mathcal{W}(T)^{\op}$, and we have a
    commutative square
    \[
      \begin{tikzcd}
        \mathcal{U}(T) \arrow{r}{f_{!}^{\op}} \arrow[d, swap, "F_{T}^{\op}"] &
        \mathcal{U}(S) \arrow{d}{F_{S}^{\op}} \\
        \mathcal{W}(T) \arrow{r}{\Psi^{\op}} & \mathcal{W}(S).
      \end{tikzcd}
      \]
  \end{enumerate}
\end{propn}
\begin{proof}
  We first prove (i). Let $u$ denote the map
  $T*\simeq Tf^{*}*\to f^{*}S*$. Since the transformation
  $T_{/*}f^* \simeq Tf^{*}\to f^{*}S$ is cartesian, the functor
  $(Tf^{*})_{/*} \colon \mathcal{S}^{\mathcal{J}} \to
  \mathcal{S}^{\mathcal{I}}_{/T*}$ is equivalent to the composite
  \[ \mathcal{S}^{\mathcal{J}} \xto{S_{/*}}
\mathcal{S}^{\mathcal{J}}_{/S*} \xto{f^{*}}
\mathcal{S}^{\mathcal{J}}_{/f^{*}S*} \xto{u^{*}}
\mathcal{S}^{\mathcal{I}}_{/T*}.\] This means we have a corresponding
equivalence of left adjoints
  \[f_{!}L_{*}^T \simeq L_{*}^Sf_{!}u_{!},\] Since $X \to TY$ is
$T$-generic, the adjoint map $L_{*}^TX \to Y$ is an equivalence,
hence so is $f_{!}L^{T}_{*}X \to f_{!}Y$. But under the equivalence of
left adjoints this map $L^{S}_{*}f_{!}u_!X \isoto f_{!}Y$ is adjoint
to $f_{!}X \to f_!TY \to Sf_{!}Y$, as required.

To prove (ii), we must show that if $X$ is in
$\mathcal{U}(T)^{\op}$, so that there
is a generic morphism $I \to TX$ with $I \in \mathcal{I}^{\op}$, then
$f_{!}X$ is in $\mathcal{U}(S)^{\op}$. By (i), the composite $f(I) \simeq f_{!}I
\to f_{!}TX \to Sf_{!}X$ is $T$-generic. Since $f(I)$ is in
$\mathcal{\mathcal{J}}$, this implies that $f_{!}X$ is in
$\mathcal{U}(T)^{\op}$.

To show part (iii), note that by Corollary~\ref{cor:AlgTpres} the
\icats{} $\Alg_{T}(\mathcal{S}^{\mathcal{I}})$ and
$\Alg_{S}(\mathcal{S}^{\mathcal{J}})$ are presentable. Since $U_{S}$ detects equivalences, preserves
limits, and is accessible, and $f^{*}$ preserves both limits and
colimits, it follows that $\Phi$ is accessible and preserves
limits. By the adjoint functor theorem this implies that $\Phi$ has a
left adjoint $\Psi$, as required.

From our commutative square of right adjoints we now get an
equivalence $\Psi F_{T}\simeq F_{S}f_{!}$. By definition the \icats{}
$\mathcal{W}(T)^{\op}$ and $\mathcal{W}(S)^{\op}$ consist of free
algebras on objects of $\mathcal{U}(T)^{\op}$ and
$\mathcal{U}(S)^{\op}$, respectively, so it follows from (ii)
that $\Psi$ takes $\mathcal{W}(T)^{\op}$ to $\mathcal{W}(S)^{\op}$,
and gives the required commutative square.
\end{proof}

\begin{lemma}\label{lem WTWS}
  The functor $\Psi\colon \xW(T)^\op\to \xW(S)^\op$ of the previous proposition preserves free maps and takes morphisms which are adjoint to $T$-generic maps to morphisms which are adjoint to $S$-generic maps.
\end{lemma}
\begin{proof}
  The commutativity of the square of Proposition~\ref{propn:Wmor}.(iv)
  shows that $\Psi$ preserves free maps. Suppose
  $\alpha\colon F_TX\to F_TY$ is a morphism in $\xW(T)^\op$ which is
  adjoint to a $T$-generic map $X\to TY$, we want to see that
  $\Psi\alpha$ is adjoint to an $S$-generic morphism. By the
  equivalence $\Psi F_T\simeq F_S f_!$ of Proposition~\ref{propn:Wmor}
  and the construction of the generic--free factorization the map
  $\Psi \alpha$ is adjoint to the composite
  \[ f_{!}X \xto{\eta^{S}f_{!}} S f_{!} X \to S f_{!} Y,\]
  where $\eta^{S}$ is the unit of the monad $S$. We claim that there
  is a commutative diagram
  \[
    \begin{tikzcd}
      & f_!TX \arrow[r] \arrow[d] & f_!TY \arrow[d]\\
      f_!X \arrow[ru, "f_!\eta^T"] \arrow[r, swap, "\eta^Sf_!"] & Sf_!X \arrow[r] & Sf_!Y
    \end{tikzcd}
  \]
  where $\eta^T$ is the unit of $T$, the right horizontal maps are
  induced by $\alpha$ and the vertical maps are induced by the
  equivalence $S f_!\simeq U_S\Psi F_T$ together with the natural
  transformation $\tau\colon f_!U_T\to U_S\Psi$ adjoint to the unit
  $\eta^{S}_{f_!}\colon f_!\to Sf_!$. The square in the diagram
  commutes by naturality. To see that the triangle commutes
  we first observe that $\tau$ is also adjoint to the counit map
  $\epsilon_\Psi\colon F_S U_S\Psi\to \Psi$. Using this it is easy to
  see that left triangle is adjoint to a triangle
  \[
  \begin{tikzcd}
  F_Sf_! X \arrow[r] \arrow[rd, swap, "\id"] & F_Sf_!TX\arrow[d]\\
  & F_Sf_!X
  \end{tikzcd}
  \] 
  which is equivalent to the commutative triangle
  \[
  \begin{tikzcd}
  \Psi F_T X \arrow[r] \arrow[rd, swap, "\id"] & \Psi F_T U_T F_T X\arrow[d, "\Psi\epsilon_{F_T X}"]\\
  & \Psi F_TX
  \end{tikzcd}
  \] 
  obtained from the adjunction identities. This shows that the diagram above commutes, and hence $\Psi \alpha$
  is adjoint to the composite $f_!X\to f_{!}TX \to f_! TY \to S f_!Y$,
  which is $S$-generic by Proposition~\ref{propn:Wmor}(i).
\end{proof}

Combining the preceding results, we get the following:
\begin{cor}\label{cor morTS}
  In the situation of Proposition~\ref{propn:Wmor}, we have a
  commutative cube
  \[
    \begin{tikzcd}[row sep=small, column sep=small]
      \Alg_{S}(\mathcal{S}^{\mathcal{J}})
\arrow[hookrightarrow]{rr} \arrow{dd} \arrow{dr} & &
\Fun(\mathcal{W}(S), \mathcal{S}) \arrow{dr} \arrow{dd} \\ &
\Alg_{T}(\mathcal{S}^{\mathcal{I}}) \arrow[crossing
over,hookrightarrow]{rr} & & \Fun(\mathcal{W}(T), \mathcal{S})
\arrow{dd} \\ \mathcal{S}^{\mathcal{J}} \arrow[hookrightarrow]{rr}
\arrow{dr} & & \Fun(\mathcal{U}(S), \mathcal{S}) \arrow{dr}
\\ & \mathcal{S}^{\mathcal{I}} \arrow[hookrightarrow]{rr}
\arrow[leftarrow,crossing over]{uu} & & \Fun(\mathcal{U}(T),
\mathcal{S}),
    \end{tikzcd}
  \]
  which exhibits the morphism of polynomial monads $T \to S$ as arising from
the commutative square in Proposition~\ref{propn:Wmor}(iv). \qed
\end{cor}
\begin{proof}
  Taking left adjoints, the morphism of polynomial monads $T \to S$
  gives a commutative square
  \[
    \begin{tikzcd}
      \mathcal{S}^{\mathcal{I}} \arrow{r}{f_{!}} \arrow[d, swap, "F_{T}"] &
      \mathcal{S}^{\mathcal{J}} \arrow{d}{F_{S}} \\
      \Alg_{T}(\mathcal{S}^{\mathcal{I}}) \arrow{r}{\Psi} &
      \Alg_{S}(\mathcal{S}^{\mathcal{J}}),
    \end{tikzcd}
    \]
    and we have shown that this restricts to a commutative square
    \[
      \begin{tikzcd}
      \mathcal{U}(T)^{\op} \arrow{d} \arrow{r} & \mathcal{U}(S)^{\op} \arrow{d} \\
      \mathcal{W}(T)^{\op} \arrow{r} & \mathcal{W}(S)^{\op}
    \end{tikzcd}
  \]
   relating these full subcategories. Thus we have a commutative cube
   \[
     \begin{tikzcd}[row sep=small, column sep=small]
       \mathcal{U}(T)^{\op}
\arrow[hookrightarrow]{rr} \arrow{dd} \arrow{dr} & &
\mathcal{S}^{\mathcal{I}} \arrow{dr} \arrow{dd} \\ &
\mathcal{U}(S)^{\op} \arrow[crossing
over,hookrightarrow]{rr} & & \mathcal{S}^{\mathcal{J}}
\arrow{dd} \\ \mathcal{W}(T)^{\op} \arrow[hookrightarrow]{rr}
\arrow{dr} & & \Alg_{T}(\mathcal{S}^{\mathcal{I}}) \arrow{dr}
\\ & \mathcal{W}(S)^{\op} \arrow[hookrightarrow]{rr}
\arrow[leftarrow,crossing over]{uu} & & \Alg_{S}(\mathcal{S}^{\mathcal{J}}).
    \end{tikzcd}
  \]
  The right-hand square consists of cocomplete \icats{} and
  colimit-preserving functors, so this canonically extends to presheaves
  on the left-hand square, giving a commutative cube
  \[
    \begin{tikzcd}[row sep=small, column sep=small]
      \Fun(\mathcal{U}(T), \mathcal{S})
\arrow{rr} \arrow{dd} \arrow{dr} & &
\mathcal{S}^{\mathcal{I}} \arrow{dr} \arrow{dd} \\ &
\Fun(\mathcal{U}(S), \mathcal{S}) \arrow[crossing
over]{rr} & & \mathcal{S}^{\mathcal{J}}
\arrow{dd} \\ \Fun(\mathcal{W}(T), \mathcal{S}) \arrow{rr}
\arrow{dr} & & \Alg_{T}(\mathcal{S}^{\mathcal{I}}) \arrow{dr}
\\ & \Fun(\mathcal{W}(S), \mathcal{S}) \arrow{rr}
\arrow[leftarrow,crossing over]{uu} & & \Alg_{S}(\mathcal{S}^{\mathcal{J}})
    \end{tikzcd}
  \]
  This consists entirely of left adjoints, and passing to right
  adjoints we get the cube we want.
\end{proof}

\section{Factorization Systems from Polynomial Monads}\label{sec:Kleislifact}
Suppose $T$ is a polynomial monad on $\mathcal{S}^{\mathcal{I}}$. Then
a morphism $F_{T}X \to F_{T}Y$ in the Kleisli \icat{} $\mathcal{K}(T)$
has a canonical factorization of the form
\[ F_{T}X \to F_{T}L_{*}X \to F_{T}Y \]
adjoint to the generic-free factorization of $X \to TY$ as $X \to
TL_{*}X \to TY$ through the unit of the local left adjoint
$L_{*}$. Our first goal in this section is to show that this canonical
factorization is well-defined, in the sense that if we have
equivalences $F_{T}X \simeq F_{T}X'$, $F_{T}Y \simeq F_{T}Y'$ in
$\mathcal{K}(T)$ (which need not come from morphisms in
$\mathcal{S}^{\mathcal{I}}$), then there is a commutative diagram
\[
  \begin{tikzcd}
    F_{T}X \arrow{d}{\wr} \arrow{r} & F_{T}L_{*}X
    \arrow{r}\arrow[d, "\wr"] & F_{T}Y \arrow{d}{\wr} \\
    F_{T}X' \arrow{r} & F_{T}L_{*}X' \arrow{r} & F_{T}Y'
  \end{tikzcd}
\]
where the middle vertical map is again an equivalence. We can then say
that a morphism $\phi \colon F_{T}X \to F_{T}Y$ is
\begin{itemize}
\item \emph{inert} if the map $F_{T}X \to F_{T}L_{*}X$ in the
  canonical factorization is an equivalence,
\item \emph{active} if the map $F_{T}L_{*}X \to F_{T}Y$ in the
  canonical factorization is an equivalence,
\end{itemize}
as this does not depend on the choice of the objects $X$ and $Y$. We
will see that the inert morphisms are obtained by closing the free
morphisms under equivalences (which need not all be free), while the
active morphisms are precisely those that are adjoint to generic
morphisms. Our main goal in this section is to prove that these
classes give a factorization system:
\begin{thm}\label{Klfact}
  Let $T$ be a polynomial monad on $\mathcal{S}^{\mathcal{I}}$. Then
  the active and inert morphisms give a factorization system on
  $\mathcal{K}(T)$, whereby every morphism factors as an active
  morphism followed by an inert morphism; this factorization is
  precisely the canonical factorization, up to equivalence.
\end{thm}
This factorization system restricts to the full subcategory
$\mathcal{W}(T)^{\op}$, which induces a canonical pattern structure on
$\mathcal{W}(T)$; in the next section we will discuss how this relates
to the original monad $T$.

We start with some observations relating the local left adjoint of $T$
to the Kleisli \icat{}:
\begin{notation}
  For $X \in \mathcal{S}^{\mathcal{I}}$, we write $L_X\colon
  \mathcal{S}^{\mathcal{I}}_{/T(X)}\to \mathcal{S}^{\mathcal{I}}_{/X}$ for the
  left adjoint of the functor $T_X \colon\mathcal{S}^{\mathcal{I}}_{/X}\to \mathcal{S}^{\mathcal{I}}_{/T(X)}$
  induced by $T$.
\end{notation}

\begin{propn}\label{propn:SIXphiwc}
  Let $\mathcal{K}(T)$ denote the Kleisli category of $T$, \ie{} the
  full subcategory of $\Alg_{T}(\mathcal{S}^{\mathcal{I}})$ spanned by
  the free algebras. For $\phi \colon FY \to FX$ in $\mathcal{K}(T)$,
  the \icat{}
  \[ (\mathcal{S}^{\mathcal{I}}_{/X})_{\phi/} :=
  \mathcal{S}^{\mathcal{I}}_{/X} \times_{ \mathcal{K}(T)_{/FX} } (
  \mathcal{K}(T)_{/FX} )_{\phi/}\]
  has an initial object.
\end{propn}
\begin{proof}
  An object in this \icat{} is a morphism $f \colon Z \to X$ together
  with a commutative triangle
  \[
    \begin{tikzcd}
      FY \arrow{rr} \arrow{dr}[swap]{\phi} & &  FZ \arrow{dl}{F(f)} \\
      & FX.
    \end{tikzcd}
  \]
  This corresponds to a commutative triangle
  \[
    \begin{tikzcd}
      Y \arrow{rr} \arrow{dr}[swap]{\phi'} & &  TZ \arrow{dl}{T(f)} \\
      & TX,
    \end{tikzcd}
  \]
  which in turn corresponds to
  \[
    \begin{tikzcd}
      L_{X}Y \arrow{rr} \arrow{dr}[swap]{\phi''} & &  Z \arrow{dl}{f} \\
      & X.
    \end{tikzcd}
  \]
  Thus $L_{X}Y \xto{\phi''} X$ gives an initial object, as required.
\end{proof}

\begin{cor}
  The functor
  $F_{X} \colon \mathcal{S}^{\mathcal{I}}_{/X} \to
  \mathcal{K}(T)_{/FX}$ given by $F$ has a left adjoint
  $\mathcal{L}_{X}$, which takes $\phi \colon FY \to FX$ to the
  corresponding map $\phi'' \colon L_{X}Y \to X$.
\end{cor}
\begin{proof}
  By a standard argument the functor $F_X$ is a right adjoint if and
  only if $(\mathcal{S}^{\mathcal{I}}_{/X})_{\phi/}$ has an initial
  object, which is the statement of Proposition~\ref{propn:SIXphiwc}.
\end{proof}

\begin{remark}\label{rmk:FXfullyff}
  For $f \colon Y \to X$, the counit map $\mathcal{L}_{X}F_{X}(f) \to f$ is
  given by the commutative triangle
  \[
  \begin{tikzcd}
  L_{X}Y \arrow{rr}{\sim} \arrow{dr} & & Y \arrow{dl}{f} \\
  & X,
  \end{tikzcd}
  \]
  where the map $L_{X}Y \to Y$ is the map adjoint to the unit $Y \to
  TY$ which is an equivalence by Lemma~\ref{lem:unitmultgen}. It
  follows that $F_{X}$ is fully
  faithful, and so $\mathcal{L}_{X}$ exhibits $\mathcal{S}^{\mathcal{I}}_{/X}$ as a
  localization of $\mathcal{K}(T)_{/FX}$.
\end{remark}

\begin{remark}
  The functor
  $F_{X} \colon \mathcal{S}^{\mathcal{I}}_{/X} \to
  \mathcal{K}(T)_{/FX}$ also has a right adjoint $U_{X}$, which takes
  $\phi \colon FY \to FX$ to the morphism obtained as the pullback of
  $U\phi \colon TY \to TX$ along the unit map
  $\epsilon_{X} \colon X \to TX$. Note that the unit
  $\id \to U_{X}F_{X}$ is an equivalence since $\epsilon$ is a
  cartesian transformation, which also implies that $F_{X}$ is fully
  faithful.
\end{remark}

\begin{remark}
  For $\phi \colon FY \to FX$, the unit map
  $\phi \to F_{X}\mathcal{L}_{X}(\phi)$ is the commutative triangle
  \[
  \begin{tikzcd}
  FY \arrow{rr} \arrow[dr, swap, "\phi"] & & FL_{X}Y \arrow{dl} \\
  & FX,
  \end{tikzcd}
  \]
  \ie{} the canonical factorization of $\phi$. By naturality, this
  means we can extend any commutative triangle
  \[
    \begin{tikzcd}
      FY \arrow{rr}{\psi} \arrow{dr} & & FY' \arrow{dl} \\
       & FX
    \end{tikzcd}
  \]
  to a commutative diagram
  \[
    \begin{tikzcd}
      FY \arrow{rr}{\psi} \arrow{d} & & FY' \arrow{d} \\
      FL_{X}Y \arrow{rr}{F_{X}\mathcal{L}_{X}\psi} \arrow{dr} & & FL_{X}Y' \arrow{dl}\\
       & FX
    \end{tikzcd}
  \]
  relating the canonical factorizations of the two maps to $FX$. The
  next observations will allow us to prove that the canonical
  factorization is also natural when we vary $FX$.
\end{remark}

\begin{propn}\label{propn:Lcommsq}
  For every object $C\in \xS^\xI$ we have a commutative diagram
  \[
  \begin{tikzcd}
  \xS^\xI_{/T^2 C} \arrow[r,"\mu_{C,!}"] \arrow[d,swap, "L_{TC}"] & \xS^\xI_{/TC}  \arrow[d, "L_C"]\\
  \xS^\xI_{/T C}\arrow[r, "L_C"] &   \xS^\xI_{/C},
  \end{tikzcd}
  \]
  where the top horizontal map is given by composition with the
  component at $C$ of the multiplication $\mu\colon T^2\to T$ at $C$.
\end{propn}
\begin{proof}
It suffices to show that the diagram
  \[
\begin{tikzcd}
\xS^\xI_{/T^2 C}  & \xS^\xI_{/T C}\arrow[l,"\mu_{C}^*"] \\
\xS^\xI_{/TC}\arrow[u,"T_{TC}"] & \xS^\xI_{/C} \arrow[l, "T_C"] \arrow[u, "T_C"]
\end{tikzcd}
\]
of the corresponding right adjoints commutes. Given an object $\alpha\colon B\to C$ in $\xS^\xI_{/C}$, its image under the composite of the right vertical and the upper horizontal map is the left vertical map of the pullback square
\[
\begin{tikzcd}
A \arrow[r] \arrow[d]& TB \arrow[d, "T\alpha"] \\
T^2C \arrow[r, "\mu_C"] & TC.
\end{tikzcd}
\]
Since the multiplication $\mu$ of the polynomial monad $T$ is a cartesian natural transformation, the map $A\to T^2C$ can be identified with $T^2 \alpha\colon T^2B\to T^2C$ which is the same as the image of $\alpha$ under the composite of $T_C$ and $T_{TC}$.
\end{proof}

\begin{corollary}\label{cor LL}
  Given morphisms $\alpha \colon A \to TB$ and $\beta \colon B \to
TC$, we have adjoint morphisms $L_BA \to B$ and $L_C B \to
C$; we also have the composites $A \xto{\alpha} TB \xto{T\beta}
T^{2}C \xto{\mu} TC$ and $L_BA \to B \to TC$  with adjoints
$L_CA \to C$ and $L_CL_BA \to C$. These are
equivalent, \ie{} $L_C A \simeq L_C L_BA$.
\end{corollary}
\begin{proof}
  Consider the diagram
  \[
  \begin{tikzcd}
  \xS^\xI_{/TB} \arrow[r, "(T\beta)_!"] \arrow[d, swap, "L_B"]& \xS^\xI_{/T^2 C} \arrow[r, "\mu_{C,!}",]\arrow[d, "L_{TC}"] & \xS^\xI_{/TC} \arrow[d, "L_C"]\\
  \xS^\xI_{/B} \arrow[r, "\beta_!"]& \xS^\xI_{/TC}\arrow[r, "L_C"] & \xS^\xI_{/C}.
  \end{tikzcd}
  \]
  Here the left square commutes since it is the square of left
  adjoints corresponding to the square
  \[
    \begin{tikzcd}
      \xS^\xI_{/TC} \arrow{r}{\beta^{*}} \arrow{d}{T_{TC}} & \xS^\xI_{/B} \arrow{d}{T_{B}} \\
      \xS^\xI_{/T^2 C} \arrow{r}{(T\beta)^{*}} & \xS^\xI_{/TB},
    \end{tikzcd}
    \]
    which commutes since $T$ preserves pullbacks, and the right square
    commutes by Proposition~\ref{propn:Lcommsq}.  By construction the
    morphisms $L_CA\to C$ and $L_CL_BA\to C$ are given by
    $L_C\mu_{C,!}T\beta_!(\alpha)$ and $L_C \beta_! L_{B}(\alpha)$,
    and so are equivalent by the commutativity of the outer square.
\end{proof}

\begin{propn}\label{propn:canfactfunct}
  Given a commutative square
  \[
    \begin{tikzcd}
      F_{T}A \arrow{r} \arrow{d} & F_{T}B \arrow{d} \\
      F_{T}C \arrow{r} & F_{T}D,
    \end{tikzcd}
  \]
  in $\mathcal{K}(T)$, there exists a canonical commutative diagram
  \[
    \begin{tikzcd}
      F_{T}A \arrow{r} \arrow{d} & F_{T}L_{B}A \arrow{r} \arrow{d}  &
      F_{T}B \arrow{d} \\
      F_{T}C \arrow{r}  & F_{T}L_{D}C \arrow{r} & F_{T}D.
    \end{tikzcd}
  \]
   Here the diagram associated to the degenerate square
  \[
    \begin{tikzcd}
      F_{T}A \arrow[equals]{d} \arrow{r} & F_{T}B \arrow[equals]{d} \\
      F_{T}A \arrow{r} & F_{T}B
    \end{tikzcd}
  \]
  is the degenerate diagram
  \[
    \begin{tikzcd}
      F_{T}A \arrow[equals]{d} \arrow{r} & F_{T}L_{B}A
      \arrow[equals]{d} \arrow{r} & F_{T}B \arrow[equals]{d} \\
      F_{T}A \arrow{r} & F_{T}L_{B}A \arrow{r} & F_{T}B.
    \end{tikzcd}
  \]
  Moreover, we have compatibility with composition, in the sense that if we have a commutative diagram
  \[
    \begin{tikzcd}
      F_{T}A \arrow{r} \arrow{d} & F_{T}B \arrow{d} \\
      F_{T}C  \arrow{d} \arrow{r} & F_{T}D \arrow{d} \\
      F_{T}X \arrow{r} & F_{T}Y,
    \end{tikzcd}
  \]
  then the vertical composite of the associated diagrams
  \[
    \begin{tikzcd}
      F_{T}A \arrow{r} \arrow{d} & F_{T}L_{B}A \arrow{r} \arrow{d} &
      F_{T}B \arrow{d} \\
      F_{T}C  \arrow{d} \arrow{r} & F_{T}L_{D}C \arrow{r} \arrow{d} &
      F_{T}D \arrow{d} \\
      F_{T}X \arrow{r} & F_{T}L_{Y}X \arrow{r} & F_{T}Y,
    \end{tikzcd}
  \]
  is the diagram associated to the composite square
  \[
    \begin{tikzcd}
      F_{T}A \arrow{r} \arrow{d} & F_{T}B \arrow{d} \\
      F_{T}X \arrow{r} & F_{T}Y.
    \end{tikzcd}
  \]
\end{propn}
\begin{proof}
  We can view the original square as a pair of morphisms
  \[
    \begin{tikzcd}
F_{T}C & F_{T}A \arrow{l} \arrow{r} & F_{T}B     
    \end{tikzcd}
  \]
  in $\mathcal{K}(T)_{/F_{T}D}$. Adding the canonical factorization of
  the arrow $F_{T}A \to F_{T}B$, the naturality of the unit for the
  adjunction $\mathcal{L}_{D} \dashv F_{D}$ gives a commutative
  diagram
  \[
    \begin{tikzcd}
    F_{T}C  \arrow{d} & F_{T}A \arrow{l} \arrow{r} \arrow{d} &
    F_{T}L_{B}A \arrow{d} \arrow{r} & F_{T}B \arrow{d} \\
    F_{T}L_{D}C & F_{T}L_{D}A \arrow{l} \arrow{r}{\sim} & F_{T}L_{D}L_{B}A
    \arrow{r} & F_{T}L_{D}B
    \end{tikzcd}
  \]
  over $F_{T}D$, where the second arrow in the bottom row is an
  equivalence by Corollary~\ref{cor LL}. If we invert this equivalence
  we can contract the diagram to
  \[
    \begin{tikzcd}
      F_{T}A \arrow{r} \arrow{d} & F_{T}L_{B}A \arrow{r} \arrow{d} &
      F_{T}B \\
      F_{T}C \arrow{r} & F_{T}L_{D}C
    \end{tikzcd}
  \]
  over $F_{T}D$; adding $F_{T}D$ back in now gives the desired
  diagram.

  From the degenerate square we can make the extended diagram
  \[
    \begin{tikzcd}[column sep=large]
      F_{T}A  \arrow{d} & F_{T}A \arrow[equals]{l} \arrow{r}{\eta_{F_{T}A}} \arrow{d}[left]{\eta_{F_{T}A}}
      & F_{T}L_{B}A \arrow{d}[left]{\eta_{F_{T}L_{B}A}} \arrow[bend
      left=65,equals]{dd} \arrow{r} & F_{T}B \arrow{d}  \arrow[equals,bend left=50]{dd}\\
      F_{T}L_{B}A & F_{T}L_{B}A \arrow[equals]{l}
      \ar[r, "F_{B}\mathcal{L}_{B}\eta_{F_{T}A}"{above}, "\sim"{below}] \arrow[bend right,equals]{dr} & F_{T}L_{B}L_{B}A
        \arrow[crossing over]{r} \arrow{d}[left]{F_{T}\epsilon_{L_{B}A}} & F_{T}L_{B}B
          \arrow{d}[left]{F_{T}\epsilon_{L_{B}}} \\
          & & F_{T}L_{B}A \arrow{r} & F_{T}B,
    \end{tikzcd}
  \]
  all over $F_{T}B$. Here the adjunction identities for $\mathcal{L}_{B} \dashv F_{B}$
  imply that the two composites $F_{T}L_{B}A \to F_{T}L_{B}A$ are
  identities, as indicated; this means the general definition indeed
  specializes to give the degenerate diagram in this case.
    
  To see we have compatibility with
  composition, consider the diagram
  \[
    \begin{tikzcd}
      F_{T}A \arrow{d} \arrow{r} & F_{T}L_{B}A \arrow{r} \arrow{d} & F_{T}B \arrow{d}\\
      F_{T}C \arrow{d} \arrow{r} & F_{T}L_{D}C \arrow{r} & F_{T}D \\
      F_{T}X
    \end{tikzcd}
  \]
  in $\mathcal{K}(T)_{/F_{T}Y}$. Using the unit for the adjunction
  $\mathcal{L}_{Y}\dashv F_{Y}$ this extends to a commutative diagram
  \[
    \begin{tikzcd}[column sep=tiny,row sep=small]
      F_{T}A \arrow{ddd} \arrow{dr} \arrow{rrr} & & &  F_{T}L_{B}A \arrow{rrr}\arrow{ddd}
      \arrow{dr} && & F_{T}B \arrow{dr} \arrow{ddd}\\
      & F_{T}C \arrow{dr} \arrow[crossing over]{rrr} & & &
      F_{T}L_{D}C \arrow[crossing over]{rrr} &&&  F_{T}D \arrow{ddd}\\
      & & F_{T}X \\
      F_{T}L_{Y}A \arrow{dr} \arrow{rrr}{\sim} & & &  F_{T}L_{Y}L_{B}A \arrow{rrr}
      \arrow{dr} && & F_{T}L_{Y}B \arrow{dr} \\
      & F_{T}L_{Y}C  \arrow[crossing over,leftarrow]{uuu}\arrow{dr}
      \arrow{rrr}{\sim} & & & F_{T}L_{Y}L_{D}C \arrow[crossing over,leftarrow]{uuu}
      \arrow{rrr} &&&  F_{T}L_{Y}D \\
      & & F_{T}L_{Y}X\arrow[crossing over,leftarrow]{uuu}
    \end{tikzcd}
  \]
  over $F_{T}Y$, where the two indicated morphisms are equivalences by
  Corollary~\ref{cor LL}. Inverting these, we can contract the diagram
  to
  \[
    \begin{tikzcd}
      F_{T}A \arrow{rr}\arrow{d} & & F_{T}L_{B}A  \arrow{rr} \arrow{d} \arrow{ddl}
      & & F_{T}B \arrow{d}\\
      F_{T}C \arrow[crossing over]{rr}  \arrow{dd} & &  F_{T}L_{D}C \arrow{rr}
      \arrow{d} & &F_{T}D \arrow{dd} \\
      & F_{T}L_{Y}A \arrow{r} \arrow{dr} & F_{T}L_{Y}C  \arrow{d} \\
      F_{T}X \arrow{rr} & & F_{T}L_{Y}X \arrow{rr} & & F_{T}Y
    \end{tikzcd}
    \]
  where we see both the composite of the diagrams for the two squares
  and the diagram for the composite square, as required.  
\end{proof}

\begin{cor}
  A commutative square
  \[
    \begin{tikzcd}
      F_{T}A \arrow{r} \arrow[d, swap, "\wr"] & F_{T}B \arrow{d}{\wr} \\
      F_{T}C \arrow{r} & F_{T}D
    \end{tikzcd}
  \]
  in $\mathcal{K}(T)$, where the vertical morphisms are equivalences,
  can be extended to a commutative diagram
  \[
    \begin{tikzcd}
      F_{T}A \arrow{r} \arrow[d, swap, "\wr"] & F_{T}L_{B}A \arrow{r} \arrow{d}{\wr} & F_{T}B \arrow{d}{\wr} \\
      F_{T}C \arrow{r} &  F_{T}L_{D}C \arrow{r} &  F_{T}D
    \end{tikzcd}
  \]
  where the middle vertical map is also an equivalence.
\end{cor}
\begin{proof}
  This follows immediately from the compatibility of the diagrams in
  Proposition~\ref{propn:canfactfunct} with composition and identities.
\end{proof}
In other words, the canonical factorizations in $\mathcal{K}(T)$ are
invariant under equivalences. This means the following conditions on
morphisms are well-defined:
\begin{defn}
  We say a morphism $\phi \colon F_{T}A \to F_{T}B$ in
  $\mathcal{K}(T)$ with canonical factorization
  \[ F_{T}A \to F_{T}L_{B}A \to F_{T}B\]
  is \emph{inert} if the morphism $F_{T}A \to F_{T}L_{B}A$ in the
  canonical factorization is an equivalence, and \emph{active} if the
  morphism $F_{T}L_{B}A \to F_{T}B$ in the canonical factorization is
  an equivalence.
\end{defn}
Our goal in the rest of this section is to show that the active and
inert morphisms form a factorization system on $\mathcal{K}(T)$. We
start with some observations about equivalences in $\mathcal{K}(T)$
that will lead to a simpler characterization of the active maps.

\begin{lemma}\label{lem Fconservative}
  If $T$ is a polynomial monad then the free functor $F_{T}$ is
  conservative: if $F_{T}(\phi)$ is an equivalence for some
  $\phi \colon X \to Y$ in $\mathcal{S}^{\mathcal{I}}$ then $\phi$ is
  an equivalence.
\end{lemma}
\begin{proof}
  The functor $F_{Y}\colon
  \mathcal{S}^{\mathcal{I}}_{/Y} \to \mathcal{K}(T)_{/F_{T}Y}$ is
  fully faithful by Remark~\ref{rmk:FXfullyff}. The inverse of
  $F_{T}\phi$ gives a morphism
  \[
    \begin{tikzcd}
      F_{T}Y \arrow[equals]{dr} \arrow{rr}{(F_{T}\phi)^{-1}}  & &
      F_{T}X  \arrow{dl}{F_{T}\phi}\\
       & F_{T}Y
    \end{tikzcd}
    \]
  in $\mathcal{K}(T)_{/F_{T}Y}$ between objects in the image of
  $F_{Y}$, hence it is also in the image of $F_{Y}$ and lifts to an
  equivalence in $\mathcal{S}^{\mathcal{I}}_{/Y}$ by faithfulness.
\end{proof}

\begin{lemma}\label{lem:freeunique}
  Given a commutative triangle
  \[
  \begin{tikzcd}
  FA \arrow{rr}{\phi} \arrow{dr}[swap]{\alpha} & & FB \arrow{dl}{\beta}
  \\
  & FC,
  \end{tikzcd}
  \]
  and morphisms $a \colon A \to C$ and $b \colon B \to C$ with
  equivalences $\alpha \simeq F(a)$ and $\beta \simeq F(b)$, then the
  triangle lifts to a unique commutative triangle
  \[
  \begin{tikzcd}
  A \arrow{rr}{f} \arrow{dr}[swap]{a} & & B \arrow{dl}{b}
  \\
  & C.
  \end{tikzcd}
  \]
\end{lemma}
\begin{proof}
  The functor
  $F_{C}\colon \mathcal{S}^{\mathcal{I}}_{/C} \to
  \mathcal{K}(T)_{/F_{T}C}$ is fully faithful by
  Remark~\ref{rmk:FXfullyff}. This immediately implies the result,
  since the first triangle is precisely a morphism in
  $\mathcal{K}(T)_{/F_{T}C}$ between objects in the image of $F_{C}$.
\end{proof}

\begin{lemma}\label{lem equfree}
  Every equivalence $\phi\colon F_TX\isoto F_TY$ is adjoint to a
  generic map. 
\end{lemma}
\begin{proof}
  Regarding $\phi$ as a morphism in $\mathcal{K}(T)_{/F_{T}Y}$, we
  apply $\mathcal{L}_{Y}$ to get a commutative triangle
  \[
    \begin{tikzcd}
      L_{Y}X \ar[rr, "\mathcal{L}_{Y}\phi"{above}, "\sim"{below}]
      \arrow{dr} & & L_{Y}Y
      \arrow{dl}{\sim} \\
      & Y,
    \end{tikzcd}
  \]
  where the right diagonal map is an equivalence by
  Lemma~\ref{lem:unitmultgen}(i) and the horizontal map is an
  equivalence by the functoriality of $\mathcal{L}_{Y}$. Hence the
  left diagonal map is also an equivalence, which is precisely the
  condition for the map $X \to TY$ adjoint to $\phi$ to be generic.
\end{proof}

\begin{lemma}\label{lem:actgen}
  A morphism $\phi \colon F_{T}X \to F_{T}Y$ is active \IFF{} the adjoint morphism
  $\phi' \colon X \to TY$ is generic.
\end{lemma}
\begin{proof}
  Let $\phi'' \colon L_{Y}X \to Y$ be the map adjoint to $\phi'$.
  By definition, $\phi$ is active \IFF{} in the canonical
  factorization
  \[ F_{T}X \xto{\phi_{a}} F_{T}L_{Y}X \xto{\phi_{i}} F_{T}Y,\] the
  map $\phi_{i} = F_{T}(\phi'')$ is an equivalence. By Lemma~\ref{lem
    Fconservative} this happens \IFF{} $\phi''$ is an equivalence,
  which is precisely the condition for $\phi'$ to be generic.
\end{proof}

\begin{lemma}\label{lem:freeint}
  For any morphism $\phi \colon X \to Y$ in
  $\mathcal{S}^{\mathcal{I}}$, the free morphism $F_{T}(\phi) \colon F_{T}X
  \to F_{T}Y$ is inert.
\end{lemma}
\begin{proof}
  The commutative triangle 
  \[
  \begin{tikzcd}
  F_{T}X \arrow[d, equal] \arrow[rd, "F_T\phi"] & \\
  F_{T}X \arrow[r, "F_T\phi"{below}] & F_{T}Y
  \end{tikzcd}
  \]
  is adjoint to 
  \[
  \begin{tikzcd}
  X \arrow[d, "\eta_X", swap] \arrow[rd] & \\
  TX \arrow[r, "T\phi"] & TY
  \end{tikzcd}
  \]
  which is in turn adjoint to 
  \[
  \begin{tikzcd}
  L_*X \arrow[d, "\lambda" swap, "\wr"] \arrow[rd] & \\
  X \arrow[r, "\phi"] & Y,
  \end{tikzcd}
  \]
  where $\lambda$ is an equivalence since the unit map $\eta_X$ is generic by Lemma~\ref{lem:unitmultgen}(i).
  By adjointness we see that $F_T$ takes the last triangle to the right triangle in the commutative diagram
  \[
  \begin{tikzcd}
  F_{T}X \arrow[r] \arrow[rd, equal] & F_{T}L_*X \arrow[d, "F_{T}\lambda" swap, "\wr"] \arrow[rd] & \\
  & F_{T}X \arrow[r, "F_{T}\phi"] & F_{T}Y,
  \end{tikzcd}
  \]
  where the upper horizontal map is adjoint to the unit $X\to TL_*X$;
  this is an equivalence as $F_{T}\lambda$ is one. By definition the upper horizontal map and the
  right diagonal map give the canonical factorization of $F_{T}\phi$,
  hence $F_{T}\phi$ is inert.
\end{proof}

\begin{warning}\label{warn equgen}
  Note, however, that it is \emph{not} necessarily true that every
  inert map is of the form $F_{T}(\phi)$ for $\phi$ a morphism in
  $\mathcal{S}^{\mathcal{I}}$: The equivalences in
  $\Alg_{T}(\mathcal{S}^{\mathcal{I}})$ need not all be in the image of
  $\mathcal{S}^{\mathcal{I}}$.
\end{warning}

\begin{remark}\label{rmk:canfact}
  By Lemmas~\ref{lem:actgen} and \ref{lem:freeint}, the canonical
  factorization of a morphism $\phi \colon F_{T}A \to F_{T}B$ as $F_{T}A \to F_{T}L_{*}A \to F_{T}B$
  is a factorization of $\phi$ as an active morphism followed by an
  inert morphism.
\end{remark}

\begin{lemma}\label{lem:Flift}
  Given a commutative square
  \[
    \begin{tikzcd}
      F_{T}A \arrow{r}{\phi} \arrow[d,swap, "\psi"] & F_{T}B
      \arrow{d}{F_{T}\beta} \\
      F_{T}C \arrow{r}{F_{T}\gamma} \arrow[dashed]{ur} & F_{T}D
    \end{tikzcd}
  \]
  where $\psi$ is active, there exists a unique diagonal filler, which
  is of the form $F_{T}(\alpha)$ for a unique commutative triangle
  \[
    \begin{tikzcd}
      C \arrow{rr}{\alpha} \arrow{dr}[swap]{\gamma} & & B \arrow{dl}{\beta}
      \\
       & D.
    \end{tikzcd}
  \]
\end{lemma}
\begin{proof}
  As in Lemma~\ref{lem:genericfillX}, the square is adjoint to
  \[
    \begin{tikzcd}
      A \arrow{r}{\phi'} \arrow[d,swap, "\psi'"] & TB
      \arrow{d}{T\beta} \\
      TC \arrow{r}{T\gamma}  & TD,
    \end{tikzcd}
  \]
  which in turn is adjoint to
  \[
    \begin{tikzcd}
      L_{*}A \arrow{r}{\phi''} \arrow[d,swap, "\wr"] & B
      \arrow{d}{\beta} \\
      C \arrow{r}{\gamma}  & D,
    \end{tikzcd}
  \]
  where the left vertical morphism is an equivalence since $\psi$ is
  active and this implies that $\psi'$ is generic by
  Lemma~\ref{lem:actgen}. This square has a unique filler, which in
  turn corresponds to a unique filler in the original square, since we
  saw in Lemma~\ref{lem:freeunique} that all fillers are uniquely of
  this form.
\end{proof}

\begin{proof}[Proof of Theorem~\ref{Klfact}]
  We check the requirements of \cite[Definition 5.2.8.8]{ht} (which
  are equivalent to our previous definition of a factorization system
  by \cite[Proposition 5.2.8.17]{ht}). We must thus check:
  \begin{enumerate}[(1)]
  \item The classes of inert and active maps are closed under
    retracts.
  \item The active maps are left orthogonal to the inert maps, \ie{} for every commutative square
    \[
      \begin{tikzcd}
        F_{T}A \arrow{r}{\alpha} \arrow[d,swap, "\beta"] & F_{T}B \arrow{d}{\gamma} \\
        F_{T}C \arrow{r}{\delta} \arrow[dashed]{ur} & F_{T}D
      \end{tikzcd}
    \]
    where $\beta$ is active and $\gamma$ is inert, there exists a unique
    filler.
  \item Every morphism can be factored as an active map followed by an
    inert map. 
  \end{enumerate}
  Condition (3) is by now clear, since by Remark~\ref{rmk:canfact} the canonical factorization gives an
  active-inert factorization. For condition (1), suppose we have a retract diagram
  \[
    \begin{tikzcd}
      F_{T}A \arrow{r} \arrow[d, swap, "\phi"] \arrow[equals,bend left]{rr} & F_{T}A'
      \arrow{r} \arrow{d}{\psi} & F_{T}A \arrow{d}{\phi} \\
      F_{T}B \arrow{r} \arrow[equals,bend right]{rr} & F_{T}B' \arrow{r} & F_{T}B.
    \end{tikzcd}
  \]
  By applying Proposition~\ref{propn:canfactfunct}  to the two squares,
  we obtain a commutative diagram
  \[
    \begin{tikzcd}
      F_{T}A \arrow{r} \arrow[d, swap, "\phi_a"] \arrow[equals,bend left]{rr} & F_{T}A'
      \arrow{r} \arrow{d}{\psi_{a}} & F_{T}A \arrow{d}{\phi_a} \\
      F_{T}L_{*}A \arrow[r, "f"] \arrow[d, swap, "\phi_i"] & F_{T}L_{*}A' \arrow[r, "g"] \arrow{d}{\psi_{i}} &
      F_{T}L_{*}A \arrow{d}{\phi_i} \\
      F_{T}B \arrow{r} \arrow[equals,bend right]{rr} & F_{T}B' \arrow{r} & F_{T}B,
    \end{tikzcd}
  \]
  relating the canonical factorizations of $\phi$ and $\psi$, where
  the compatibility with composition and identities in
  Proposition~\ref{propn:canfactfunct} implies that $gf \simeq
  \id$. If $\psi$ is active then by definition the map labelled
  $\psi_{i}$ is an equivalence, and so $\phi_{i}$ is a retract of an
  equivalence; hence $\phi_{i}$ is also an equivalence, which means
  $\phi$ is active. The same argument shows that inert morphisms are
  also closed under retracts.
  
  It remains to prove (2). Consider a commutative square
  \[
    \begin{tikzcd}
      F_{T}A \arrow{r}{\alpha} \arrow[d,swap, "\beta"] & F_{T}B \arrow{d}{\gamma} \\
      F_{T}C \arrow{r}{\delta} & F_{T}D
    \end{tikzcd}
  \]
  with $\beta$ active and $\gamma$ inert. Including the canonical
  factorizations of $\gamma$ and $\delta$, we get a diagram
  \[
    \begin{tikzcd}
      F_{T}A \arrow{rr}{\alpha} \arrow[dd,swap, "\beta"] & & F_{T}B
      \arrow{d}{\wr} \\
      & & F_{T}L_{*}B \arrow{d} \\
      F_{T}C \arrow{r} \arrow[dashed]{urr} & F_{T}L_{*}C \arrow{r}
      \arrow[dashed]{ur} & F_{T}D,
    \end{tikzcd}
  \]
  and since the map $F_{T}B \to F_{T}L_{*}B$ is an equivalence
  (since $\gamma$ is inert), a lift in the original square
  corresponds to a lift $F_{T}C \to F_{T}L_{*}B$ here. Applying
  Lemma~\ref{lem:Flift} to the square
  \[
    \begin{tikzcd}
      F_{T}A \arrow{r} \arrow{d} & F_{T} L_{*}B \arrow{d} \\
      F_{T} L_{*}C \arrow{r} \arrow[dashed]{ur} & F_{T}D,
    \end{tikzcd}
  \]
  we see that there is a unique diagonal filler $F_{T}L_{*}C \to
  F_{T}L_{*}B$, which comes from a unique commutative triangle
  \[
    \begin{tikzcd}
      L_{*}C \arrow{rr} \arrow{dr} & & L_{*}B \arrow{dl} \\
      & D.
    \end{tikzcd}
  \]
  This gives in particular a lift in the original square, but now
  applying Lemma~\ref{lem:Flift} to a square
  \[
    \begin{tikzcd}
      F_{T}C \arrow{r} \arrow{d} & F_{T} L_{*}B \arrow{d} \\
      F_{T} L_{*}C \arrow{r} \arrow[dashed]{ur} & F_{T}D,
    \end{tikzcd}
  \]
  we see that any lift $F_{T}C \to F_{T}L_{*}B$ must factor
  through $F_{T}L_{*}C$ and so must be the lift we just
  constructed.
\end{proof}

Let us say that a morphism in $\mathcal{W}(T)$ is inert or active if
it corresponds to an inert or active morphisms in $\mathcal{K}(T)$
under the inclusion $\mathcal{W}(T)^{\op} \hookrightarrow
\mathcal{K}(T)$. Then the factorization system we constructed
restricts to one on $\mathcal{W}(T)$:
\begin{cor}\label{cor intactfact}
  The inert and active morphisms restrict to a factorization system on
  $\mathcal{W}(T)$.
\end{cor}
\begin{proof}
  It is enough to show that for a morphism $F_{T}I \to F_{T}J$ in
  $\mathcal{W}(T)^{\op}$, if $F_{T}I \to F_{T}X \to F_{T}J$ is its
  active-inert factorization in $\mathcal{K}(T)$, then $F_{T}X$ also lies in
  $\mathcal{W}(T)^{\op}$. Since the canonical factorization is an
  active-inert factorization, this follows from
  Proposition~\ref{propn:WTint}(ii).
\end{proof}

\section{Patterns from Polynomial Monads}
\label{sec:poly1}

Suppose $T$ is a polynomial monad on $\mathcal{S}^{\mathcal{I}}$. In
the previous section we saw that the \icat{} $\mathcal{W}(T)$ has a
canonical inert--active factorization system. Using this we can define
a natural algebraic pattern structure on $\mathcal{W}(T)$ by taking
the elementary objects to be those of the form $F_{T}(I)$ with
$I \in \mathcal{S}^{\mathcal{I}}$ in the image of $\mathcal{I}^{\op}$
under the Yoneda embedding.

In this section we will study these algebraic patterns. We will see
that $\mathcal{W}(T)$ is always an extendable pattern, and that the
free Segal $\mathcal{W}(T)$-space monad is closely related to
the original monad $T$: there is a canonical morphism $T \to
T_{\mathcal{W}(T)}$ in $\PolyMnd$, which induces an equivalence on
\icats{} of algebras. 
Moreover, the
patterns $\mathcal{W}(T)$ are natural in $T$, and so determine 
a functor
\[ \mathfrak{P} \colon \PolyMnd \to \AlgPattSE;\]
the morphisms $T \to T_{\mathcal{W}(T)}$ then give a natural
transformation $\id \to \mathfrak{M}\mathfrak{P}$.

\begin{notation}\label{not:WTftrs}
  In the first part of this section we fix a polynomial monad $T$ on
  $\mathcal{S}^{\mathcal{I}}$. From our work in the previous two
  sections we then have the following commutative diagram, where it
  will be convenient to name the various functors as indicated:
  \[
    \begin{tikzcd}
      \mathcal{I} \arrow{r}{e} \arrow[hookrightarrow, d, swap, "i"] & \mathcal{W}(T)^{\el}
\arrow[hookrightarrow]{d}{i'} \\
\mathcal{U}(T) \arrow{r}{u} \arrow[dr, swap, "j"] & \mathcal{W}(T)^{\xint}
\arrow{d}{j'} \\
& \mathcal{W}(T).
\end{tikzcd}
\]
\end{notation}

\begin{propn}
  Let $\mathcal{K}(T)^{\xint}$ denote the subcategory of
  $\mathcal{K}(T)$ containing only the inert morphisms. Then the slice
  $\mathcal{K}(T)^{\xint}_{/F_{T}X}$ is equivalent to the full subcategory
  of $\mathcal{K}(T)_{/F_{T}X}$ spanned by the inert morphisms to $FX$.
  The functor $F_{X} \colon \mathcal{S}^{\mathcal{I}}_{/X} \to
  \mathcal{K}(T)_{/F_{T}X}$ restricts to an equivalence
  \[ \mathcal{S}^{\mathcal{I}}_{/X} \isoto
  \mathcal{K}(T)^{\xint}_{/F_{T}X} \]
  with inverse $\mathcal{L}_{X}$.
\end{propn}
\begin{proof}
  It follows from the existence of the active--inert factorization
  system on $\mathcal{K}(T)$ that if we have a commutative triangle
  \[
    \begin{tikzcd}
      F_{T}A \arrow{dr} \arrow{rr} & & F_{T}B \arrow{dl} \\
       & F_{T}X
    \end{tikzcd}
  \]
  where the two diagonal morphisms are inert, then the horizontal
  morphism is also inert. This implies that
  $\mathcal{K}(T)_{/F_{T}X}^{\xint}$ is the full subcategory of
  $\mathcal{K}(T)_{/F_{T}X}$ spanned by the inert morphisms. Moreover,
  every inert morphism to $F_{T}X$ is equivalent to a free morphism in
  $\mathcal{K}(T)_{/F_{T}X}$, so this full subcategory consists
  precisely of the objects in the image of $F_{X}$. Since $F_{X}$ is
  fully faithful by Remark~\ref{rmk:FXfullyff}, it follows that the
  adjunction $\mathcal{L}_{X} \dashv F_{X}$ restricts to an
  equivalence between $\mathcal{S}^{\mathcal{I}}_{/X}$ and
  $\mathcal{K}(T)_{/F_{T}X}^{\xint}$.  
\end{proof}

Restricting this equivalence, we get the following:
\begin{cor}\label{cor:sliceeq}
  The functor $F_{X}^{\op}$ restricts to an equivalence
  \[ \mathcal{I}_{X/} \isoto \mathcal{W}(T)^{\el}_{F_{T}X/}\]
  for every $X \in \mathcal{U}(T)$. \qed
\end{cor}

\begin{cor}\label{cor:elintcartsq}
  The top commutative square in Notation~\ref{not:WTftrs} induces a
  commutative square of functors to $\mathcal{S}$. Taking the mate of this
  square gives a commutative square
  \[
  \begin{tikzcd}
  \Fun(\mathcal{W}(T)^{\el}, \mathcal{S}) \arrow{r}{i'_{*}}  \arrow[d,swap,"e^{*}"]&
  \Fun(\mathcal{W}(T)^{\xint}, \mathcal{S}) \arrow{d}{u^{*}} \\
  \Fun(\mathcal{I}, \mathcal{S}) \arrow{r}{i_{*}} &
  \Fun(\mathcal{U}(T), \mathcal{S}),
  \end{tikzcd}
  \]
  and this is moreover cartesian.
\end{cor}
\begin{proof}
  To see that there is such a commutative square amounts to checking
  that the mate transformation
  \[ u^{*}i'_{*}\Phi \to i_{*}e^{*}\Phi \]
  is an equivalence for $\Phi \colon \mathcal{W}(T)^{\el}\to
  \mathcal{S}$. Evaluated at $X \in \mathcal{U}(T)$, this is the map
  on limits
  \[ \lim_{\mathcal{W}(T)^{\el}_{F_{T}X/}} \Phi \to
  \lim_{\mathcal{I}_{X/}} \Phi e,\]
  induced by the functor $\mathcal{I}_{X/} \to
  \mathcal{W}(T)^{\el}_{F_{T}X/}$. Since this functor is an equivalence by
  Corollary~\ref{cor:sliceeq}, the mate transformation is indeed an
  equivalence. The functors $i_{*}$ and $i'_{*}$ are fully
  faithful, since they are given by right Kan extensions along the fully faithful
  functors $i$ and $i'$. To see that the square is cartesian
  it therefore suffices to check that an object $\Phi \in
  \Fun(\mathcal{W}(T)^{\xint}, \mathcal{S})$ is in the image of
  $i'_{*}$ \IFF{} $u^{*}\Phi$ is in the image of $i_{*}$.
  Here $\Phi$ is in the image of $i'_{*}$ \IFF{} the unit map $\Phi
  \to i'_{*}i'^{*}\Phi$ is an equivalence.
  The functor $u^{*}$ is conservative, because $u$ is essentially
  surjective, and so this holds
  \IFF{} $u^{*}\Phi \to u^{*}i'_{*}i'^{*}\Phi$ is an equivalence. We
  can identify the composite
  \[ u^{*}\Phi \to u^{*}i'_{*}i'^{*}\Phi \isoto i_{*}e^{*}i'^{*}\Phi
    \simeq i_{*}i^{*}u^{*}\Phi \]
  with the unit map for $i^{*} \dashv i_{*}$, and since the mate
  transformation is an equivalence this means that the latter is an
  equivalence \IFF{} $\Phi$ is in the image of $i'_{*}$. As $i_{*}$ is
  also fully faithful, this condition holds precisely when $u^{*}\Phi$
  is in the image of $i_{*}$, as required.
\end{proof}

\begin{cor}\label{cor morAlgFun}
  We have a commutative diagram
  \[
  \begin{tikzcd}
  \Alg_{T}(\mathcal{S}^{\mathcal{I}}) \arrow[hookrightarrow]{r}{\nu_{\mathcal{W}(T)}}
  \arrow{d} \arrow[bend right=70]{dd}[left]{U_{T}} &
  \Fun(\mathcal{W}(T), \mathcal{S}) \arrow{d}{j'^{*}} \arrow[bend left=70]{dd}{j^{*}} \\
  \Fun(\mathcal{W}(T)^{\el}, \mathcal{S}) \arrow[d,swap,"e^{*}"] \arrow[hookrightarrow]{r}{i'_{*}} &
  \Fun(\mathcal{W}(T)^{\xint}, \mathcal{S}) \arrow{d}{u^{*}} \\
  \Fun(\mathcal{I}, \mathcal{S}) \arrow[hookrightarrow]{r}{i_{*}} & \Fun(\mathcal{U}(T), \mathcal{S}),
  \end{tikzcd}
  \]
  where both squares are cartesian.
\end{cor}
\begin{proof}
  By the Nerve Theorem~\ref{thm:nervecartsq}, we have a cartesian
  square
  \[
    \begin{tikzcd}
      \Alg_{T}(\mathcal{S}^{\mathcal{I}}) \arrow{r}{\nu_{\mathcal{W}(T)}}
      \arrow[d,swap,"U_{T}"]  & \Fun(\mathcal{W}(T), \mathcal{S})
      \arrow{d}{j^{*}} \\
      \Fun(\mathcal{I}, \mathcal{S}) \arrow[hookrightarrow]{r}{i_{*}} &
      \Fun(\mathcal{U}(T), \mathcal{S}).
    \end{tikzcd}
  \]
  Here the right vertical functor $j^{*}$ factors as $\Fun(\mathcal{W}(T),
  \mathcal{S}) \xto{j'^{*}} \Fun(\mathcal{W}(T)^{\xint}, \mathcal{S})
  \xto{u^{*}} \Fun(\mathcal{U}(T), \mathcal{S})$. The left vertical functor
  therefore factors uniquely through the pullback of $i_{*}$ along
  $u^{*}$, which we can identify with $\Fun(\mathcal{W}(T)^{\el},
  \mathcal{S})$ by Corollary~\ref{cor:elintcartsq}. This gives the
  desired commutative diagram. Here the bottom and outer squares are
  cartesian, and so the top square is also cartesian.
\end{proof}

\begin{cor}\label{cor AlgSeg}
  We have a commutative square
  \[
  \begin{tikzcd}
  \Alg_{T}(\mathcal{S}^{\mathcal{I}}) \arrow{r}{\sim} \arrow{d} &
  \Seg_{\mathcal{W}(T)}(\mathcal{S}) \arrow{d} \\
  \Fun(\mathcal{W}(T)^{\el}, \mathcal{S}) \arrow{r}{\sim} &
  \Seg_{\mathcal{W}(T)^{\xint}}(\mathcal{S})
  \end{tikzcd}
  \]
  where the horizontal functors are equivalences.
\end{cor}
\begin{proof}
  By definition, $\Seg_{\mathcal{W}(T)^{\xint}}(\mathcal{S})$ is the
  essential image of the fully faithful functor $i'_{*}$ in
  $\Fun(\mathcal{W}(T)^{\xint}, \mathcal{S})$, and
  $\Seg_{\mathcal{W}(T)}(\mathcal{S})$ is the full subcategory of
  $\Fun(\mathcal{W}(T), \mathcal{S})$ spanned by the functors whose
  restriction along $j'$ lies in this full subcategory; we thus have a
  pullback square
  \[
    \begin{tikzcd}
      \Seg_{\mathcal{W}(T)}(\mathcal{S}) \arrow[hookrightarrow]{r} \arrow{d} &
      \Fun(\mathcal{W}(T), \mathcal{S}) \arrow{d}{j'^{*}} \\
      \Seg_{\mathcal{W}(T)^{\xint}}(\mathcal{S})
      \arrow[hookrightarrow]{r}  & \Fun(\mathcal{W}(T)^{\xint}, \mathcal{S}).
    \end{tikzcd}
  \]
  The top cartesian square in the diagram of Corollary~\ref{cor morAlgFun}
  factors through this, giving a commutative diagram
  \[
    \begin{tikzcd}
     \Alg_{T}(\mathcal{S}^{\mathcal{I})} \arrow{r} \arrow{d} & \Seg_{\mathcal{W}(T)}(\mathcal{S}) \arrow[hookrightarrow]{r} \arrow{d} &
      \Fun(\mathcal{W}(T), \mathcal{S}) \arrow{d}{j'^{*}} \\
      \Fun(\mathcal{W}(T)^{\el}, \mathcal{S}) \ar[r, "\sim"{below},
      "i_{*}"{above}] & \Seg_{\mathcal{W}(T)^{\xint}}(\mathcal{S})
      \arrow[hookrightarrow]{r}  & \Fun(\mathcal{W}(T)^{\xint}, \mathcal{S}).
    \end{tikzcd}
  \]
  Here the left-hand square is cartesian, since the outer and
  right-hand squares are cartesian, and so the induced functor
  $\Alg_{T}(\mathcal{S}^{\mathcal{I}}) \to
  \Seg_{\mathcal{W}(T)}(\mathcal{S})$ is indeed an equivalence.  
\end{proof}

\begin{propn}\label{cor:UTFXcoinit}
  For $X \in \mathcal{U}(T)$, the functor \[\mathcal{U}(T)_{X/} \to
  \mathcal{U}(T)_{F_{T}X/}:= \mathcal{U}(T)
  \times_{\mathcal{W}(T)^{\xint}} \mathcal{W}(T)^{\xint}_{F_{T}X/}\] is coinitial.
\end{propn}
\begin{proof}
  By \cite[Theorem 4.1.3.1]{ht} it suffices to check that for
  $Y, \phi \colon FY \to FX$, the slice \icat{}
  $(\mathcal{U}(T)_{X/})_{/\phi}$ is weakly contractible. Here the
  canonical factorization of $\phi$ determines a terminal object, as
  in the proof of Proposition~\ref{propn:SIXphiwc}.
\end{proof}

\begin{cor}\label{cor:j!eq}
  There are natural equivalences of functors
  \[ \id \isoto u_{*}u^{*}, \]
  \[ u_{!}u^{*} \isoto \id, \]
  \[ j_{!}u^{*} \isoto j'_{!}.\]
\end{cor}
\begin{proof}
  For $\Phi \colon \mathcal{W}(T)^{\xint} \to \mathcal{S}$ the unit map
  $\Phi \to u_{*}u^{*}\Phi$ evaluates at $F_{T}X \in \mathcal{W}(T)^{\xint}$ as
  \[ \Phi(F_{T}X) \to \lim_{\mathcal{U}(T)_{F_{T}X/}} \Phi \circ u, \]
  which is an equivalence by Corollary~\ref{cor:UTFXcoinit}. This gives
  the first equivalence, which implies the second by passing to left
  adjoints. Applying $j'_{!}$ this gives the third equivalence, since
  $j'_{!}u_{!} \simeq (j'u)_{!} \simeq j_{!}$.
\end{proof}

\begin{cor}\label{cor WTextendable}
  The algebraic pattern $\mathcal{W}(T)$ is extendable.
\end{cor}
\begin{proof}
  We must show that $j'_{!}$ restricts to a functor
  $\Seg_{\mathcal{W}(T)^{\xint}}(\mathcal{S}) \to
  \Seg_{\mathcal{W}(T)}(\mathcal{S})$. Thus for
  $\Phi \in \Seg_{\mathcal{W}(T)^{\xint}}(\mathcal{S})$ we must show
  that $j'_{!}\Phi$ is a Segal object. By Corollary~\ref{cor:j!eq} the
  functor $j'_{!}\Phi$ is equivalent to $j_{!}u^{*}\Phi$. But since
  $\Phi$ is by assumption in
  $\Seg_{\mathcal{W}(T)^{\xint}}(\mathcal{S})$, we know by
  Corollary~\ref{cor:elintcartsq} that $u^{*}\Phi$ is right Kan
  extended from $\mathcal{I}$. Hence $j_{!}u^{*}\Phi$ is in
  $\Alg_{T}(\mathcal{S}^{\mathcal{I}}) \simeq
  \Seg_{\mathcal{W}(T)}(\mathcal{S})$ by
  Theorem~\ref{thm:nervecartsq}, as required.
\end{proof}

\begin{cor}\label{cor:morTWT}
  Inverting the equivalence of Corollary~\ref{cor AlgSeg}, we have a
  commutative square
  \[
    \begin{tikzcd}
      \Seg_{\mathcal{W}(T)}(\mathcal{S}) \arrow[d, swap, "U_{\mathcal{W}(T)}"]
     \ar[r, "\sim"{below}, "\phi"{above}] &
      \Alg_{T}(\mathcal{S}^{\mathcal{I}}) \arrow[d,"U_{T}"] \\
      \mathcal{S}^{\mathcal{W}(T)^{\el}} \arrow{r}{e^{*}} &
      \mathcal{S}^{\mathcal{I}}.
    \end{tikzcd}
  \]
  This square is a morphism of polynomial monads $T \to T_{\mathcal{W}(T)}$.
\end{cor}
\begin{proof}
  Since $\mathcal{W}(T)$ is extendable, we know that the free Segal
  $\mathcal{W}(T)$-space monad $T_{\mathcal{W}(T)}$ is polynomial by
  Proposition~\ref{prop:polymonad}. For the square to be a morphism of
  polynomial monads, it remains to show that the mate transformation
  $F_{T}e^{*} \to \phi F_{\mathcal{W}(T)}$ is cartesian. The
  equivalence $U_{T}\phi \simeq e^{*}U_{\mathcal{W}(T)}$ gives an
  equivalence of left adjoints $\phi^{-1}F_{T} \simeq
  F_{\mathcal{W}(T)}e_{!}$ under which the mate transformation
  corresponds to the transformation
  \[ \phi F_{\mathcal{W}(T)}e_{!}e^{*} \to \phi F_{\mathcal{W}(T)} \]
  induced by the counit $e_{!}e^{*} \to \id$.   This counit is easily
  seen to be cartesian (as in \cite[Lemma 2.1.5]{AnalMnd}), and since
  $U_{\mathcal{W}(T)}$ is conservative and preserves limits, it
  suffices to check this implies the transformation
  \[ T_{\mathcal{W}(T)}e_{!}e^{*} \to T_{\mathcal{W}(T)} \]
  is cartesian, which is true since $T_{\mathcal{W}(T)}$ preserves pullbacks.
\end{proof}

We now show that the pattern $\mathcal{W}(T)$ is natural with respect
to morphisms of polynomial monads:
\begin{thm}\label{thm:functorP}
  There is a functor
  \[ \mathfrak{P} \colon \PolyMnd \to \name{AlgPatt}^{\Seg}_{\name{ext}}\]
  that takes a polynomial monad $T$ on $\xS^\xI$ to the algebraic
  pattern $\mathcal{W}(T)$, and a natural transformation
  \[ \tau \colon \id \to \mathfrak{M}\mathfrak{P},\]
  given by the morphism $T \to \mathfrak{M}(\mathcal{W}(T))$ from
  Corollary~\ref{cor:morTWT}, where $\mathfrak{M}$ is the functor from
  Corollary~\ref{cor:patttopolymnd} that takes an extendable pattern
  $\mathcal{O}$ to the free Segal $\mathcal{O}$-space monad.
\end{thm}
\begin{proof}
  Suppose we have a morphism of polynomial monads $T \to S$, given by
  a functor $f \colon \mathcal{I} \to \mathcal{J}$ and a commutative square
  \[
    \begin{tikzcd}
      \Alg_{S}(\mathcal{S}^{\mathcal{J}}) \arrow{r}{\Phi}
      \arrow[d,swap,"U_{S}"] & \Alg_{T}(\mathcal{S}^{\mathcal{I}})
      \arrow{d}{U_{T}} \\
      \mathcal{S}^{\mathcal{J}} \arrow{r}{f^{*}} & \mathcal{S}^{\mathcal{I}}.
    \end{tikzcd}
  \]
  By Proposition~\ref{propn:Wmor}, the functor $\Phi$ has a left
  adjoint $\Psi$ which restricts to a functor $\Psi^{\op} \colon
  \mathcal{W}(T) \to \mathcal{W}(S)$. Lemma~\ref{lem WTWS} implies
  that this functor preserves active and inert morphisms, since the
  active morphisms are precisely those that are adjoint to generic
  morphisms by Lemma~\ref{lem:actgen}, while the inert morphisms are
  the composites of free morphisms and equivalences. The commutative
  square from Proposition~\ref{propn:Wmor}(iv) restricts to a commutative square
  \[
    \begin{tikzcd}
      \mathcal{I} \arrow{d} \arrow{r}{f} & \mathcal{J} \arrow{d} \\
      \mathcal{W}(T) \arrow{r}{\Psi^{\op}} & \mathcal{W}(S),
    \end{tikzcd}
  \]
  and so $\Psi^{\op}$ also preserves elementary objects. Thus
  $\Psi^{\op}$ is a morphism of algebraic patterns. 
  
  It follows from Corollary~\ref{cor AlgSeg} and Corollary~\ref{cor
    morTS} that
  $\Phi \colon \Alg_{S}(\mathcal{S}^{\mathcal{J}}) \to
  \Alg_{T}(\mathcal{S}^{\mathcal{I}})$ can be identified with the
  restriction of $(\Psi^{\op})^{*}$ to Segal objects, thus
  $\Psi^{\op}$ is a Segal morphism by Lemma~\ref{lem:Segalmor}.

  Since this construction is obviously compatible with composition we
  obtain a functor \[\mathfrak P \colon
  \PolyMnd\to\AlgPatt^\Seg_{\xext}.\] Using
  Corollary~\ref{cor:elintcartsq} the commutative cube in Corollary~\ref{cor
    morTS} extends to a commutative diagram
  \[
    \begin{tikzcd}[row sep=tiny,column sep=tiny]
      \Alg_{S}(\mathcal{S}^{\mathcal{J}})
\arrow[hookrightarrow]{rr} \arrow{dd} \arrow{dr} & &
\Fun(\mathcal{W}(S), \mathcal{S}) \arrow{dr} \arrow{dd} \\
 & \Alg_{T}(\mathcal{S}^{\mathcal{I}}) \arrow[crossing
over,hookrightarrow]{rr} & & \Fun(\mathcal{W}(T), \mathcal{S})
\arrow{dd} \\
\Fun(\mathcal{W}(S)^{\el}, \mathcal{S}) \arrow{dd} \arrow{dr}
\arrow[hookrightarrow]{rr} & & \Fun(\mathcal{W}(S)^{\xint},
\mathcal{S}) \arrow{dd} \arrow{dr} \\
 & \Fun(\mathcal{W}(T)^{\el}, \mathcal{S}) \arrow{dd}
\arrow[leftarrow,crossing over]{uu} \arrow[hookrightarrow, crossing over]{rr} & & \Fun(\mathcal{W}(T)^{\xint},
\mathcal{S}) \arrow{dd} \\
 \mathcal{S}^{\mathcal{J}} \arrow[hookrightarrow]{rr}
\arrow{dr} & & \Fun(\mathcal{U}(S), \mathcal{S}) \arrow{dr}
\\ & \mathcal{S}^{\mathcal{I}} \arrow[hookrightarrow]{rr}
\arrow[leftarrow,crossing over]{uu} & & \Fun(\mathcal{U}(T),
\mathcal{S}),
    \end{tikzcd}
  \]
where the left side gives the naturality square
\[
  \begin{tikzcd}
    S \arrow{r} \arrow{d} & T_{\mathcal{W}(S)} \arrow{d} \\
    T \arrow{r} & T_{\mathcal{W}(T)}.
  \end{tikzcd}
\]
Since this construction is again compatible with composition, it gives
a natural transformation $\id \to \mathfrak{M}\mathfrak{P}$.
\end{proof}

\begin{variant}
  Let us say that a \emph{flagged algebraic pattern} is a pair
  $(\xxO, \xI\to \xxO^\el)$ where $\xxO$ is an algebraic pattern and
  $\xI\to \xxO^\el$ is an essentially surjective functor of
  \icats{}. We write $\FlAlgPatt$ for the full subcategory of
  $\AlgPatt \times_{\CatI} \xFun(\Delta^1, \CatI)$ spanned by the
  flagged algebraic patterns, and $\FlAlgPatt^{\Seg}_{\name{ext}}$ for
  the subcategory consisting of flagged algebraic patterns whose
  underlying patterns are extendable, with morphisms those such that
  the underlying morphisms of patterns are Segal morphisms. As a
  variant of the construction of $\mathfrak{P}$ above, we can
  define a functor 
  \[ \mathfrak{P}' \colon \PolyMnd \to
    \FlAlgPatt^{\Seg}_{\name{ext}} \] that takes a polynomial monad
  $T$ on $\mathcal{S}^{\mathcal{I}}$ to the flagged algebraic pattern
  $(\mathcal{W}(T), \mathcal{I} \xto{e} \mathcal{W}(T)^{\el})$. Note that
  we can recover the monad $T$ from this flagged pattern, since
  $U_{T}$ is equivalent to the composite
  \[\Seg_{\mathcal{W}(T)}(\mathcal{S}) \xto{U_{\mathcal{W}(T)}}
    \Fun(\mathcal{W}(T)^{\el}, \mathcal{S}) \xto{e^{*}}
    \Fun(\mathcal{I}, \mathcal{S}).\]
  For any flagged extendable pattern
  $(\mathcal{O}, f \colon \mathcal{I} \to \mathcal{O}^{\el})$
  the composite
  \[ \Seg_{\mathcal{O}}(\mathcal{S}) \xto{U_{\mathcal{O}}}
    \Fun(\mathcal{O}^{\el}, \mathcal{S}) \xto{f^{*}} \Fun(\mathcal{I},
    \mathcal{S}) \]
  is a monadic right adjoint (since $f^{*}$ preserves all limits and
  colimits and is conservative when $f$ is essentially
  surjective), but we do not know under what conditions on $f$ the
  corresponding monad is polynomial. This means that we do not have a
  satisfactory flagged version of the functor $\mathfrak{M}$ in
  general. 
  However, if we restrict to patterns $\mathcal{O}$ such that
  $\mathcal{O}^{\el}$ is
  an $\infty$-groupoid, then this construction \emph{does} give a polynomial
  monad for any essentially surjective morphism $f$ of
  $\infty$-groupoids, since in this case the left adjoint $f_{!}$
  preserves weakly contractible limits by \cite[Lemma 2.2.10]{AnalMnd}
  and the unit and counit for the adjunction $f_{!} \dashv f^{*}$ are
  cartesian transformations by \cite[Lemma 2.1.5]{AnalMnd}.
\end{variant}

\section{Saturation and Canonical Patterns}
\label{sec:poly2}

Suppose $\mathcal{O}$ is an extendable algebraic pattern. Then the
free Segal $\mathcal{O}$-space monad $T_{\mathcal{O}}$ is polynomial,
and our results in the previous section associate to this another
algebraic pattern $\olO:=
\mathcal{W}(T_{\mathcal{O}})$ such that there is an
equivalence\footnote{In the next section, we will see that furthermore
  the patterns $\mathcal{O}$ and $\olO$ determine the same polynomial monad.}
\[\Seg_{\mathcal{O}}(\mathcal{S}) \simeq
  \Seg_{\olO}(\mathcal{S}).\]
In this section we will explore the relationship between the patterns
$\mathcal{O}$ and $\olO$. We will show that under a
mild hypothesis on $\mathcal{O}$ (which can always be enforced by
passing to a full subcategory without changing the monad) there is a
canonical morphism of patterns
$\mathcal{O} \to \olO$, which gives a natural
transformation
\[ \id \to \mathfrak{P}\mathfrak{M}.\]
We will also give an explicit necessary and sufficient condition on
$\mathcal{O}$ for the map $\mathcal{O} \to \olO$ to be an
equivalence, and discuss some examples where this holds.

\begin{notation}
  In the first part of this section we fix an extendable pattern
  $\mathcal{O}$, and use the notations
  \[ \mathcal{O}^{\el} \xto{i} \mathcal{O}^{\xint} \xto{j}
    \mathcal{O}\]
  for the standard inclusions.
\end{notation}

We begin by studying the localized Yoneda embedding
\[ \mathcal{O}^{\op} \to \Fun(\mathcal{O}, \mathcal{S}) \to
  \Seg_{\mathcal{O}}(\mathcal{S}) \] for a pattern $\mathcal{O}$,
which will give the canonical map to $\olO$.
\begin{notation}
  Let
  $\LOpi \colon \mathcal{O}^{(\xint),\op} \to
  \Seg_{\mathcal{O}^{(\xint)}}(\mathcal{S})$ denote the composite of
  the Yoneda embedding
  $\yOpi \colon \mathcal{O}^{(\xint),\op} \to
  \Fun(\mathcal{O}^{(\xint)}, \mathcal{S})$ with the localization
  $\Fun(\mathcal{O}^{(\xint)}, \mathcal{S}) \to
  \Seg_{\mathcal{O}^{(\xint)}}(\mathcal{S})$.
\end{notation}

\begin{lemma}\label{lem:LambdaisF}
  For $X \in \mathcal{O}$, there is an equivalence
  \[ \Lambda_{\mathcal{O}}X \simeq
    F_{\mathcal{O}}\Lambda_{\mathcal{O}}^{\xint}X\]
  in $\Seg_{\mathcal{O}}(\mathcal{S})$. This equivalence is natural with respect to inert morphisms, \ie{}
  we have a commutative square
  \[
    \begin{tikzcd}
      \mathcal{O}^{\op} \arrow{r}{\Lambda_{\mathcal{O}}} &
      \Seg_{\mathcal{O}}(\mathcal{S}) \\
      \mathcal{O}^{\xint,\op} \arrow{u}{j^{\op}}
      \arrow{r}{\Lambda_{\mathcal{O}}^{\xint}} &
      \Seg_{\mathcal{O}^{\xint}}(\mathcal{S}). \arrow[u,swap,"F_{\mathcal{O}}"]
    \end{tikzcd}
  \]
\end{lemma}
\begin{proof}
  For $\Phi \in \Seg_{\mathcal{O}}(\mathcal{S})$, we have natural
  equivalences
  \[
    \begin{split}
\Map_{\Seg_{\mathcal{O}}(\mathcal{S})}(\Lambda_{\mathcal{O}}X,
    \Phi) & \simeq
    \Map_{\Fun(\mathcal{O},\mathcal{S})}(y_{\mathcal{O}}X, \Phi) \\ & 
    \simeq \Phi(X), \\
      \Map_{\Seg_{\mathcal{O}}(\mathcal{S})}(F_{\mathcal{O}}\Lambda_{\mathcal{O}}^{\xint}X,
    \Phi) & \simeq
    \Map_{\Seg_{\mathcal{O}^{\xint}}(\mathcal{S})}(\Lambda_{\mathcal{O}}^{\xint}X,
    U_{\mathcal{O}}\Phi) \\ & \simeq \Map_{\Fun(\mathcal{O}^{\xint},
      \mathcal{S})}(y_{\mathcal{O}}^{\xint}X, U_{\mathcal{O}}\Phi) \\
    & 
    \simeq U_{\mathcal{O}}\Phi(X) \\ & \simeq \Phi(X).    
    \end{split}
  \]
  The objects
  $\Lambda_{\mathcal{O}}X$ and
  $F_{\mathcal{O}}\Lambda_{\mathcal{O}}^{\xint}X$ therefore corepresent
  the same copresheaf on $\Seg_{\mathcal{O}}(\mathcal{S})$ and hence
  are equivalent. Moreover, this equivalence is by construction
  natural in $\mathcal{O}^{\xint}$.
\end{proof}

\begin{lemma}\label{lem:Lambdadesc}
  The map
  \[ \Map_{\mathcal{O}}(X, Y) \to
    \Map_{\Seg_{\mathcal{O}}(\mathcal{S})}(\LO Y,
    \LO X)\]
  given by the functor $\LO$ fits in a commutative square
  \[
    \begin{tikzcd}
      \colim_{O \to Y \in
      \Act_{\mathcal{O}}(Y)} \Map_{\mathcal{O}^{\xint}}(X, O)
    \arrow{r} \arrow[d, swap, "\wr"] & \colim_{O \to Y \in
      \Act_{\mathcal{O}}(Y)}
    \Map_{\Seg_{\mathcal{O}^{\xint}}(\mathcal{S})}(\LOi O, \LOi X)
    \arrow{d}{\wr} \\
    \Map_{\mathcal{O}}(X, Y) \arrow{r} & 
    \Map_{\Seg_{\mathcal{O}}(\mathcal{S})}(\LO Y,
    \LO X),
    \end{tikzcd}
    \]
    where the vertical maps are equivalences and the top horizontal
    map comes from the functor $\LOi$.  
\end{lemma}
\begin{proof}
  From the commutative square of functors in Lemma~\ref{lem:LambdaisF}
  we get for all $O \in \mathcal{O}$ a commutative square
  \[
    \begin{tikzcd}
      \Map_{\mathcal{O}^{\xint}}(X,O) \arrow{r} \arrow{d} & 
       \Map_{\Seg_{\mathcal{O}^{\xint}}(\mathcal{S})}(\Lambda_{\mathcal{O}}^{\xint}O,
      \Lambda_{\mathcal{O}}^{\xint}X)
       \arrow{d} \\
      \Map_{\mathcal{O}}(X,O) \arrow{r} &\Map_{\Seg_{\mathcal{O}}(\mathcal{S})}(\Lambda_{\mathcal{O}}O,
      \Lambda_{\mathcal{O}}X)  ,
    \end{tikzcd}
  \]
  where the right-hand map can be identified with $\Lambda^{\xint}_{\mathcal{O}}X(O)\to \colim_{O' \in
    \Act_{\mathcal{O}}(O)} \Lambda^{\xint}_{\mathcal{O}}X(O')$ which is the canonical map to the colimit
  from the component at $\id_{O}$.
  On the other hand, for any active morphism $O \to Y$ we have a
  natural commutative diagram
  \[
    \begin{tikzcd}[column sep=small]
      \Map_{\mathcal{O}}(X,O) \arrow{r} \arrow{d} &
      \Map_{\Seg_{\mathcal{O}}(\mathcal{S})}(\Lambda_{\mathcal{O}}O,
      \Lambda_{\mathcal{O}}X) \arrow{r}{\sim} \arrow{d} &
      T_{\mathcal{O}}\Lambda^{\xint}_{\mathcal{O}}X(O) \arrow{r}{\sim}
      \arrow{d} & \colim_{O' \in \Act_{\mathcal{O}}(O)}
      \Lambda^{\xint}_{\mathcal{O}}X(O') \arrow{d} \\
      \Map_{\mathcal{O}}(X,Y) \arrow{r} &
      \Map_{\Seg_{\mathcal{O}}(\mathcal{S})}(\Lambda_{\mathcal{O}}X,
      \Lambda_{\mathcal{O}}Y) \arrow{r}{\sim} &
      T_{\mathcal{O}}\Lambda^{\xint}_{\mathcal{O}}X(Y) \arrow{r}{\sim}
      & \colim_{O'' \in \Act_{\mathcal{O}}(Y)}
      \Lambda^{\xint}_{\mathcal{O}}X(O''),
    \end{tikzcd}
  \]
  where the description of
  $F_{\mathcal{O}}\Lambda^{\xint}_{\mathcal{O}}X$ as a left Kan
  extension implies that the right-hand map is given on the component
  $\Lambda^{\xint}_{\mathcal{O}}X(O')$ for $O' \to O$ by the canonical
  map
  $\Lambda^{\xint}_{\mathcal{O}}X(O') \to \colim_{O'' \in
    \Act_{\mathcal{O}}(Y)} \Lambda^{\xint}_{\mathcal{O}}X(O'')$ for
  the component at $O' \to O \to Y$. Putting these two diagrams
  together we therefore obtain natural commutative squares
  \[
    \begin{tikzcd}
      \Map_{\mathcal{O}^{\xint}}(X,O) \arrow{r} \arrow{d} &
      \Lambda^{\xint}_{\mathcal{O}}X(O) \arrow{d} \\
      \Map_{\mathcal{O}}(X,Y) \arrow{r} & \colim_{O' \in \Act_{\mathcal{O}}(Y)}
      \Lambda^{\xint}_{\mathcal{O}}X(O'), 
    \end{tikzcd}
  \]
  for every active morphism $\phi \colon O \to Y$, where the right
  vertical map is the canonical one from the component of
  the colimit at $\phi$. Taking colimits over $\Act_{\mathcal{O}}(Y)$
  we therefore get a commutative square
  \[
    \begin{tikzcd}
      \colim_{O \in \Act_{\mathcal{O}}(Y)}\Map_{\mathcal{O}^{\xint}}(X,O) \arrow{r} \arrow{d} &
      \colim_{O \in \Act_{\mathcal{O}}(Y)} \Lambda^{\xint}_{\mathcal{O}}X(O) \arrow[equals]{d} \\
      \Map_{\mathcal{O}}(X,Y) \arrow{r} & \colim_{O \in \Act_{\mathcal{O}}(Y)}
      \Lambda^{\xint}_{\mathcal{O}}X(O).
    \end{tikzcd}
  \]
  Here the inert--active factorization system on $\mathcal{O}$ implies
  that the left vertical map is an equivalence, since its fibre at a
  morphism $\psi \colon X \to Y$ can be identified with the space of
  inert--active factorizations of $\psi$, and this completes the proof.
\end{proof}
  
\begin{remark}\label{rem MapSegOint}
  For $Y \in \mathcal{O}$, we have a natural equivalence
  \[
    \Map_{\Seg_{\mathcal{O}^{\xint}}(\mathcal{S})}(\Lambda^{\xint}_{\mathcal{O}}
    Y, \Phi) \simeq \Map_{\Fun(\mathcal{O}^{\xint},
      \mathcal{S})}(y_{\mathcal{O}}^{\xint}Y, \Phi) \simeq \Phi(Y).\]
  In particular,
  \[
    \Map_{\Seg_{\mathcal{O}^{\xint}}(\mathcal{S})}(\Lambda^{\xint}_{\mathcal{O}}Y,
    T_{\mathcal{O}}*) \simeq \Act_{\mathcal{O}}(Y),\] and so a
  morphism $\Lambda^{\xint}_{\mathcal{O}}Y \to T_{\mathcal{O}}*$
  corresponds to an active morphism $X \to Y$ in $\mathcal{O}$.
\end{remark}

We will now show that this equivalence identifies active morphisms in
$\mathcal{O}$ with generic morphisms in
$\Seg_{\mathcal{O}^{\xint}}(\mathcal{S})$:
\begin{propn}\label{propn:Oactgeneric}
  Suppose $\Lambda^{\xint}_{\mathcal{O}}Y \xto{\eta} T_{\mathcal{O}}*$
  corresponds to the active morphism $X \overset{\phi}{\actto} Y$ in
  $\Act_{\mathcal{O}}(Y)$ under the equivalence of Remark~\ref{rem
    MapSegOint}. Then the generic--free factorization of $\eta$ is
  \[\Lambda^{\xint}_{\mathcal{O}}Y \xto{\hat{\phi}}
    T_{\mathcal{O}}\Lambda^{\xint}_{\mathcal{O}}X \to
    T_{\mathcal{O}}*,\]
  where the first morphism is adjoint to $\Lambda_{\mathcal{O}}(\phi)
  \colon \Lambda_{\mathcal{O}}Y \to \Lambda_{\mathcal{O}}X$.
\end{propn}
\begin{proof}
  We first check that this factorization exists. By
  Lemma~\ref{lem:Lambdadesc} the morphism $\hat{\phi}$ adjoint to
  $\Lambda(\phi)$ corresponds to the point in
  $T_{\mathcal{O}}\Lambda^{\xint}_{\mathcal{O}}X(Y) \simeq \colim_{O
    \in \Act_{\mathcal{O}}(Y)} \Lambda^{\xint}_{\mathcal{O}}X(O)$
  given by the composite
  \[ \{\id_{X}\} \to \Map_{\mathcal{O}^{\xint}}(X,X) \to  \colim_{O \in \Act_{\mathcal{O}}(Y)}
    \Map_{\mathcal{O}^{\xint}}(X,O) \to \colim_{O
      \in \Act_{\mathcal{O}}(Y)} \Lambda^{\xint}_{\mathcal{O}}X(O),\]
  where the second morphism is the canonical one from the component of
  the colimit at $\phi$. We therefore have a commutative diagram
  \[
    \begin{tikzcd}
      {} & * \arrow{dl} \arrow{ddr} \\
      \colim_{O \in \Act_{\mathcal{O}}(Y)}
      \Map_{\mathcal{O}^{\xint}}(X,O) \arrow{d} \arrow{drr} &\\
      \colim_{O \in \Act_{\mathcal{O}}(Y)}
      \Lambda^{\xint}_{\mathcal{O}}X(O) \arrow{rr} & & \Act_{\mathcal{O}}(Y),
    \end{tikzcd}
  \]
  where the outer triangle corresponds to the desired factorization
  \[
    \begin{tikzcd}
      {} & \Lambda^{\xint}_{\mathcal{O}}Y \arrow{dr}{\eta}
      \arrow{dl}[above left]{\hat{\phi}} \\
      T_{\mathcal{O}}\Lambda^{\xint}_{\mathcal{O}}X \arrow{rr} & & T_{\mathcal{O}}(*).
    \end{tikzcd}
  \]
  Now we must show that $\hat{\phi}$ is generic, so suppose we have a commutative square
  \[
    \begin{tikzcd}
      \Lambda^{\xint}Y  \arrow{r}{\theta} \arrow[d, swap, "\hat{\phi}"] & T_\xxO\Phi \arrow{d} \\
      T_\xxO\Lambda^{\xint}X \arrow{r} & T_\xxO*,
    \end{tikzcd}
  \]
  where the top horizontal map corresponds to a point $p$ in the fibre $\Phi(X)$ of
  $T_{\mathcal{O}}\Phi(Y) \simeq \colim_{O \in
    \Act_{\mathcal{O}}(Y)}\Phi(O)$ at $\phi$. Suppose we have a
  commutative triangle of the form
  \[
    \begin{tikzcd}
      {} & \Lambda^{\xint}_{\mathcal{O}}Y \arrow{dl}[above left]{\hat{\phi}} \arrow{dr}{\theta} \\
      T_{\mathcal{O}}\Lambda^{\xint}X \arrow{rr}{T\psi} & & T_{\mathcal{O}}\Phi.
    \end{tikzcd}
  \]
  This amounts to an equivalence between $p$ and the image
  \[ * \xto{\id_{X}} \Map_{\mathcal{O}^{\xint}}(X,X) \to \colim_{O \in
    \Act_{\mathcal{O}}(Y)} \Map_{\mathcal{O}^{\xint}}(X,O) \to \colim_{O \in
    \Act_{\mathcal{O}}(Y)} \Lambda^{\xint}_{\mathcal{O}}X(O) \to \colim_{O \in
      \Act_{\mathcal{O}}(Y)} \Phi(O).\]
  But since the last map arises from $T\psi$, there is a commutative
  diagram
  \[
    \begin{tikzcd}
      \Map_{\mathcal{O}^{\xint}}(X,X) \arrow{d} \arrow{r} &
      \Lambda^{\xint}_{\mathcal{O}}X(X) \arrow{d} \arrow{r} & \Phi(X)
      \arrow{d} \\
      \colim_{O \in
    \Act_{\mathcal{O}}(Y)} \Map_{\mathcal{O}^{\xint}}(X,O) \arrow{r} &  \colim_{O \in
    \Act_{\mathcal{O}}(Y)} \Lambda^{\xint}_{\mathcal{O}}X(O) \arrow{r}
  &  \colim_{O \in
      \Act_{\mathcal{O}}(Y)} \Phi(O),
    \end{tikzcd}
    \]
  which tells us that $\psi$ must be the morphism
  $\Lambda^{\xint}_{\mathcal{O}}X \to \Phi$ obtained by localizing the
  unique natural transformation $y_{\mathcal{O}}^{\xint}X \to \Phi$
  that takes $\id_{X}$ to the point $p$. Thus $\hat{\phi}$
  satisfies the universal property of generic morphisms
  described in Remark~\ref{rmk:genericfill}. By the uniqueness of
  generic--free factorizations, this completes the proof.
\end{proof}

This proposition allows us to identify the objects of $\mathcal{U}(T_{\mathcal{O}})$:
\begin{defn}\label{def slim}
  We say an object $O \in \mathcal{O}$ is \emph{necessary} if it
  admits an active morphism $O \to E$ for some
  $E \in \mathcal{O}^{\el}$, and denote by $\mathcal{O}^{\circ}$ the
  full subcategory of $\mathcal{O}$ spanned by the necessary
  objects. We say the pattern $\mathcal{O}$ is \emph{slim} if all
  objects are necessary, and write $\AlgPattSES$
  for the full subcategory of $\AlgPattSE$ spanned
  by the slim extendable patterns.
\end{defn}

\begin{cor}\label{cor:overlineOdesc}
  Let $\mathcal{O}$ be an extendable algebraic pattern, and let
  $\olO := \mathcal{W}(T_{\mathcal{O}})$ denote the
  corresponding canonical pattern. Then:
  \begin{enumerate}[(i)]
  \item The objects of $\mathcal{U}(T_{\mathcal{O}})$ are the
    objects of $\Seg_{\mathcal{O}^{\xint}}(\mathcal{S})$ of the form
    $\LOi X$ with $X \in
    \mathcal{O}^{\circ}$. Thus
    $\LOi$ induces an essentially
    surjective functor $\mathcal{O}^{\circ,\xint} \to
    \mathcal{U}(T_{\mathcal{O}})$.
  \item The objects of $\olO$ are the
    objects of $\Seg_{\mathcal{O}}(\mathcal{S})$ of the form
    $\Lambda_{\mathcal{O}}X$ with $X \in
    \mathcal{O}^{\circ}$. Thus
    $\Lambda_{\mathcal{O}}$ induces an essentially
    surjective functor $\mathcal{O}^{\circ} \to
    \olO$.
  \item A morphism $\Lambda_{\mathcal{O}}X \to \Lambda_{\mathcal{O}}Y$
    is active \IFF{} it is a composite of an equivalence and the
    image of an active morphism $X \to Y$
    in $\mathcal{O}^{\circ}$. In particular, the functor
    $\mathcal{O}^{\circ}\to \olO$ preserves inert
    and active morphisms.
  \end{enumerate}
\end{cor}
\begin{proof}
  By definition, the objects of $\mathcal{U}(T_{\mathcal{O}})$ are
  the objects $\Phi$ of $\Seg_{\mathcal{O}^{\xint}}(\mathcal{S})$ that
  admit a generic morphism
  $\LOi E \to T_{\mathcal{O}}\Phi$ with
  $E \in \mathcal{O}^{\el}$. Such a generic morphism is determined by
  a morphism $\Lambda^{\xint}_{\mathcal{O}}E \to T_{\mathcal{O}}*$,
  and from Proposition~\ref{propn:Oactgeneric} we see that the
  generic--free factorizations of such morphisms yield precisely the
  objects of $\mathcal{O}^{\circ}$. This proves (i), from which (ii)
  follows using Lemma~\ref{lem:LambdaisF}. Finally, as active
  morphisms in $\olO$ are those morphisms which are adjoint to generic
  maps by Lemma~\ref{lem:actgen}, the first part of (iii) follows from the identification
  of such generic morphisms with active morphisms in $\mathcal{O}$ in
  Proposition~\ref{propn:Oactgeneric}. This shows that $\LO$ preserves
  active morphisms, while the commutative square of
  Lemma~\ref{lem:LambdaisF} implies that it preserves inert morphisms,
  since free morphisms in $\olO$ are in particular inert.
\end{proof}

\begin{remark}\label{rem barO}
  If $\mathcal{O}$ is a slim extendable pattern, then
  Corollary~\ref{cor:overlineOdesc} says that $\olO$ has the same
  objects as $\mathcal{O}$, and the active morphisms are obtained by
  combining active morphisms from $\mathcal{O}$ with equivalences
  (which may not all come from $\mathcal{O}$).
\end{remark}

\begin{remark}
  If $O$ is necessary and $O' \to O$ is an active morphism, then $O'$
  is also necessary. This implies that the inert--active factorization
  system in $\mathcal{O}$ restricts to
  $\mathcal{O}^{\circ}$, and that $\Act_{\mathcal{O}}(O) \simeq
  \Act_{\mathcal{O}^{\circ}}(O)$ for $O \in
  \mathcal{O}^{\circ}$. It follows that $\mathcal{O}^{\circ}$ is
  extendable when $\mathcal{O}$ is. In this case we
  therefore have a commutative diagram
  \[
    \begin{tikzcd}
      \Seg_{\mathcal{O}}(\mathcal{S}) \arrow{d} \arrow{rr} & &
      \Seg_{\mathcal{O}^{\circ}}(\mathcal{S}) \arrow{d} \\
      \Seg_{\mathcal{O}^{\xint}}(\mathcal{S}) \arrow{rr}
      \arrow{dr}[below left]{\sim}& &
      \Seg_{\mathcal{O}^{\circ,\xint}}(\mathcal{S}) \arrow{dl}{\sim}\\
       & \Fun(\mathcal{O}^{\el}, \mathcal{S}),
    \end{tikzcd}
  \]
  where the vertical maps are monadic right adjoints. The lower
  horizontal map is an equivalence since the diagonal maps are
  equivalences. Since the two monads on
  $\Fun(\mathcal{O}^{\el}, \mathcal{S})$ are the same (by definition
  $\Act_{\mathcal{O}}(E) \simeq \Act_{\mathcal{O}^{\circ}}(E)$ for
  $E \in \mathcal{O}^{\el}$), the top horizontal morphism is also an
  equivalence. Thus the patterns $\mathcal{O}$ and
  $\mathcal{O}^{\circ}$ describe the same monad, and so the
  objects of $\mathcal{O}$ that do not lie in $\mathcal{O}^{\circ}$
  are in this sense \emph{unnecessary}.
\end{remark}

\begin{examples}
  The examples of patterns discussed in \S\ref{sec:ex} are all slim,
  with the exception of the pattern $\simp^{\op, \natural}_{\Phi}$ of
  Example~\ref{ex DPhi}. The corresponding slim pattern
  $\simp^{\op, \natural,\circ}_{\Phi}$ is the full subcategory spanned
  by objects $([m], f)$ such that $f(m) \cong *$. Another non-slim
  example is the extension of the dendroidal category $\bbOmega^{\op,\natural}$ to a category
  of forests considered in \cite{HeutsHinichMoerdijk}, which has
  $\bbOmega^{\op,\natural}$ as its slim subpattern.
\end{examples}

\begin{remark}\label{rmk WTslim}
  If $T$ is a polynomial monad on $\xS^\xI$ then the algebraic pattern
  $\xW(T)$ is slim. This follows from the fact that objects in
  $\xW(T)$ can be identified with objects in $\xU(T)$, \ie{} objects
  $X$ admitting a generic map $I\to TX$ with $I\in \xI$. Since $\xI$
  has the same objects as $\xW(T)^\el$ and every generic map is
  adjoint to an active morphism in $\xW(T)$, the algebraic pattern
  $\xW(T)$ is indeed slim. We can thus regard $\mathfrak{P}$ as a
  functor
  \[ \PolyMnd \to \AlgPattSES.\]
\end{remark}

\begin{remark}\label{rem:overlineO}
  Suppose $f \colon \mathcal{O} \to \mathcal{P}$ is a Segal morphism
  between slim extendable patterns. Then we have a commutative diagram
  \[
    \begin{tikzcd}
      \mathcal{O}^{\op} \arrow{r}{f^{\op}}\arrow{d} &
      \mathcal{P}^{\op} \arrow{d} \\
      \Fun(\mathcal{O}, \mathcal{S}) \arrow{d}
      \arrow{r}{f_{!}}  & \Fun(\mathcal{P},
      \mathcal{S}) \arrow{d} \\
      \Seg_{\mathcal{O}}(\mathcal{S}) \arrow{r} & \Seg_{\mathcal{P}}(\mathcal{S}).
    \end{tikzcd}
  \]
  In other words, we have a commutative square
  \[
    \begin{tikzcd}
      \mathcal{O}^{(\xint),\op} \arrow{r}{f^{(\xint),\op}}
      \arrow[d, swap, "\Lambda_{\mathcal{O}}"] &
      \mathcal{P}^{(\xint),\op}
      \arrow{d}{\Lambda_{\mathcal{P}}} \\
      \Seg_{\mathcal{O}}(\mathcal{S}) \arrow{r} & \Seg_{\mathcal{P}}(\mathcal{S}),
    \end{tikzcd}
  \]
  which restricts to a commutative square
  \[
    \begin{tikzcd}
      \mathcal{O} \arrow{r}{f} \arrow{d} &
      \mathcal{P} \arrow{d} \\
      \olO \arrow{r} &  \overline{\mathcal{P}},
    \end{tikzcd}
  \]
  where all the morphisms are Segal morphism of algebraic patterns. 
  Thus we have proved:
\end{remark}

\begin{propn}\label{propn:nattrPM}
  There is a natural transformation
  $\sigma \colon \id \to \mathfrak{P}\mathfrak{M}$ of functors
  $\AlgPattSES \to \AlgPattSES$.
\end{propn}

Our next goal is to identify when the map $\sigma_{\mathcal{O}}$ is an
equivalence, which turns out to correspond to the following condition:
\begin{defn}
  If $\mathcal{O}$ is a slim extendable pattern, we say that
  $\mathcal{O}$ is \emph{saturated} if for every object $O \in
  \mathcal{O}$ the copresheaf \[\Map_{\mathcal{O}}(O, \blank) \colon
    \mathcal{O} \to \mathcal{S}\]
  is a Segal $\mathcal{O}$-space. We write $\AlgPattSS$ for the full
  subcategory of $\AlgPattSES$ spanned by the saturated patterns.
\end{defn}

\begin{propn}\label{propn:Osatcond}
  The following conditions are equivalent for a slim extendable pattern $\mathcal{O}$:
  \begin{enumerate}[(1)]
  \item $\mathcal{O}$ is saturated.
  \item For every $X \in \mathcal{O}$, the canonical functor
    $\mathcal{O}^{\xint,\triangleleft}_{X/} \to \mathcal{O}$ is a
    limit diagram.
  \item The Yoneda embedding $\mathcal{O}^{\op} \to
    \Fun(\mathcal{O}, \mathcal{S})$ factors through
    $\Seg_{\mathcal{O}}(\mathcal{S})$.
  \item The functor $\LO \colon
    \mathcal{O}^{\op} \to
    \Seg_{\mathcal{O}}(\mathcal{S})$ is fully faithful.
  \end{enumerate}
\end{propn}
\begin{proof}
  The equivalence of (1), (2), and (3) is clear, and it is also clear
  that (3) implies (4). We prove the remaining implication from (4) to (3)
  by showing that (4) implies that for every $X \in \mathcal{O}$ there is an equivalence $\yO X \simeq
  \LO X$ in $\Fun(\mathcal{O},\mathcal{S})$.
  We have 
  \[ \Map_{\mathcal{O}}(X,Y)\isoto
  \Map_{\Fun(\mathcal{O}, \mathcal{S})}(\yO Y, \yO X)
  \to   \Map_{\Seg_{\mathcal{O}}(\mathcal{S})}(\LO Y, \LO X)
  \to \Map_{\Fun(\mathcal{O}, \mathcal{S})}(\yO Y, \LO X),\]
  where the first map is the Yoneda embedding. Since the composition
  of the first two morphisms is an equivalence by (4), the second map
  is an equivalence. The last map is an equivalence because $\LO X$ is
  a local object and $\yO Y \to \LO Y$ is a local equivalence.
  Hence, we have
  \[\Map_{\Fun(\mathcal{O}, \mathcal{S})}(\yO Y,\yO X) \isoto
    \Map_{\Fun(\mathcal{O}, \mathcal{S})}(\yO Y, \LO X)\] for every
  object $Y \in \mathcal{O}$, which then implies that
  $\yO X \simeq \LO X$ in $\Fun(\mathcal{O}, \mathcal{S})$ by the
  Yoneda Lemma.
\end{proof}

\begin{lemma}\label{lem:WTsat}
  Suppose $T$ is a polynomial monad. Then the pattern $\mathcal{W}(T)$
  is saturated.
\end{lemma}
\begin{proof}
  We already know the pattern $\mathcal{W}(T)$ is extendable (by
  Corollary~\ref{cor WTextendable}) and slim (by Remark~\ref{rmk
    WTslim}).
  By definition, $\mathcal{W}(T)^{\op}$ is a full subcategory of
  $\Alg_{T}(\mathcal{S}^{\mathcal{I}})$, and the Nerve
  Theorem~\ref{thm:nervecartsq} implies that the restricted Yoneda
  functor $\Alg_{T}(\mathcal{S}^{\mathcal{I}}) \to
  \Fun(\mathcal{W}(T), \mathcal{S})$ is fully faithful with image
  $\Seg_{\mathcal{W}(T)}(\mathcal{S})$. This implies in particular
  that the Yoneda embedding of $\mathcal{W}(T)$ takes values in Segal
  $\mathcal{W}(T)$-spaces, which implies that $\mathcal{W}(T)$ is
  saturated by Proposition~\ref{propn:Osatcond}.
\end{proof}

Lemma~\ref{lem:WTsat} implies in particular that the pattern $\olO$ is
always saturated, which gives the following:
\begin{cor}\label{cor:sigmasateq}
  The morphism $\sigma_{\mathcal{O}} \colon \mathcal{O} \to \olO$ is
  an equivalence \IFF{} $\mathcal{O}$ is saturated. \qed
\end{cor}

\begin{cor}
  The natural transformation $\sigma$ exhibits the full subcategory $\AlgPattSS$ as a localization of
  $\AlgPattSES$.
\end{cor}
\begin{proof}
  Let $L := \mathfrak{P}\mathfrak{M}$; then the essential image of $L$
  is precisely $\AlgPattSS$: by Corollary~\ref{cor:sigmasateq} the
  image of $L$ contains all saturated patterns, while all patterns in
  the image of $L$ are saturated by Lemma~\ref{lem:WTsat}. To see that
  $L$ and $\sigma$ exhibit $\AlgPattSS$ as a localization, we apply
  \cite[Proposition 5.2.7.4]{ht}. It suffices to verify condition (3) of this result, namely that the two
  morphisms
  \[ \sigma_{L\mathcal{O}}, L(\sigma_{\mathcal{O}}) \colon L
    \mathcal{O} \to LL\mathcal{O}\]
   are both equivalences for all $\mathcal{O}$ in $\AlgPattSES$. For
   $\sigma_{L\mathcal{O}}$ this holds by
   Corollary~\ref{cor:sigmasateq}, since $L \mathcal{O}$ is saturated,
   and for $L(\sigma_{\mathcal{O}})$ it holds since
   $\sigma_{\mathcal{O}}$ induces an equivalence
   \[ \sigma_{\mathcal{O}}^{*} \colon \Seg_{L
       \mathcal{O}}(\mathcal{S}) \isoto
     \Seg_{\mathcal{O}}(\mathcal{S}),\]
   and $L \sigma_{\mathcal{O}}$ is obtained by restricting the inverse
   of this equivalence.
\end{proof}

The following proposition shows that we can equivalently characterize
saturated patterns in terms of their subcategories of inert morphisms:
\begin{propn}\label{propn:Ointsatcond}
  The following conditions are equivalent for a slim extendable pattern $\mathcal{O}$:
  \begin{enumerate}[(1)]
  \item $\mathcal{O}$ is saturated.
  \item For every $X \in \mathcal{O}$, the functor
    \[\Map_{\mathcal{O}^{\xint}}(X,\blank) \colon \mathcal{O}^{\xint}
      \to \mathcal{S}\]
    is a Segal $\mathcal{O}^{\xint}$-space.
  \item For every $X$ in $\mathcal{O}$, the diagram
    $\mathcal{O}^{\el,\triangleleft}_{X/} \to \mathcal{O}^{\xint}$ is
    a limit diagram.
  \item The Yoneda embedding $\mathcal{O}^{\xint,\op} \to
    \Fun(\mathcal{O}^{\xint}, \mathcal{S})$ factors through
    $\Seg_{\mathcal{O}^{\xint}}(\mathcal{S})$.
  \item The functor $\Lambda_{\mathcal{O}}^{\xint} \colon
    \mathcal{O}^{\xint,\op} \to
    \Seg_{\mathcal{O}^{\xint}}(\mathcal{S})$ is fully faithful.
  \end{enumerate}
\end{propn}
\begin{proof}
  The equivalence of conditions (2)--(5) follows exactly as in the
  proof of Proposition~\ref{propn:Osatcond}. It remains to show that
  these conditions are equivalent to $\mathcal{O}$ being saturated.
  
  Since $\mathcal{O}$ is by assumption extendable, by
  Proposition~\ref{prop:extendable} we have a commutative square
  \[
    \begin{tikzcd}
      \Seg_{\mathcal{O}}(\mathcal{S}) \arrow{r} & \Fun(\mathcal{O},
      \mathcal{S}) \\
      \Seg_{\mathcal{O}^{\xint}}(\mathcal{S}) \arrow{u}{F_{\mathcal{O}}}
      \arrow{r} & \Fun(\mathcal{O}^{\xint}, \mathcal{S}). \arrow[u, swap, "j_{\mathcal{O},!}"]
    \end{tikzcd}
  \]
  Omitting notation for the horizontal inclusions, we have equivalences
  \[\Lambda_{\mathcal{O}}X \simeq
    F_{\mathcal{O}}\Lambda_{\mathcal{O}}^{\xint}X \simeq
    j_{\mathcal{O},!}\Lambda_{\mathcal{O}}^{\xint}X.\]
  If condition (4) holds, then $\Lambda_{\mathcal{O}}^{\xint}X$
  is the representable presheaf $y_{\mathcal{O}}^{\xint}X$, hence
  \[j_{\mathcal{O},!}\Lambda_{\mathcal{O}}^{\xint}X \simeq
    j_{\mathcal{O},!}y_{\mathcal{O}}^{\xint}X \simeq
    y_{\mathcal{O}}j_{\mathcal{O}}(X).\]
  In other words, $\Lambda_{\mathcal{O}}X$ is precisely the presheaf
  represented by $X \in \mathcal{O}$, which implies that $\mathcal{O}$
  is saturated by Proposition~\ref{propn:Osatcond}.

  Conversely, suppose $\mathcal{O}$ is saturated. By
  Proposition~\ref{propn:Osatcond} this means that for every $X \in
  \mathcal{O}$, the diagram $\mathcal{O}^{\el,\triangleleft}_{X/} \to
  \mathcal{O}$ is a limit diagram. To show that this diagram is then
  also a limit in the subcategory $\mathcal{O}^{\xint}$ (and hence
  verify condition (3)), it is enough to show that a morphism $\phi
  \colon Y \to X$
  is inert if the composites $Y \to X \intto E$ are all
  inert. Using the inert--active factorization system, we see that it
  suffices to consider the case where $\phi$ is active and prove that it
  is an equivalence. Recall that we have a morphism
  \[ \Act_{\mathcal{O}}(X) \to \lim_{E\in \mathcal{O}^{\el}_{X/}}
    \Act_{\mathcal{O}}(E),\]
  which takes $\phi \colon Y \actto X$ to the active parts of the
  inert--active decompositions of the composites $Y \actto X \intto
  E$. Since these composites are inert, the image of $\phi$ is given
  by $\id_{E}$ for all $E \in \mathcal{O}^{\el}_{X/}$, so that $\phi$
  has the same image as $\id_{X}$. But since $\mathcal{O}$ is
  extendable, this map of $\infty$-groupoids is an equivalence, and
  hence $\phi$ is equivalent to $\id_{X}$ in $\Act_{\mathcal{O}}(X)$,
  which means precisely that $\phi$ is an equivalence.
\end{proof}

We end this section by looking at some examples of saturated and non-saturated patterns.
\begin{examples}
  The patterns $\simp^{n,\op,\natural}$,
  $\bbTheta_{n}^{\op,\natural}$, and $\bbOmega^{\op,\natural}$
  (described in Examples~\ref{ex:simp}, \ref{ex:Theta}, and \ref{ex
    Omega}, respectively)
  are all saturated. In the case of
  $\simp^{\op,\natural}$, for example, this amounts to the observation
  that the object $[n] \in \simp^{\xint}$ is a colimit,
  \[ [1] \amalg_{[0]} \cdots \amalg_{[0]} [1] \simeq [n],\] while for
  $\bbOmega^{\op,\natural}$ the required colimit in
  $\bbOmega^{\xint}$ amounts to a decomposition of a tree as a
  colimit of its nodes and edges, and follows from \cite[Proposition
  1.1.19]{Kock}.
\end{examples}

\begin{example}\label{ex:commspan}
  The pattern $\xF_{*}^{\flat}$ from Example~\ref{ex:xF*flat} is
  \emph{not} saturated: The functor
  $\Lambda^{\xint}_{\xF_{*}^{\flat}} \colon \xF_{*}^{\flat,\xint,\op}
  \to \mathcal{S}$ takes $\angled{n}$ to a finite set $\mathbf{n}$
  with $n$ elements, and an inert morphism $\angled{n} \to \angled{m}$
  to the map $\mathbf{m} \to \mathbf{n}$ that takes $i \in \mathbf{m}$
  to its unique preimage $\phi^{-1}(i) \in \mathbf{n}$. Thus inert
  morphisms correspond bijectively to \emph{injective} morphisms of
  finite sets, and the functor is not fully faithful. The canonical
  pattern
  $\overline{\xF}_{*}^{\flat} \subseteq
  \Seg_{\xF_{*}^{\flat}}(\mathcal{S})^{\op}$ consists of the free
  commutative monoids on finite sets. By work of Cranch~\cite{Cranch}
  this can be identified with the (2,1)-category $\name{Span}(\xF)$
  whose objects are finite sets and whose morphisms are \emph{spans}
  of finite sets, with $\xF_{*} \to \overline{\xF}_{*}$ identifying
  $\xF_{*}$ with the subcategory where the morphisms from $I$ to $J$
  are spans $I \from K \to J$ with the backward map \emph{injective}.
\end{example}

\begin{example}
  More generally, for any \iopd{} $\mathcal{O}$ (in the sense of
  \cite{ha}) the canonical pattern $\olO$ can be
  identified with the opposite of the \icat{} of finitely generated
  free $\mathcal{O}$-monoids in $\mathcal{S}$, \ie{} the \emph{Lawvere
    theory} for $\mathcal{O}$-monoids.
\end{example}

\begin{remark}
  See \cite{GGN,Berman} for more
  on Lawvere theories in the \icatl{} context. Note that the monads
  corresponding to Lawvere theories always preserve sifted colimits,
  so the (coloured) Lawvere theories that fit into our theory
  are precisely the monads on $\mathcal{S}^{X}$ for an \igpd{} $X$
  that preserve both sifted colimits and weakly contractible
  limits. These are precisely the \emph{analytic} monads studied in
  \cite{AnalMnd}, where they are identified with $\infty$-operads in
  the sense of dendroidal Segal spaces.
\end{remark}

\begin{example}
  The pattern $\bbGamma^{\op,\natural}$ of Example~\ref{ex Gamma} is
  not saturated. We expect that its saturation is the $(2,1)$-category
  of graphs implicitly defined by Kock in \cite[\S
  3.3]{Kock_Properads}.
\end{example}

\section{Completion of  Polynomial Monads}\label{sec:poly3}
In this section we will study a class of polynomial monads that is
particularly closely related to algebraic patterns, namely the
\emph{complete} ones in the following sense:
\begin{defn}
  Let $T$ be a polynomial monad on $\mathcal{S}^{\mathcal{I}}$. We say
  that $T$ is \emph{complete} if the functor $\mathcal{I} \to
  \mathcal{W}(T)^{\el}$ underlying $\tau_{T}\colon T \to
  \mathfrak{M}\mathfrak{P}T$ is an equivalence. We write $\cPolyMnd$
  for the full subcategory of $\PolyMnd$ spanned by the complete
  polynomial monads.
\end{defn}
We will see that the polynomial monad corresponding to an extendable
pattern is \emph{always} complete, so that the functor $\mathfrak{M}$
takes values in $\cPolyMnd$. Moreover, we will show that the
transformation $\tau \colon \id \to \mathfrak{M}\mathfrak{P}$ exhibits
$\cPolyMnd$ as a localization of $\PolyMnd$, and the functors
$\mathfrak{M}$ and $\mathfrak{P}$ restrict to an \emph{equivalence}
\[ \cPolyMnd \simeq \AlgPattSS \]
between complete polynomial monads and saturated patterns.

\begin{remark}
  The term \emph{complete} is inspired by the equivalence of
  \cite{AnalMnd} between dendroidal Segal spaces and \emph{analytic} monads,
  which are the polynomial monads on presheaves over
  $\infty$-groupoids that preserve sifted colimits. Under this
  equivalence, the complete dendroidal Segal spaces (meaning those
  whose underlying Segal spaces are complete in the sense of
  Rezk~\cite{Rezk}) are precisely those analytic monads that are
  complete in our sense.
\end{remark}

We begin by giving some alternative descriptions of the complete
polynomial monads:
\begin{proposition}\label{prop complete}
  Let $T$ be a polynomial monad on $\mathcal{S}^{\mathcal{I}}$. The following are equivalent:
  \begin{enumerate}
    \item $T$ is complete.
    \item The morphism $\tau_{T} \colon T \to \mathfrak{M}\mathfrak{P}T$
      is an equivalence.
    \item The functor $u \colon \mathcal{U}(T) \to \mathcal{W}(T)^{\xint}$ is
      an equivalence.
    \item The functor $j \colon \mathcal{U}(T) \to \mathcal{W}(T)$ is
      a subcategory inclusion, \ie{} it is faithful and induces an
      equivalence $\xU(T)^\simeq \isoto \xW(T)^\simeq$ on underlying
      $\infty$-groupoids.
    \item The functor $j$ is faithful and every equivalence is in its
      image.
  \end{enumerate}
\end{proposition}
\begin{proof}
  To see that (1) is equivalent to (2), observe that the morphism $\tau_{T}$ in
  $\PolyMnd$ is given by the morphism $e \colon \mathcal{I} \to
  \mathcal{W}(T)^{\el}$ together with the commutative square
  \[
    \begin{tikzcd}
      \Seg_{\mathcal{W}(T)}(\mathcal{S}) \arrow{r}{\sim} \arrow{d} &
      \Alg_{T}(\mathcal{S}^{\mathcal{I}}) \arrow{d} \\
      \Fun(\mathcal{W}(T)^{\el}, \mathcal{S}) \arrow{r}{e^{*}} &
      \Fun(\mathcal{I}, \mathcal{S}), 
    \end{tikzcd}
  \]
  and so $\tau_{T}$ is an equivalence \IFF{} $e$ is an equivalence.

  It is clear that (3) implies (1), since $e$ is obtained from $u$ by
  restricting to a full subcategory. Conversely, if $e$ is an
  equivalence, then the commutative square of
  Corollary~\ref{cor:elintcartsq} gives a commutative square
  \[
    \begin{tikzcd}
      \Fun(\mathcal{W}(T)^{\el}, \mathcal{S}) \arrow[d, swap, "\wr"] \arrow{r}{\sim}
      & \Seg_{\mathcal{W}(T)^{\xint}}(\mathcal{S}) \arrow{d}{u^{*}} \\
      \Fun(\mathcal{I}, \mathcal{S}) \arrow{r}{\sim} &
     \Seg_{\mathcal{U}(T)}(\mathcal{S}), 
    \end{tikzcd}
  \]
  where $\Seg_{\mathcal{U}(T)}(\mathcal{S})$ denotes the full
  subcategory of $\Fun(\mathcal{U}(T), \mathcal{S})$ of functors
  right Kan extended from $\mathcal{I}$; the functor
  \[ u^{*} \colon \Seg_{\mathcal{W}(T)^{\xint}}(\mathcal{S}) \to
    \Seg_{\mathcal{U}(T)}(\mathcal{S}) \] is therefore an
  equivalence. Here $\mathcal{W}(T)^{\xint,\op}$ is a full subcategory
  of $\Seg_{\mathcal{W}(T)^{\xint}}(\mathcal{S})$ via the Yoneda
  embedding by Proposition~\ref{propn:Ointsatcond}, since
  $\mathcal{W}(T)$ is saturated by Lemma~\ref{lem:WTsat}. Moreover,
  $\mathcal{U}(T)^{\op}$ is a full subcategory of
  $\Seg_{\mathcal{U}(T)}(\mathcal{S})$ by
  Proposition~\ref{propn:WTintnerve}. The inverse of $u^{*}$ is given
  by left Kan extension $u_{!}$ followed by localization from
  $\Fun(\mathcal{W}(T)^{\xint}, \mathcal{S})$ to
  $\Seg_{\mathcal{U}(T)}(\mathcal{S})$, which restricts to just $u$ on
  $\mathcal{U}(T)^{\op}$ since $\mathcal{W}(T)^{\xint,\op}$ is already
  in $\Seg_{\mathcal{W}(T)^{\xint}}(\mathcal{S})$. Hence $u$ is the
  restriction of the equivalence $(u^{*})^{-1}$ to a full subcategory,
  which implies that $u$ is indeed an equivalence.

  Since $\mathcal{W}(T)^{\xint}$ is by definition a subcategory of
  $\mathcal{W}(T)$, (3) immediately implies (4). On the other hand,
  (4) implies (3) since the inert morphisms in $\mathcal{W}(T)$ are
  precisely those that are composites of morphisms in the image of $u$
  and equivalences in $\mathcal{W}(T)$.

  Finally, (4) trivially implies (5), while given (5) we know that
  \[ \Map_{\mathcal{U}(T)}(X,Y) \to \Map_{\mathcal{W}(T)}(jX,jY) \]
  is a monomorphism whose image contains the components that
  correspond to equivalences in $\mathcal{W}(T)$. Since $j$ is
  conservative by Lemma~\ref{lem Fconservative}, the components that
  map to these are precisely those that correspond to equivalences in
  $\mathcal{U}(T)$, so that $j$ restricts to an equivalence
  $\mathcal{U}(T)^{\simeq} \to \mathcal{W}(T)^{\simeq}$. 
\end{proof}

\begin{propn}\label{propn:TOcomplete}
  Suppose $\mathcal{O}$ is a slim extendable pattern. Then
  $T_{\mathcal{O}}$ is a complete polynomial monad.
\end{propn}

For the proof we need the following observation:
\begin{lemma}\label{lem:LOacteq}
  Suppose $\phi \colon X \to Y$ is an active morphism such that $\LO
  \phi$ is an equivalence in $\Seg_{\mathcal{O}}(\mathcal{S})$. Then
  $\phi$ is an equivalence in $\mathcal{O}$.
\end{lemma}
\begin{proof}
  Suppose $\alpha \colon \LO X \to \LO Y$ is the inverse of
  $\LO \phi$. By Proposition~\ref{propn:Oactgeneric} we can factor
  $\alpha$ as $\LO X \xto{\LO \psi} \LO Y' \xto{\alpha'} \LO Y$ where
  $\alpha'$ is free and $\psi$ is an active morphism determined up to equivalence in
  $\mathcal{O}$ (and both $\LO \psi$ and $\alpha'$ are equivalences since this is an
  active--inert factorization). Now the composite $\alpha \LO \phi$ is
  the identity, so by Proposition~\ref{propn:Oactgeneric} the
  composite $\phi \psi$ lies in the same component of 
  $\Act_{\mathcal{O}}(Y)$ as
  $\id_{Y}$, \ie{} $\phi \psi$ must be an equivalence. Applying the
  same argument to $\psi$, we see that $\psi$ has inverses on both
  sides in $\mathcal{O}$ and so is an equivalence, hence $\phi$ is
  also an equivalence.
\end{proof}

\begin{proof}[Proof of Proposition~\ref{propn:TOcomplete}]
  By Proposition~\ref{prop complete} the polynomial monad $T_\xxO$ is complete if and only if $j \colon \mathcal{U}(T_{\mathcal{O}}) \to
  \mathcal{W}(T_{\mathcal{O}})$ is faithful and all equivalences are
  in its image.

  Since $\mathcal{O}$ is slim, the objects of
  $\mathcal{U}(T_{\mathcal{O}})$ are precisely the objects $\LOi X$
  for $X \in \mathcal{O}^{\xint}$, by
  Corollary~\ref{cor:overlineOdesc}. To show that $j$ is faithful, we
  must check that for all $X, Y \in \mathcal{O}^{\xint}$, the map
  \[ \Map_{\Seg_{\mathcal{O}^{\xint}}(\mathcal{S})}(\LOi X, \LOi Y) \to
    \Map_{\Seg_{\mathcal{O}}(\mathcal{S})}(\LO X, \LO Y) \]
  is a monomorphism. Lemma~\ref{lem:Lambdadesc} and Remark~\ref{rem
    MapSegOint} imply that this map can be identified with the map
\[ (\LOi X)(Y) \to \colim_{O\in
    \Act_{\mathcal{O}}(Y)} (\LOi X)(O), \]
 given by
  taking $(\LOi X)(Y)$ to the component in the colimit corresponding
  to $\id_{Y} \in \Act_{\mathcal{O}}(Y)$. This component is of the
  form $(\mathcal{O}^{\simeq})_{/Y}$ and so is contractible, which
  means that the colimit decomposes as a disjoint union of $(\LOi
  X)(Y)$ and the colimit over the other components of
  $\Act_{\mathcal{O}}(Y)$. This means $j$ is indeed faithful.

  Now suppose $\alpha \colon \LO X \to \LO X'$ is an equivalence in
  $\mathcal{W}(T_{\mathcal{O}})$. Then by
  Proposition~\ref{propn:Oactgeneric} we can factor $\alpha$ as
  \[ \LO X \xto{\LO \phi} \LO Y \xto{j \psi} \LO X', \]
  where $\phi$ is active and both $\LO \phi$ and $j \psi$ are equivalences (since this is
  in particular an active--inert factorization). Then
  Lemma~\ref{lem:LOacteq} implies that $\phi$ is an equivalence in
  $\mathcal{O}$; but then $\phi$ is also inert, and so the commutative
  square in Lemma~\ref{lem:LambdaisF} implies that $\LO \phi$ is
  $j(\LOi \phi)$. Thus $\alpha$ is in the image of $j$, as required.
\end{proof}

\begin{remark}
  It follows from Proposition~\ref{prop complete} that for any slim
  extendable pattern $\mathcal{O}$, the morphism
  $\tau_{T_{\mathcal{O}}} \colon T_{\mathcal{O}} \to T_{\olO}$ is an
  equivalence, \ie{} the extendable patterns $\mathcal{O}$ and $\olO$
  correspond to the same monad. The saturated pattern $\olO$ is thus a
  canonical pattern associated to the free Segal $\mathcal{O}$-space
  monad $T_{\mathcal{O}}$.
\end{remark}

\begin{cor}
  The natural transformation
  $\tau \colon \id \to \mathfrak{M}\mathfrak{P}$ exhibits the full
  subcategory $\cPolyMnd$ as a localization of $\PolyMnd$.
\end{cor}
\begin{proof}
  Let $L := \mathfrak{M}\mathfrak{P}$; then the essential image of $L$
  is precisely $\cPolyMnd$: by Proposition~\ref{prop complete} the
  image of $L$ contains all complete polynomial monads, while all
  monads in the image of $L$ are complete by
  Proposition~\ref{propn:TOcomplete}.

  To see that $L$ and $\tau$ exhibit $\cPolyMnd$ as a localization,
  we again apply the criterion of \cite[Proposition
  5.2.7.4]{ht}(3). We must thus show that the two morphisms
  \[ \tau_{LT}, L(\tau_{T}) \colon LT \to LLT \]
   are both equivalences for all $T$ in $\PolyMnd$. For
   $\tau_{LT}$ this holds by Proposition~\ref{prop complete}, since
   $LT$ is complete, while for $L(\tau_{T})$ it holds since
   $\mathfrak{P}(\tau_{T})$ is an equivalence (given by restricting
   the equivalence $\Alg_{T}(\mathcal{S}^{\mathcal{I}}) \isoto
   \Seg_{\mathcal{W}(T)}(\mathcal{S})$ to a full subcategory).
\end{proof}

\begin{thm}
  The functors $\mathfrak{M}$ and $\mathfrak{P}$ restrict to give an
  equivalence
  \[ \cPolyMnd \simeq \AlgPattSS. \]
\end{thm}
\begin{proof}
  We have shown that $\mathfrak{M}\mathcal{O}$ is always complete and
  $\mathfrak{P}T$ is always saturated, so the functors do restrict to
  these full subcategories. Moreover, we know that
  $\sigma_{\mathcal{O}}$ is an equivalence \IFF{} $\mathcal{O}$ is
  saturated, and $\tau_{T}$ is an equivalence \IFF{} $T$ is
  complete. These natural transformations therefore restrict to
  natural equivalences on the full subcategories of saturated patterns
  and complete polynomial monads, and hence exhibit the restrictions
  of $\mathfrak{P}$ and $\mathfrak{M}$ as inverse equivalences.
\end{proof}

\bibliographystyle{hamsalpha}
\providecommand{\bysame}{\leavevmode\hbox to3em{\hrulefill}\thinspace}
\providecommand{\MR}{\relax\ifhmode\unskip\space\fi MR }
\providecommand{\MRhref}[2]{%
  \href{http://www.ams.org/mathscinet-getitem?mr=#1}{#2}
}
\providecommand{\href}[2]{#2}

\enlargethispage{1cm}

\end{document}